\documentclass[11pt,a4paper]{amsart}
\usepackage[alphabetic]{amsrefs}
\usepackage{amsmath,amssymb,booktabs,hyperref}

 \theoremstyle{plain}
 \newtheorem{thm}{Theorem}[section]
 \newtheorem{lem}[thm]{Lemma}
 \newtheorem{prop}[thm]{Proposition}
 \newtheorem{cor}[thm]{Corollary}
 \theoremstyle{definition}
 \newtheorem{rem}[thm]{Remark}
 \newtheorem{defn}[thm]{Definition}
 \newtheorem{example}[thm]{Example}

\newcommand{\divides}{\mid}
\newcommand{\notdivides}{\nmid} %\nmid?
\newcommand{\map}[4][\to]{#2 \colon #3 #1 #4}
\DeclareMathOperator{\charac}{char}
\newcommand{\sm}{\mathrm{sm}}
\newcommand{\Het}{H_{\mathrm{\acute et}}}

\newcommand{\Hcrys}{H_{\mathrm{crys}}}
\newcommand{\Tcrys}{T_{\mathrm{crys}}}
\DeclareMathOperator{\ord}{ord}
\newcommand{\injto}{\hookrightarrow}
\newcommand{\actson}{\curvearrowright}
\newcommand{\isomto}{\stackrel{\sim}{\to}}
\newcommand{\isomfrom}{\stackrel{\sim}{\leftarrow}}
\newcommand{\idealm}{\mathfrak{m}}
\DeclareMathOperator{\Spec}{Spec}
\newcommand{\bZ}{\mathbb Z}
\newcommand{\bQ}{\mathbb Q}
\newcommand{\bF}{\mathbb F}
\newcommand{\bP}{\mathbb P}
\newcommand{\bG}{\mathbb G}
\DeclareMathOperator{\SL}{SL}
\DeclareMathOperator{\Bl}{Bl}
\newcommand{\set}[1]{\{#1\}}
\newcommand{\dual}[1]{#1^{\vee}}
\DeclareMathOperator{\Fix}{Fix}
\DeclareMathOperator{\Aut}{Aut}
\DeclareMathOperator{\Stab}{Stab}
\newcommand{\isolatedfix}[1]{\langle #1 \rangle}
\newcommand{\divisorialfix}[1]{(#1)}
\DeclareMathOperator{\Sing}{Sing}
\DeclareMathOperator{\Hom}{Hom}
\DeclareMathOperator{\sHom}{\mathcal{H}\mathit{om}}
\DeclareMathOperator{\pr}{pr}
\DeclareMathOperator{\wt}{wt}
\DeclareMathOperator{\divisor}{div}
\DeclareMathOperator{\Pic}{Pic}
\DeclareMathOperator{\disc}{disc}

\DeclareMathOperator{\Der}{Der}
\newcommand{\restrictedto}[1]{\rvert_{#1}}
\newcommand{\partialdd}[1]{\frac{\partial}{\partial #1}}
\newcommand{\red}{\mathrm{red}}

\newcommand{\thpower}[2]{{#1}^{(#2)}}
\newcommand{\pthpower}[1]{\thpower{#1}{p}}
\newcommand{\card}[1]{\lvert #1 \rvert}
\DeclareMathOperator{\rank}{rank}

\newcommand{\cO}{\mathcal{O}}
\newcommand{\cI}{\mathcal I}
\newcommand{\floor}[1]{\lfloor #1 \rfloor}
\newcommand{\positive}[1]{(#1)^{+}}
\newcommand{\abs}[1]{\lvert #1 \rvert}
\DeclareMathOperator{\Image}{Im}
\newcommand{\id}{\mathrm{id}}

\newcommand{\notfixed}{[\mathsf{n}]}
\newcommand{\another}{[*]}
\newcommand{\Scyc}{S_{\mathrm{cyc}}}
\newcommand{\Smu}{S_{\mu}}

\title
{On $\mu_{n}$-actions on K3 surfaces in positive characteristic}

\author{Yuya Matsumoto}
\address{Department of Mathematics, Faculty of Science and Technology, Tokyo University of Science, 2641 Yamazaki, Noda, Chiba, 278-8510, Japan}
\email{\url{matsumoto.yuya.m@gmail.com}}
\email{\url{matsumoto_yuya@ma.noda.tus.ac.jp}}

\date{2022/02/20}
\thanks{This work was supported by JSPS KAKENHI Grant Numbers 15H05738 and 16K17560.}

\begin{document}

\begin{abstract}
In characteristic $0$,
symplectic automorphisms of K3 surfaces (i.e.\ automorphisms preserving the global $2$-form) and non-symplectic ones behave differently.
In this paper we consider the actions of the group schemes $\mu_{n}$
on K3 surfaces (possibly with rational double point singularities) in characteristic $p$, 
where $n$ may be divisible by $p$.
We introduce the notion of symplecticness of such actions,
and we show that symplectic $\mu_{n}$-actions 
have similar properties,
such as possible orders, fixed loci, and quotients,
to symplectic automorphisms of order $n$ in characteristic $0$.
We also study local $\mu_n$-actions on rational double points.
\end{abstract}

\maketitle

\section{Introduction}
K3 surfaces are proper smooth surfaces $X$ with $\Omega^2_X \cong \cO_X$ and $H^1(X, \cO_X) = 0$.
The first condition implies that $X$ admits an everywhere non-vanishing $2$-form, and such a $2$-form is unique up to scalar.
An automorphism of a K3 surface is called \emph{symplectic} if it preserves the global $2$-form.
It is known that symplectic and non-symplectic automorphisms behave very differently.

For example, 
Nikulin \cite{Nikulin:auto}*{Sections 4--5} proved that quotients of K3 surfaces in characteristic $0$ by a symplectic action of a finite group $G$  
has only rational double points (RDPs for short) as singularities and that the minimal resolutions of the quotients are again K3 surfaces.
Moreover he determined the number of fixed points (which are always isolated) if $G$ is cyclic.
To the contrary, the quotients by non-symplectic actions of finite groups are never birational to K3 surfaces;
instead they are birational to either Enriques surfaces or rational surfaces.

These results hold in characteristic $p > 0$ provided $p$ does not divide the order of $G$ (see Theorem \ref{thm:quotient:tame symplectic}),
but are no longer true for order $p$ automorphisms in characteristic $p$.
In this case the notion of symplecticness is useless, 
since any such automorphism is automatically symplectic (since there are no nontrivial $p$-th root of unity in characteristic $p$) 
and, for small $p$, there exist examples of automorphisms with non-K3 quotients 
(see \cite{Dolgachev--Keum:wild-p-cyclic},\cite{Dolgachev--Keum:auto}).

In this paper we consider actions of the finite group schemes $\mu_{n}$ ($n$ may be divisible by $p$) on 
RDP K3 surfaces, by which we mean surfaces with RDP singularities whose minimal resolutions are K3 surfaces,
It is essential to allow RDPs since smooth K3 surfaces never admit actions of $\mu_p$
(see Remark \ref{rem:no global derivation}).
We introduce the notion of symplecticness and fixed points
of such actions (Definitions \ref{def:symplectic} and \ref{def:fixedpoints}).
Then we prove the following properties, which are parallel to the properties of automorphisms of order not divisible by the characteristic.

\begin{thm}[Theorems \ref{thm:quotient:symplectic} and \ref{thm:quotient:non-symplectic}]
Let $X$ be an RDP K3 surface in characteristic $p$, equipped with a $\mu_{n}$-action.
If the action is symplectic, then the quotient $X/\mu_{n}$ is an RDP K3 surface. 
If $n = p$ and the action is non-symplectic, 
then the quotient $X/\mu_{p}$ is an RDP Enriques surface if the action is fixed-point-free (which is possible only if $p = 2$),
and otherwise it is a rational surface.
\end{thm}

\begin{thm}[Theorems \ref{thm:symplectic mu_n} and \ref{thm:mu_n}]
\begin{itemize}
	\item There exists an RDP K3 surface $X$ in characteristic $p$ equipped with a $\mu_{p}$-action 
	if and only if $p \leq 19$.
	\item If $X$ is an RDP K3 surface $X$ in characteristic $p$ equipped with a $\mu_{n}$-action, then $\phi(n) \leq 20$, in particular $n \leq 66$.
	Moreover, for each $p$ we determine the set of $n$ for which such an action exists.
	\item For each $p$, there exists an RDP K3 surface $X$ in characteristic $p$ equipped with a symplectic $\mu_{n}$-action 
	if and only if $n \leq 8$,
	and we determine the number of fixed points.
%	and then there are exactly 
%	$(24/n) \prod_{l:\text{prime},l \divides n} (l/(l+1))$ 
%	fixed points (counted with suitable multiplicities).
\end{itemize}
\end{thm}

To prove the main results we first study (in Sections \ref{sec:tame} and \ref{sec:mu_p RDP}) $\mu_n$-actions on local rings of surfaces at smooth points and RDPs.
We define the notion of symplecticness of such actions (Definitions \ref{def:symplectic local tame} and \ref{def:symplectic local})
and prove the following result.
\begin{thm}[Theorem \ref{thm:mu_p RDP} and Propositions \ref{prop:non-fixed mu_p RDP}, \ref{prop:list of fixed symplectic RDP}]
Let $X$ be the localization  
at a closed point $w$ of an RDP surface in characteristic $p$ 
equipped with a $\mu_{p}$-action.
Let $\map{\pi}{X}{X/\mu_p}$ be the quotient morphism.
\begin{itemize}
\item If $w$ is not fixed by the action, 
then $\pi(w)$ is either a smooth point or an RDP.
\item If $w$ is fixed and the action is symplectic at $w$,
then $w$ is an isolated fixed point and $\pi(w)$ is an RDP.
\item If $w$ is an isolated fixed point and the action is non-symplectic at $w$,
then $\pi(w)$ is a non-RDP singularity.
\end{itemize}
We classify the possible actions in the non-fixed case (Table \ref{table:non-fixed RDP}) and the symplectic case (Table \ref{table:fixed symplectic RDP}).
\end{thm}
Moreover we also give a partial classification of local $\mu_{p^e}$- and $\mu_n$-actions (Propositions \ref{prop:non-fixed mu_n RDP} and \ref{prop:mu_p^e RDP})
and a complete classification of local symplectic $\mu_n$-actions (Proposition \ref{prop:symplectic mu_n RDP}).
We hope that these local results would have applications other than K3 surfaces.

The results on $\mu_n$-quotients, orders of symplectic $\mu_n$-actions, and orders of $\mu_n$-actions on K3 surfaces 
are discussed in Sections \ref{sec:quotient}, \ref{sec:symplectic order}, and \ref{sec:order} respectively.

In Section \ref{sec:examples} we give several examples of $\mu_{n}$-actions on K3 surfaces.

\medskip

Throughout the paper we work over an algebraically closed field $k$ of $\charac k = p \geq 0$.
Varieties are separated integral $k$-schemes of finite type (not necessarily proper or smooth),
and surfaces are $2$-dimensional varieties.
We denote the smooth locus of a variety $X$ by $X^\sm$.

\section{Preliminaries}

\subsection{K3 surfaces and rational double points} \label{subsec:RDP}

Rational double point singularities (RDPs) of surfaces 
are precisely the canonical surface singularities that are not smooth.
They are classified into types $A_n$ ($n \geq 1$), $D_n$ ($n \geq 4$), $E_n$ ($n = 6,7,8$) 
by their dual graph of the exceptional curves of the minimal resolution, which are Dynkin diagrams of type $A_n$, $D_n$, or $E_n$.
The number $n$ is equal to the number of the exceptional curves, and 
we say that the RDP is of \emph{index} $n$.
The dual graph determines the formal isomorphism class of an RDP 
except in certain cases in characteristic $2,3,5$.
For the exceptional cases we use Artin's notation $D_n^r$ and $E_n^r$ (see \cite{Artin:RDP}).

We recall the classification, given by 
Bombieri--Mumford \cite{Bombieri--Mumford:II}, of proper smooth surfaces $X$ with numerically trivial canonical divisor $K_X$: 
they consist of four classes, with the characterizing properties as shown in Table \ref{table:kappa=0 surfaces}.
Here $b_i = \dim \Het^i(X, \bQ_l)$ is the $i$-th $l$-adic Betti number for an auxiliary prime $l \neq \charac k$.
Enriques and (quasi-)hyperelliptic surfaces in characteristic $2$ and $3$ may have unusual values of $\dim H^1(\cO_X)$ and $\ord(K_X)$.

\begin{table} 
\caption{Surfaces with numerically trivial canonical divisors} \label{table:kappa=0 surfaces}
\begin{tabular}{llllll}
\toprule
& $\dim H^1(\cO_X)$ & $b_1$ & $b_2$ & $\ord(K_X)$ & $\charac$ \\
\midrule
abelian                & $2$ & $4$ & $6$  & $1$       & any \\
K3                     & $0$ & $0$ & $22$ & $1$       & any \\
Enriques               & $0$ & $0$ & $10$ & $2$       & any \\
Enriques               & $1$ & $0$ & $10$ & $1$       & $2$ \\
(quasi-)hyperelliptic  & $1$ & $2$ & $2$  & $2,3,4,6$ & any \\
(quasi-)hyperelliptic  & $2$ & $2$ & $2$  & $1$       & $2,3$ \\
\bottomrule
\end{tabular}
\end{table}

The distinction between hyperelliptic and quasi-hyperelliptic surfaces is not important in this paper.
Also the choice of the origin of an abelian surface is not important.

\begin{defn}
\emph{RDP surfaces}
are surfaces that have only RDPs as singularities (if any).
Hence, any smooth surface is an RDP surface by definition.

\emph{RDP K3 surfaces}
are proper RDP surfaces whose minimal resolutions are (smooth) K3 surfaces. 
We similarly define RDP abelian, RDP Enriques, and RDP (quasi-)hyperelliptic surfaces.

However, since abelian surfaces and (quasi-)hyperelliptic surfaces do not admit smooth rational curves,
any RDP abelian or RDP (quasi-)hyperelliptic surface is smooth.
\end{defn}

\begin{rem} \label{rem:no global derivation}
Smooth K3 surfaces in characteristic $p > 0$ admit no nontrivial global vector fields 
(Rudakov--Shafarevich \cite{Rudakov--Shafarevich:inseparable}*{Theorem 7}, Nygaard \cite{Nygaard:vectorK3}*{Corollary 3.5}),
and hence admit no nontrivial actions of $\mu_p$ (or $\alpha_p$).
However RDP K3 surfaces may admit such actions.
\end{rem}

\begin{prop} \label{prop:2-forms of RDP}
For any RDP surface $X$,
the pullback by the morphism $X^{\sm} \cong \tilde X \setminus E \injto \tilde X$ to the minimal resolution $\tilde X$ of $X$
induces an isomorphism $H^0(X^{\sm}, (\Omega^2_X)^{\otimes n}) \cong H^0(\tilde X, (\Omega^2_{\tilde X})^{\otimes n})$, 
where $E$ is the exceptional divisor.
Non-vanishing forms on one side correspond to non-vanishing ones on the other side.
\end{prop}
\begin{proof}
This follows from the following local version applied repeatedly.
\end{proof}

\begin{prop} \label{prop:2-forms of local RDP}
Let $(A,\idealm)$ be the localization at a closed point of an RDP surface.
Then $H^0(\Spec A \setminus \set{\idealm}, \Omega^2_{A/k})$ 
is a free $A$-module of rank one.
If $A$ is smooth then this space is isomorphic to $H^0(\Spec A, \Omega^2_{A/k})$. 
If $A$ is an RDP
and $(A',\idealm')$ is the localization at a closed point of $\Bl_{\idealm} A$
then any generator of the above space
extends to a generator of $H^0(\Spec A' \setminus \set{\idealm'}, \Omega^2_{A'/k})$.
\end{prop}
\begin{proof}
If $A$ is smooth then $\Omega^2_{A/k}$ is free of rank $1$ and the assertion is clear.
Suppose $A$ is an RDP.
Then it is a hypersurface isolated singularity, and it is well-known that for such singularities the canonical divisor is trivial, and then the former assertion follows.
Since an RDP is a canonical singularity, the pullback of the canonical divisor to $\Bl_{\idealm} A$ is also trivial, hence the latter assertion follows.
\end{proof}

\subsection{Group schemes of multiplicative type} \label{subsec:groupscheme}

Recall that we are working over an algebraically closed field $k$. 

We consider finite commutative group schemes $G$ of multiplicative type over $k$.
This means that $G$ is of the form $\prod_j \mu_{n_j}$ for some positive integers $n_j$.
The Cartier dual $\dual{G} = \sHom(G, \bG_m)$ of $G$ is a finite \'etale group scheme and can be identified with the finite group $\dual{G}(k)$ of $k$-valued points.
Using this finite commutative group $\dual{G}$ we have the following explicit description:
$G = \Spec k[t_{i}]_{i \in \dual{G}} / (t_{i} t_{j} - t_{i+j}, t_0 - 1)$,
with the group operations 
$m \colon G \times G \to G$,
$e \colon \Spec k \to G$,
$i \colon G \to G$ given by 
$m^*(t_{i}) = t_{i} \otimes t_{i}$, 
$e^*(t_{i}) = 1$,
$i^*(t_{i}) = t_{-i}$.

An action $\alpha \colon G \times \Spec B \to \Spec B$ corresponds, 
via $\alpha^*(b) = \sum_{i \in \dual{G}} t_{i} \otimes \pr_{i}(b)$,
to decompositions $B = \bigoplus_{i \in \dual{G}} B_i$ to $k$-vector subspaces
satisfying $B_i B_j \subset B_{i+j}$.
We say an element $b$ or a subset of $B_i$ to be \emph{homogeneous of weight} $i$ and we write $\wt(b) = i$.

Such a decomposition $B = \bigoplus_{i} B_i$ naturally extends to a decomposition $\Omega^*_{B/k} = \bigoplus_i (\Omega^*_{B/k})_i$
satisfying $d(B_i) \subset (\Omega^1_{B/k})_i$ and $(\Omega^*_{B/k})_i (\Omega^*_{B/k})_j \subset (\Omega^*_{B/k})_{i+j}$.

If $G$ acts on a scheme $X$ that is not necessarily affine but admits a covering by $G$-stable affine open subschemes 
(which is the case if e.g.\ $X$ is quasi-projective or $G$ is local),
then the $G$-action admits a quotient $\map{\pi}{X}{X/G}$,
and induces decompositions 
$\pi_* \cO_X = \bigoplus_{i} (\pi_* \cO_X)_i$,
$\pi_* \Omega^*_{X/k} = \bigoplus_{i} (\pi_* \Omega^*_{X/k})_i$, and
$H^0(X, (\Omega^*_{X/k})^{\otimes n}) = \bigoplus_i (H^0(X, (\Omega^*_{X/k})^{\otimes n}))_i$,
compatible with multiplications.

If $\charac k$ does not divide the order of $\dual{G}$,
then $B_i$ are the eigenspaces for the action of $G(k)$ with eigenvalues $i \in \dual{G}(k) = \Hom(G(k),k^*)$.

If $\charac k = p > 0$ and $\dual{G}$ is cyclic of order $p$ (hence $G \cong \mu_p = \Spec k[t_1]/(t_1^p - 1)$ for a choice of a generator $1$ of $\dual{G}$) then
giving such a decomposition is also equivalent to giving a $k$-derivation $D$ on $B$ of multiplicative type (i.e.\ $D^p = D$)
under the correspondence $B_i = B^{D = i} = \set{b \in B \mid D(b) = ib }$
(this correspondence depends on the choice of a generator $1$ of $\dual{G}$).
Moreover $D$ extends to a $k$-linear endomorphism on $\Omega^*_{B/k}$ satisfying $D(df) = d(D(f))$, $D^p = D$, 
and the Leibniz rule $D(\omega \wedge \eta) = \omega \wedge D(\eta) + D(\omega) \wedge \eta$.

Now we generalize the notion of symplecticness of automorphisms to actions of group schemes like $\mu_n$.

\begin{defn} \label{def:symplectic}
Let $G$ be a finite group scheme of multiplicative type.
Let $X$ be either an abelian surface or an RDP K3 surface,
equipped with an action of $G$.
We say that the action is \emph{symplectic} 
if the $1$-dimensional space $H^0(X^{\sm}, \Omega^2_{X/k})$ with respect to the action of $G$ is of weight $0$.
\end{defn}

\begin{rem} \label{rem:commutative}
Under the assumptions of Definition \ref{def:symplectic}, suppose $G$ is reduced. 
Equivalently, this means that $G$ is a constant group scheme corresponding to a finite abelian group of order prime to $p$.
Then, by Proposition \ref{prop:2-forms of RDP}, 
our symplecticness is equivalent to the symplecticness of the induced $G$-action on the minimal resolution $\tilde{X}$ in the usual sense (i.e.\ preserving the global $2$-form).
This suggests that our definition of the symplecticness of $\mu_n$-actions is a natural generalization 
of that of $\bZ/m\bZ$-actions (order $m$ automorphisms) for $m$ not divisible by $\charac k$.

On the other hand, if $G = \bZ/p\bZ$ (which does not belong to the class considered in Definition \ref{def:symplectic}), 
then any action of $G$ preserves the global $2$-form, since there are no nontrivial $p$-th roots of unity.
Thus the usual definition of symplecticness is useless in this case.
We do not know whether there is a useful notion of symplecticness 
in a larger class of group schemes containing $\bZ/p\bZ$ or $\alpha_p$.
\end{rem}

\subsection{Derivations of multiplicative type} \label{subsec:derivation}

In this section assume $\charac k = p > 0$.

Recall that, given an action of a group scheme $G$ on a scheme $X$,
the fixed point scheme $X^G \subset X$ is characterized by the property 
$X^G(T) = \Hom_G(T,X)$ for any $k$-scheme $T$ equipped with the trivial $G$-action.
If $G = \mu_p$ and $D$ is the corresponding derivation,
we write $\Fix(D) = X^G$ and also call it the fixed locus of $D$.

\begin{defn} \label{def:fixedpoints}
We say that a closed point $w \in X$ is \emph{fixed} by the $\mu_{n}$-action,
or by the corresponding derivation if $n = p$, 
if $w \in X^{\mu_{n}}$.
\end{defn}

\begin{prop} \label{prop:fixedpoints}
Let $k$ be an algebraically closed field.
Let $X = \Spec B$ be a Noetherian affine $k$-scheme equipped with a $\mu_{p^e}$-action.
For each closed point $w \in X$,
the assertions (\ref{item:fixedpointscheme})--(\ref{item:in m}) are equivalent.
If $e = 1$ and $D$ is the corresponding derivation then 
the assertions (\ref{item:fixedpointscheme})--(\ref{item:DA}) are equivalent,
and if moreover $X$ is a smooth variety then 
also (\ref{item:RS}) is equivalent.
\begin{enumerate}
\item \label{item:fixedpointscheme} $w$ is a $\mu_{p^e}$-fixed point.
\item \label{item:eigenvectors} The maximal ideal $\idealm_w$ of $\cO_{X,w}$ is generated by homogeneous elements.
\item \label{item:A/m} The canonical morphism $B \to B / \idealm_w$ is $\mu_{p^e}$-equivariant,
where $B/\idealm_w$ is equipped with the trivial action (i.e. the decomposition concentrated on $(-)_0$).
\item \label{item:in m} $B_i \subset \idealm_w$ for each $i \neq 0$. 
\item \label{item:Dm} $D(\idealm_w) \subset \idealm_w$.
\item \label{item:DA} $D(\cO_{X,w}) \subset \idealm_w$.
\item \label{item:RS} $D$ has singularity at $w$ in the sense of \cite{Rudakov--Shafarevich:inseparable}*{Section 1}.
\end{enumerate}
Finally, if (\ref{item:fixedpointscheme}) holds then the $\mu_{p^e}$-action extends to the blow-up $\Bl_w X$.
\end{prop}

\begin{proof}
Let $B = \bigoplus_{i \in \bZ/p^e\bZ} B_i$ be the corresponding decomposition.

(\ref{item:fixedpointscheme} $\iff$ \ref{item:A/m})
By the definition of $X^{\mu_p}$, a closed point $w \in X$ is a ($k$-valued) point of $X^{\mu_p}$
if and only if $B \to B/\idealm_w$ is compatible with the projections $\pr_i$ to the $i$-th summand $(-)_i$ for all $i$,
where $B/\idealm_w$ is equipped with the trivial decomposition.

(\ref{item:eigenvectors} $\iff$ \ref{item:A/m})
If (\ref{item:A/m}) holds then we have $\pr_i(\idealm_w) \subset \idealm_w$ for all $i$,
and then each element $x$ of $\idealm_w$ is the sum of homogeneous elements $\pr_i(x) \in \idealm_w$.
Conversely if $\idealm_w$ is generated by homogeneous elements then $\pr_i(\idealm_w) \subset \idealm_w$ for all $i$,
which implies (\ref{item:A/m}).

(\ref{item:A/m} $\iff$ \ref{item:in m}) 
Easy.

Assume $e = 1$.

(\ref{item:eigenvectors} $\iff$ \ref{item:Dm})
Assume $D(\idealm_w) \subset \idealm_w$.
Take a system of generators $(x_j)$ of $\idealm_w$.
For each $j$ let $x_j = \sum_{i \in \bF_p} x_{j,i}$ be the decomposition of $x_j$ in $B = \bigoplus_{i} B_i$.
Then $D^l(x_j) = \sum_i i^l x_{j,i}$ is also in $\idealm_w$.
Since the matrix $(i^l)_{i,l=0}^{p-1}$ is invertible,
this implies $x_{j,i} \in \idealm_w$.
Thus $\idealm_w$ is generated by eigenvectors.
The converse is clear.

(\ref{item:Dm} $\iff$ \ref{item:DA}) This is clear since $\cO_{X,w} = \idealm_w + k$ and $D \restrictedto{k} = 0$.

(\ref{item:Dm} $\iff$ \ref{item:RS})
Take coordinates $x_1, \ldots, x_n$ at a point $w$ 
and write $D = \sum_j f_j \cdot (\partial / \partial x_j)$.
Then the both conditions are equivalent to $(f_j) \subset \idealm_w$.

We show the final assertion assuming (\ref{item:eigenvectors}).
If the maximal ideal $\idealm$ is generated by homogeneous elements $x_j \in B_{i_j}$,
then for each $j$ we can extend the action on the affine piece $\Spec B[x_h/x_j]_h$ of $\Bl_w X$
by declaring $x_h/x_j$ to be homogeneous of weight $i_h - i_j$.
\end{proof}

The next lemma enables us to take useful coordinates at a point not fixed by $D$.

\begin{lem} \label{lem:good coord}
If $B$ is a Noetherian local ring, $D$ is a derivation of multiplicative type, and the closed point is not fixed by $D$,
then the maximal ideal $\idealm$ of $B$ is generated by elements $x_1, \ldots, x_{m-1}, y$
with $\wt(x_j) = 0$ and $\wt(1+y) = 1$.
If $\idealm$ is generated by $n$ elements then we can take $m = n$.
If $\dim B \geq 2$ then $D$ does not extend to a derivation of the blow-up $\Bl_{\idealm} B$.
\end{lem}
\begin{proof}
Recall that a subset of $\idealm$ generates $\idealm$
if and only if it generates $\idealm/\idealm^2$.

Take elements $x'_1, \dots, x'_m$ generating $\idealm$
and let $x'_j = \sum_{i \in \bF_p} x'_{j,i}$ be the decompositions to eigenvectors.
By assumption there exists a pair $(j,i)$ with $x'_{j,i} \not\in \idealm$. 
We take such $j_0,i_0$ and we may assume $i_0 \neq 0$.
We may assume $x'_{j_0,i_0} - 1 \in \idealm$.
Then $y = x'_{j_0,i_0} - 1 $ satisfies $y \in \idealm$ and $D(y) = i_0 (y + 1)$.
We have $y \not\in \idealm^2$,
since $D(\idealm^2) \subset \idealm$.
By replacing $y$ with $(y+1)^q - 1$ for 
an integer $q$ with $q i_0 \equiv 1 \pmod p$
we may assume $i_0 = 1$.
For each $j$, let $x_j = \sum_i (y+1)^{-i} x'_{j,i}$.
Then we have $D(x_j) = 0$ and, 
since $x_j \equiv x'_j \pmod{(y)}$,
the elements $x_j,y$ generate $\idealm/\idealm^2$ and hence generate $\idealm$.
We can omit one of the $x_j$'s and then the remaining elements satisfies the required conditions
(after renumbering).

To show the latter assertion it suffices to show that $D$ does not extend to $B' := B[x_j/y]_j$.
If it extends then we have $D(x_j/y) = -x_j(y+1)/y^2 \in B'$,
hence $x_j/y^2 \in B'$
and then on $\Spec B'$ we have that $y = 0$ implies $x_j/y = 0$,
which is impossible since $\dim B' \geq 2$.
\end{proof}

Before stating the next proposition we recall the following notion from \cite{Rudakov--Shafarevich:inseparable}. 
	Assume $X$ is a smooth irreducible variety and $D$ is a nontrivial derivation. 
	Then $\Fix(D)$ consists of its divisorial part $\divisorialfix{D}$ and non-divisorial part $\isolatedfix{D}$.
	If we write $D = f \sum_i g_i \partialdd{x_i}$ for some local coordinates $x_1, \dots, x_m$
	with $g_i$ having no common factor,
	then $\divisorialfix{D}$ and $\isolatedfix{D}$ corresponds to the ideal $(f)$ and $(g_i)$ respectively.
If $D$ is of multiplicative type with $\isolatedfix{D} = \emptyset$, 
then it follows from Proposition \ref{prop:fixedpoints}
that for suitable coordinates near any fixed point we have $D = a x_m \cdot (\partial / \partial x_m)$
and that $\Fix(D)$ is a smooth divisor (possibly empty).

Assuming that $\Fix(D)$ is divisorial,
in which case the quotient is a smooth variety by \cite{Seshadri:Cartier}*{Proposition 6},
the highest differential forms on smooth loci of $X$ and $X^D$ are related in the following way.

\begin{prop} \label{prop:2-forms}
Let $X$ be a smooth variety of dimension $m$ (not necessarily proper)
equipped with a nontrivial derivation $D$ of multiplicative type
such that $\Fix(D)$ is divisorial.
Let $\Delta$ be the divisor $\Fix(D)$.
Then there is a unique collection of isomorphisms 
\[
(\pi_* (\Omega^m_{X/k}(\Delta))^{\otimes n})_0
\cong (\Omega^m_{X^D/k}(\pi_*(\Delta)))^{\otimes n}
\]
for all integers $n$,
compatible with multiplication, preserving the zero loci, 
and sending (for $n = 1$)
\[ f_0 \cdot df_1 \wedge \dots \wedge df_{m-1} \wedge d\log(f_m)
\mapsto
   f_0 \cdot df_1 \wedge \dots \wedge df_{m-1} \wedge d\log(f_m^p) \]
if $f_0, \dots, f_{m-1}$ are homogeneous of weight $0$
and $f_m$ is homogeneous of some weight (not necessarily $0$).

In particular, if the action is fixed-point-free,
then we have isomorphisms 
\begin{align*}
(\pi_* (\Omega^m_{X/k})^{\otimes n})_0
&\cong (\Omega^m_{X^D/k})^{\otimes n} 
\quad \text{and} \\
H^0(X, (\Omega^m_{X/k})^{\otimes n})_0 
&\cong H^0(X^D, (\Omega^m_{X^D/k})^{\otimes n})
\end{align*}
with the same properties.
\end{prop}

\begin{proof}
The isomorphism for $n = 0$ is clear. 
It suffices to construct the isomorphism for $n = 1$ 
that is compatible with multiplication with $n = 0$ forms 
and with restriction to open subschemes.

Take a closed point $w \in X$.
Let $\varepsilon = 1$ (resp.\ $\varepsilon = 0$) if $w \not\in \Delta$ (resp.\ $w \in \Delta$).
By Lemma \ref{lem:good coord} (resp.\ by \cite{Rudakov--Shafarevich:inseparable}*{Theorem 1}),
there are coordinates $x_1, \dots, x_m$ on a neighborhood of $w$
with $D(x_j) = 0$ for $j < m$
and $D(x_m) = a (\varepsilon + x_m)$ for some $a \in \bF_p^*$. 
We define 
\begin{align*}
\phi \colon (\pi_* (\Omega^m_{X/k}(\Delta)))_0
&\to \Omega^m_{X^D/k}(\pi_*(\Delta)) \\
f \cdot dx_1 \wedge \dots \wedge dx_{m-1} \wedge d\log(\varepsilon + x_m)
&\mapsto f \cdot dx_1 \wedge \dots \wedge dx_{m-1} \wedge d\log(\varepsilon + x_m^p)
\end{align*}
for $f$ of weight $0$
(note that $dx_1 \wedge \dots \wedge dx_{m-1} \wedge d\log(\varepsilon + x_m)$ is a local generator of the left-hand side).
We show that then $\phi$ sends
\[
f_0 \cdot df_1 \wedge \dots \wedge df_{m-1} \wedge d\log(f_m)
\mapsto f_0 \cdot df_1 \wedge \dots \wedge df_{m-1} \wedge d\log(f_m^p)
\]
for any $f_0, \dots, f_{m-1}$ and $f_m$ as in the statement.
This implies that $\phi$ does not depend on the choice of the coordinates
and hence that $\phi$ induces a well-defined morphism of sheaves.
Then since $dx_1 \wedge \dots \wedge dx_{m-1} \wedge d\log(\varepsilon + x_m)$ 
(resp.\ $dx_1 \wedge \dots \wedge dx_{m-1} \wedge d\log(\varepsilon + x_m^p)$)
is a local generator of $(\Omega^m_{X/k}(\Delta))_0$ 
(resp.\ $\Omega^m_{X^D/k}(\pi_*(\Delta))$),
it follows that $\phi$ is an isomorphism
and $\phi^{\otimes n}$ are well-defined isomorphisms.

We may pass to the completion, so consider $f_h \in k[[x_1, \dots, x_m]]$.
By the assumption on the weight we have 
$f_h \in k[[x_1, \dots, x_{m-1}, x_m^p]]$ for $h < m$
and $f_m \in (\varepsilon + x_m)^b k[[x_1, \dots, x_{m-1}, x_m^p]]$ for some $0 \leq b < p$.
Then we have $\partial f_h / \partial x_m = 0$ for $h < m$
and $\partial f_m / \partial x_m = b f_m/(\varepsilon + x_m)$.
Hence we have 
\begin{align*}
 & f_0 \cdot df_1 \wedge \dots \wedge df_{m-1} \wedge d\log(f_m) \\
 &= f_0 ((\varepsilon + x_m)/f_m) \det (\partial f_h / \partial x_j)_{1 \leq h,j \leq m} \cdot dx_1 \wedge \dots \wedge dx_{m-1} \wedge d\log(\varepsilon + x_m) \\
 &= b f_0 \det (\partial f_h / \partial x_j)_{1 \leq h,j \leq m-1} \cdot dx_1 \wedge \dots \wedge dx_{m-1} \wedge d\log(\varepsilon + x_m) \\
 &\stackrel{\phi}{\mapsto}
    b f_0 \det (\partial f_h / \partial x_j)_{1 \leq h,j \leq m-1} \cdot dx_1 \wedge \dots \wedge dx_{m-1} \wedge d\log(\varepsilon + x_m^p).
\end{align*}
On the other hand, 
in the invariant subalgebra $k[[x_1, \dots, x_{m-1}, x_m^p]]$
we have $\partial f_m^p / \partial x_j = 0$ for $j < m$ and
$\partial f_m^p / \partial x_m^p = b f_m^p/(\varepsilon + x_m^p)$.
Hence we have 
\begin{align*}
 &  f_0 \cdot df_1 \wedge \dots \wedge df_{m-1} \wedge d\log(f_m^p) \\
 &= \dots 
  = b f_0 \det (\partial f_h / \partial x_j)_{1 \leq h,j \leq m-1} \cdot dx_1 \wedge \dots \wedge dx_{m-1} \wedge d\log(\varepsilon + x_m^p).
\end{align*}
The assertion follows.
\end{proof}
We will give another abstract proof of Proposition \ref{prop:2-forms}
in \cite{Matsumoto:k3alphap}*{Proposition 2.12}.

\subsection{Global properties of derivations}

\begin{lem} \label{lem:derivation on quadratic curve}
	Let $C \subset \bP^2$ be a quadratic curve (not necessarily irreducible nor reduced) in characteristic $p$ 
	and $D$ a $p$-closed derivation.
	Then $\Fix(D) \neq \emptyset$.
\end{lem}
\begin{proof}
	Suppose $C$ is integral. Then $C \cong \bP^1$ and the result is well-known
	(indeed, $T_{\bP^1} \cong \cO_{\bP^1}(2)$).
	
	Suppose $C$ is reducible.
	We may assume $C = (xy = 0)$.
	Let $U = \Spec k[x,y] / (xy)$.
	Then $T_U = ( x \frac{d}{dx} - y \frac{d}{dy}) \cdot \cO_U$
	and hence the origin belongs to $\Fix(D)$.
	
	Suppose $C$ is non-reduced.
	We may assume $C = (X_3^2 = 0)$.
	If $p \neq 2$, then $D$ induces a derivation $D_{\red}$ on $C_{\red} \cong \bP^1$,
	and we have $\Fix(D) \approx \Fix(D_{\red}) \neq \emptyset$.
	Suppose $p = 2$.
	It is easy to see that $H^0(C, T_C) \isomto H^0(C, T_{\bP^2} \restrictedto{C}) \isomfrom H^0(\bP^2, T_{\bP^2})$.
	Hence there exist $f_1, f_2, f_3 \in H^0(\bP^2, \cO(1)) = \bigoplus_{i=1}^3 k X_i$
	such that $D\bigl(\frac{X_i}{X_j}\bigr) = \frac{f_i}{X_j} - \frac{X_i f_j}{X_j^2}$.
	If $f_3 \in k X_3$, then $D$ induces a derivation $D_{\red}$ on $C_{\red} \cong \bP^1$, and we conclude as above.
	Suppose $f_3 \not\in k X_3$. 
	By a coordinate change we may assume $f_3 - X_2 \in k X_3$.
	Letting $x_i = X_i/X_1$ ($i = 2,3$) and restricting to $\Spec k[x_2,x_3]/(x_3^2) = (X_1 \neq 0) \subset C$,
	we have $D(x_3) - x_2 \in (x_3)$,
	in particular $D(x_3) \in \idealm := (x_2, x_3)$.
	If $D(x_2) \in \idealm$ then the origin is a fixed point.
	Suppose $D(x_2) \not\in \idealm$ and $D^2 = h D$.
	Then $h = D^2(x_2) / D(x_2) \in \cO_{\idealm}$, hence 
	$D(x_2) \equiv D^2(x_3) = h D(x_3) \equiv 0 \pmod{\idealm}$, contradiction.
\end{proof}

\begin{cor} \label{cor:fixed point above}
	Suppose $\mu_p$ acts on an RDP surface $X$ and fixes an RDP $w$.
	Then the action extends to the blow-up $\Bl_w X$
	and there exists a fixed point above $w$.
\end{cor}
\begin{proof}
The action extends to the blow-up by Proposition \ref{prop:fixedpoints}.
Let $D'$ be the induced derivation on $\Bl_w X$.
Let $C \subset \Bl_w X$ be the (possibly non-reduced) exceptional divisor,
which is a quadratic curve in $\bP^2$ since $w$ is an RDP.
Since $D'(\cI_C) \subset \cI_C$, 
$D'$ induces a derivation $D'_{C}$ (of multiplicative type) on $C$.
By Lemma \ref{lem:derivation on quadratic curve}, $D'_C$ has at least one fixed point,
and that point is also a fixed point of $D'$.
\end{proof}

Later we will also need the following Katsura--Takeda formula
on \emph{rational} vector fields (i.e.\ derivations on the fraction field $k(X)$).
For a rational derivation $D$ locally of the form $f^{-1} D'$ for some regular function $f$ and (regular) derivation $D'$,
we define the divisorial and non-divisorial parts by 
$\divisorialfix{D} = \divisorialfix{D'} - \divisor(f)$ and $\isolatedfix{D} = \isolatedfix{D'}$.

\begin{prop}[\cite{Katsura--Takeda:quotients}*{Proposition 2.1}] \label{prop:globalvectorfields}
Let $X$ be a smooth proper surface and $D$ a nonzero rational vector field.
Then we have
\[
 \deg c_2(X) = \deg \isolatedfix{D}  - K_X \cdot \divisorialfix{D} - \divisorialfix{D}^2.
\]
\end{prop}

\section{Tame symplectic actions on RDPs} \label{sec:tame}

Hereafter, all action of groups and group schemes on schemes are assumed faithful.

Throughout this section we work under the following setting.
$B = \cO_{X,w}$ is the localization of an RDP surface $X$ over an algebraically closed field $k$ at a closed point $w$ (either a smooth point or an RDP),
$\idealm \subset B$ is the maximal ideal,
$G$ is a finite group acting on $X$, and the action restricts to $\Spec B$.
Assume the order of $G$ is not divisible by $p = \charac k$.

\begin{defn} \label{def:symplectic local tame}
We say that the $G$-action on $B$ is \emph{symplectic}
if it acts on the $1$-dimensional $k$-vector space $H^0(\Spec B \setminus \set{\idealm}, \Omega^2_{B/k}) \otimes_B (B/\idealm)$ trivially.
\end{defn}

If $G = \bZ/p\bZ$, then any action is symplectic (cf.\ Remark \ref{rem:commutative}), hence the notion is useless in this case.

\begin{rem} \label{rem:invariant tame}
If $B$ is as above and the $G$-action is symplectic,
then the rank $1$ free $B$-module $H^0(\Spec B \setminus \set{\idealm}, \Omega^2_{B/k})$
admits a generator $\omega$ that is $G$-invariant.
Indeed, take a generator $\omega'$,
then $\omega := (1 / \card{G}) \sum_{g \in G} g^* \omega'$
is clearly $G$-invariant and it is non-vanishing, since it is non-vanishing after $\otimes (B / \idealm)$.
\end{rem}
\begin{rem} \label{rem:check locally tame}
If $X$ is an RDP K3 surface and $w \in X$ is a fixed closed point,
then this is consistent with the usual notion of symplecticness,
since a generator of $H^0(X^{\sm}, \Omega^2) \cong H^0(\tilde{X}, \Omega^2)$ (Proposition \ref{prop:2-forms of RDP}) restricts to a generator of this $1$-dimensional space.
Thus the symplecticness of an automorphism of an RDP K3 surface can be checked locally at any fixed point (if there exists any).
Same for abelian surfaces.
\end{rem}

\begin{prop} \label{prop:tame}
Let $B$ and $G$ be as above (in particular, the order of $G$ is not divisible by $p = \charac k$).
Then the invariant ring $B^G$ is again the localization at a closed point of an RDP surface.

Let $\tilde{X} \to X$ be the minimal resolution at $w$.
Then $\tilde{X}/G \to X/G$ is crepant.
\end{prop}

\begin{proof}
Let $\omega$ be a generator of the rank $1$ free $B$-module $H^0(\Spec B \setminus \set{\idealm}, \Omega^2_{B/k})$.
By Remark \ref{rem:invariant tame}, we may assume $\omega$ is $G$-invariant.
The action of $G$ on $X$ induces an action on $\tilde{X}$, 
and $\omega$ extends to a regular non-vanishing $2$-form on $\tilde{X}$.
At each closed point $w' \in \tilde{X}$ the stabilizer $G_{w'} \subset G$ acts on $T_{w'} \tilde{X}$ via $\SL_2(k)$ since $G$ preserves $\omega$.
Hence the quotient $\tilde{X}/G$ has only RDPs as singularities. 
Since $\omega$ is preserved by $G$ it induces a regular non-vanishing $2$-form on $(\tilde{X}/G)^\sm$, and
since RDPs are canonical singularities it extends to a regular non-vanishing $2$-form on $\widetilde{\tilde{X}/G}$, the minimal resolution of $\tilde{X}/G$ above $w$.
Thus $B^G$ is a canonical singularity, that is, either a smooth point or an RDP.
\end{proof}

\begin{rem} \label{rem:tame action}
We \cite{Matsumoto:extendability}*{Proposition 3.8} described possible symplectic actions of finite tame groups on RDPs.
For actions of cyclic groups $G = \bZ/n\bZ$ ($n > 1$) we have a complete classification:
possible $n$ and the types of $X$ and $X/G$ are listed in Table \ref{table:tame RDP}.
\begin{table} 
\caption{Tame symplectic cyclic actions on RDPs} \label{table:tame RDP}
\begin{tabular}{llll}
\toprule
$n = \card{G}$ & $X$ &                            & $X/G$ \\
\midrule
any & $A_{m-1}$ &                      & $A_{mn-1}$ \\
$2$ & $A_{m-1}$ & ($m \geq 4$ even)    & $D_{m/2+2}$ \\
$4$ & $A_{m-1}$ & ($m \geq 3$ odd)     & $D_{m+2}$ \\
$3$ & $D_4$     & ($p \neq 2$)         & $E_6$ \\
$3$ & $D_4^r$   & ($p = 2$, $r = 0,1$) & $E_6^r$ \\
$2$ & $D_{m+2}$ &                      & $D_{2m+2}$ \\
$2$ & $E_6$     & ($p \neq 3$)         & $E_7$ \\
$2$ & $E_6^r$   & ($p = 3$, $r = 0,1$) & $E_7^r$ \\
\bottomrule
\end{tabular}
\end{table}
\end{rem}

\begin{rem} \label{rem:wild}
Singularities of quotients by order $p$ automorphisms in characteristic $p > 0$ tends to be worse than those in characteristic $\neq p$.
For example, the quotient of a supersingular abelian surface in characteristic $2$ by the automorphism $x \mapsto -x$
is a rational surface with an elliptic singularity
\cite{Katsura:Kummer2}*{Theorem C}.
\end{rem}

\section{\texorpdfstring{$\mu_n$-actions}{mu\_n-actions} on RDPs and quotients} \label{sec:mu_p RDP}

Throughout this section we work under the following setting.
$B = \cO_{X,w}$ is the localization of an RDP surface $X$ over an algebraically closed field $k$ of characteristic $p \geq 0$ at a closed point $w$ (either a smooth point or an RDP),
$\idealm \subset B$ is the maximal ideal,
$n$ is a positive integer possibly divisible by $p$,
$\mu_n$ acts on $X$, and the action restricts to $\Spec B$.
(Note that $w$ is not necessarily fixed by $\mu_n$.)

If $n = p > 0$, then the corresponding derivation of multiplicative type is denoted by $D$.

\subsection{Symplecticness of \texorpdfstring{$\mu_n$-actions}{mu\_n-actions}}

Assume $w$ is fixed by the $\mu_n$-action.
Then the action on $B$ induces an action on $V := H^0(\Spec B \setminus \set{\idealm}, \Omega^2_{B/k}) \otimes_B (B/\idealm)$,
that is, a decomposition $V = \bigoplus_{i \in \bZ/n\bZ} V_i$ of $k$-vector spaces.
Since $\dim_k V = 1$, $V$ is equal to one of the summands.
In other words, $V$ is of some weight $i_0 \in \bZ/n\bZ$.

\begin{defn} \label{def:symplectic local}
We say that the $\mu_{n}$-action, or the corresponding derivation if $n = p$, on $B$ is \emph{symplectic} if $V$ is of weight $0$.

We say that a $\mu_{n}$-action, or a derivation $D$ of multiplicative type, on an RDP surface $X$ is \emph{symplectic} at a fixed closed point $w$ 
if the induced action or derivation on $\cO_{X,w}$ is symplectic in the above sense.
\end{defn}
\begin{rem} \label{rem:consistency of symplecticness}
If $p \notdivides n$,
then $\mu_n$ is (non-canonically) isomorphic to $\bZ/n\bZ$ and this definition is consistent with Definition \ref{def:symplectic local tame}.
\end{rem}
\begin{rem}[cf.\ Remark \ref{rem:invariant tame}] \label{rem:D=i}
If $B$ is as above and $V$ is of weight $i_0$, 
then the rank $1$ free $B$-module $H^0(\Spec B \setminus \set{\idealm}, \Omega^2_{B/k})$
admits a generator $\omega$ of weight $i_0$.
Indeed, take a generator $\omega'$,
let $\omega' = \sum_{i} \omega'_i$ be its decomposition, 
and write $\omega'_i = f_i \omega'$ with $f_i \in B$.
Since $\sum f_i = 1$, there exists $i_1 \in \bZ/n\bZ$ with $f_{i_1} \in B^*$.
Then $i_0 = i_1$ and hence we can take $\omega = \omega'_{i_1}$.
If $n = p$ then this means $D(\omega) = i_0 \omega$.

From this it follows that if $\mu_{n}$ acts on an RDP surface 
then the weight $i_0$ is a locally constant function on the fixed locus.
\end{rem}
\begin{rem}[cf.\ Remark \ref{rem:check locally tame}] \label{rem:check locally mu_p}
If $X$ is an RDP K3 surface and $w \in X$ is a fixed closed point, 
then the action is symplectic in the sense of Definition \ref{def:symplectic} 
if and only if action is symplectic at $w$,
since a generator of $H^0(X^{\sm}, \Omega^2)$ restricts to a generator of this $1$-dimensional space.
Thus the symplecticness of a $\mu_{n}$-action on an RDP K3 surface can be checked locally at any fixed point (if there exists any).
Same for abelian surfaces.
\end{rem}

\begin{lem} \label{lem:mu_n RDP}
	Suppose the closed point $w$ of $B$ is fixed under the $\mu_n$-action.
	Then 
	$B$ is generated by $2$ or $3$ homogeneous elements
	respectively if $B$ is smooth or an RDP. Moreover,
	\begin{enumerate}
	\item \label{item:mu_n RDP:smooth} 
	If $B$ is smooth and generated by elements $x,y$ of respective weight $a,b$,
	then the action is symplectic if and only if $a + b = 0$ (in $\bZ/n\bZ$).
	\item \label{item:mu_n RDP:RDP}
	If $B$ is an RDP and generated by $x,y,z$ of respective weight $a,b,c$,
	then there is $d \in \bZ/n\bZ$ and a homogeneous power series $F \in k[[x,y,z]]$ of weight $d$ 
	such that $\hat{B} \cong k[[x,y,z]]/(F)$.
	The action is symplectic if and only if $a + b + c = d$.
	\end{enumerate}
\end{lem}
\begin{proof}
The first assertion follows from Proposition \ref{prop:fixedpoints}.

(\ref{item:mu_n RDP:smooth}) 
$H^0(\Spec B \setminus \set{\idealm}, \Omega^2_{B/k})$ is generated by 
$dx \wedge dy$, which is of weight $a+b$.

(\ref{item:mu_n RDP:RDP})
Take an element $H \in k[[x,y,z]]$ such that $\hat{B} = k[[x,y,z]] / (H)$,
and let $H = \sum_{i \in \bZ/n\bZ} H_i$ be the decomposition with respect to the $\mu_n$-action.
Since $H = 0$ in $\hat{B}$, we have $H_i = 0$ in $\hat{B}$, 
hence there are $f_i \in k[[x,y,z]]$ such that $H_i = f_i H$.
Since $\sum f_i = 1$, there exists $d \in \bZ/n\bZ$ with $f_d \in k[[x,y,z]]^*$.
We can take $F = H_d$, which is of weight $d$.

Then $H^0(\Spec \hat{B} \setminus \set{\idealm}, \Omega^2_{\hat{B}/k})$ is generated by 
$\omega = F_x^{-1} dy \wedge dz = F_y^{-1} dz \wedge dx = F_z^{-1} dx \wedge dy$
(this means that the restriction of $\omega$ to the open subscheme $\Spec \hat{B}[F_x^{-1}]$ is equal to $F_x^{-1} dy \wedge dz$, and so on),
and we have $\wt(\omega) = a + b + c - d$ 
since $\wt(F_x^{-1}) = -(d - a)$ and $\wt(dy \wedge dz) = b + c$, and so on.
\end{proof}

\subsection{\texorpdfstring{$\mu_p$-actions}{mu\_p-actions} on RDPs}

As noted in Section \ref{subsec:derivation},
we know by \cite{Seshadri:Cartier}*{Proposition 6} (see also \cite{Rudakov--Shafarevich:inseparable}*{Theorem 1 and Corollary}) 
that the quotient of a smooth variety by a $\mu_p$-action with no isolated fixed point is smooth.
We need to consider, more generally, the quotients of surfaces with RDP singularities and with isolated fixed points.

Let $\cO_{X,w}$ and $\mu_n$ be as in the beginning of Section \ref{sec:mu_p RDP},
and suppose $n = p$.
Let $\pi \colon X \to Y = X/\mu_p$ be the quotient morphism.

\begin{thm} \label{thm:mu_p RDP}
\begin{enumerate}
\item \label{thm:mu_p RDP:non-fixed RDP}
Assume $w$ is non-fixed.
If $w$ is a smooth point then $\pi(w) \in Y$ is also a smooth point.
If $w$ is an RDP then $\pi(w)$ is either a smooth point or an RDP.
In either case $X \times_Y \tilde{Y} \to X$ is crepant, where $\tilde{Y} \to Y$ is the minimal resolution at $\pi(w)$.
\item \label{thm:mu_p RDP:fixed RDP symplectic}
If $w$ is fixed and the action is symplectic at $w$,
then $w$ is an isolated fixed point and $\pi(w)$ is an RDP.

\item \label{thm:mu_p RDP:fixed RDP non-symplectic}
If $w$ is an isolated fixed point and the action is non-symplectic at $w$,
then $\pi(w)$ is a non-RDP singularity.
\end{enumerate}
\end{thm}

First we consider non-symplectic actions on isolated fixed points.
\begin{proof}[Proof of Theorem \ref{thm:mu_p RDP}(\ref{thm:mu_p RDP:fixed RDP non-symplectic})]
	By Proposition \ref{prop:2-forms},
	we have an isomorphism 
	\[ (H^0(\Spec \cO_{X,w} \setminus \set{w}, \Omega^2))_0 \cong 
	H^0(\Spec \cO_{Y,\pi(w)} \setminus \set{\pi(w)}, \Omega^2) \]
	preserving the zero loci of $2$-forms.
	If $\pi(w)$ is either a smooth point or an RDP then the right hand side has a non-vanishing $2$-form 
	and hence there is a non-vanishing form $\omega$ on $\Spec \cO_{X,w} \setminus \set{w}$ of weight $0$.
	Being non-vanishing, $\omega$ is a generator of $H^0(\Spec \cO_{X,w} \setminus \set{w}, \Omega^2)$.
	But this contradicts the non-symplecticness assumption.
\end{proof}

Next we consider non-fixed points.
In fact, we can classify all possible actions and give explicit equations.
\begin{prop} \label{prop:non-fixed mu_p RDP}
Assume $w$ is not fixed.
	\begin{itemize}
\item
	If $w$ is a smooth point, then there are coordinates $x,y$ of $\cO_{X,w}$
	satisfying $D(x) = 0$ and $D(y) \neq 0$,
	hence $\cO_{Y,\pi(w)}$ has $x,y^p$ as coordinates and in particular $\pi(w)$ is a smooth point.
\item
	If $w$ is an RDP, then there is an element $F \in k[[x,y,z^p]]$ and 
	an isomorphism $\hat{\cO}_{X,w} \cong k[[x,y,z]] / (F)$
	with $D(x) = D(y) = 0$ and $D(z) \neq 0$,
	hence $\hat{\cO}_{Y,\pi(w)} \cong k[[x,y,z^p]] / (F)$.
	Moreover we can take $F$ to be one in Table \ref{table:non-fixed RDP}.
\end{itemize}
\end{prop}
\begin{table} 
\caption{Non-fixed $\mu_p$-actions on RDPs} \label{table:non-fixed RDP} 
\begin{tabular}{llllll} 
\toprule
 $p$ & equation            &              & $X$          & $Y = X^D$ & $X \times_Y \tilde{Y}$ \\ 
\midrule
 any & $xy + z^{mp}$       & ($m \geq 2$) & $A_{mp-1}$   & $A_{m-1}$ & $mA_{p-1}$ \\
 any & $xy + z^{p}$        &              & $A_{p-1}$    & smooth    & --- \\
\midrule
 $5$ & $x^2 + y^3 + z^5$   &              & $E_8^0$      & smooth    & --- \\
\midrule
 $3$ & $x^2 + z^3 + y^4$   &              & $E_6^0$      & smooth    & --- \\
 $3$ & $x^2 + y^3 + yz^3$  &              & $E_7^0$      & $A_1$     & $E_6^0$  \\
 $3$ & $x^2 + z^3 + y^5$   &              & $E_8^0$      & smooth    & --- \\
\midrule
 $2$ & $z^2 + x^2y + xy^m$ & ($m \geq 2$) & $D_{2m}^0$   & smooth    & --- \\
 $2$ & $x^2 + yz^2 + xy^m$ & ($m \geq 2$) & $D_{2m+1}^0$ & $A_1$     & $D_{2m}^0$  \\
 $2$ & $x^2 + xz^2 + y^3$  &              & $E_6^0$      & $A_2$     & $D_4^0$  \\
 $2$ & $z^2 + x^3 + xy^3$  &              & $E_7^0$      & smooth    & --- \\
 $2$ & $z^2 + x^3 + y^5$   &              & $E_8^0$      & smooth    & --- \\
\bottomrule
\end{tabular} 

\end{table}

\begin{proof}[Proof of Theorem \ref{thm:mu_p RDP}(\ref{thm:mu_p RDP:non-fixed RDP}) and Proposition \ref{prop:non-fixed mu_p RDP}]
If $w$ is a smooth point
then taking coordinates $x,y$ as in Lemma \ref{lem:good coord} (i.e.\ $D(x) = 0$ and $D(y) = 1+y$)
we have $\hat \cO_{Y,\pi(w)} \cong k[[x,y^p]]$,
hence $\cO_{Y,\pi(w)}$ is smooth.

Assume $w$ is an RDP.
By Lemma \ref{lem:good coord} we have coordinates $x,y,z$ satisfying $D(x) = D(y) = 0$ and $D(z) \neq 0$.
We have $\hat{\cO}_{X,w} \cong k[[x,y,z]] / (F)$
for some $F \in k[[x,y,z]]$ such that $D(F) \in (F)$, and we may assume $F \in k[[x,y,z^p]]$.
We show that, 
after replacing $F$ with a multiple by a unit,
and after a coordinate change of $k[[x,y,z]]$ that preserves the subring $k[[x,y,z^p]]$,
$F$ coincides with one in Table \ref{table:non-fixed RDP}.
(Such coordinate changes are given by $x',y',z' \in \idealm$ that are linearly independent in $\idealm/\idealm^2$ and satisfy $x',y' \in \idealm \cap k[[x,y,z^p]]$.)
A similar classification is given in \cite{Ekedahl--Hyland--Shepherd-Barron}*{Proposition 3.8}, but they missed the case of $E_7^0$ in characteristic $3$.

Assume the classification for the moment. 
Then in each case we observe that $\pi(w)$ is either a smooth point or an RDP,
and it is straightforward to check that $X \times_Y \tilde{Y}$ is an RDP surface crepant over $X$.
(In Table \ref{table:non-fixed RDP}, the entries of the singularities of $X \times_Y \tilde{Y}$ is omitted if $Y$ is already smooth.)
For example, consider $X = \Spec k[x,y,z]/(F)$, $F = xy + z^{mp}$ with $m \geq 2$.
Then $X' := X \times_Y \Bl_{\pi(w)} Y$ is covered by three affine pieces 
\begin{align*}
X'_1 &= \Spec k[x, y_1, v_1, z] / (y_1 + x^{m-2} v_1^{m}, x v_1 - z^p), & y_1 &= y/x,   & v_1 &= z^p/x,   \\
X'_2 &= \Spec k[x_2, y, v_2, z] / (x_2 + y^{m-2} v_2^{m}, y v_2 - z^p), & x_2 &= x/y,   & v_2 &= z^p/y,   \\
X'_3 &= \Spec k[x_3, y_3, z]    / (x_3 y_3 + z^{(m-2)p}),               & x_3 &= x/z^p, & y_3 &= y/z^p. 
\end{align*}
One observes that $\Sing(X')$ consists of two RDPs of type $A_{p-1}$ at the origins of $X'_1$ and $X'_2$ and, if $m \geq 3$, one RDP of type $A_{(m-2)p-1}$ at the origin of $X'_3$.
Repeating this we observe that $X \times_Y \tilde{Y}$ has $m A_{p-1}$.

\medskip

Now we show the classification.
We say that $F$ \emph{has} a monomial if the coefficient of that monomial is nonzero.
We also write $F = \sum_{h,i,j} a_{hij} x^h y^i z^j$.

First assume $p > 2$.
We may assume that the degree $2$ part $F_2$ is either $xy$ or $x^2$.
Assume $F_2 = xy$.
We may assume $F$ has no $xz^{j}$ and $yz^{j}$. $F$ must have $z^{j}$, $j = mp$, and then it is $A_{mp-1}$.
Then, by replacing $x$ with $x + a_{0ij} y^{i-1} z^j$ and $y$ with $y + a_{h0j} x^{h-1} z^j$, and so on,
we may assume $F$ has no $y^i z^j$ with $i > 0$ and no $x^h z^j$ with $h > 0$.
Thus $F = u_1 xy + u_2 z^{mp}$ for some units $u_1, u_2$, and then by replacing $x,y,F$ by suitable multiples we obtain $F = xy + z^{mp}$.

Assume $p > 3$ and $F_2 = x^2$. We may assume that the degree $3$ part $F_3$ is $y^3$. 
If $p \geq 7$ it cannot be an RDP.
If $p = 5$ then $F$ must have $z^5$ and then it is $E_8^0$.
We have $F = u_1 x^2 + u_2 y^3 + u_3 z^5$, and then by replacing $x,y,F$ by suitable multiples, 
we obtain $F = x^2 + y^3 + z^5$.
(For example, we let $F = u_3 F'$, $x = (u_3 u_1^{-1})^{1/2} x'$, $y = (u_3 u_2^{-1})^{1/3} y'$.
Note that we can take $n$-th roots of units provided $p \notdivides n$.)

Assume $p = 3$ and $F_2 = x^2$.
We may assume $F_3 = y^3$ or $F_3 = z^3$.
If $F_3 = z^3$ then $F$ must have $y^4$ or $y^5$ and then it is $E_6^0$ or $E_8^0$.
We may assume $a_{130} = a_{140} = 0$ by replacing $x$ with $x + (1/2)(a_{130} y^3 + a_{140} y^4)$, and then we transform $F$ as above.
If $F_3 = y^3$ then $F$ must have $yz^3$ and then it is $E_7^0$.
We eliminate $a_{1ij}$ as above, then we have $F = u_1 x^2 + u_2 y^3 + u_3 yz^3 + z^6 g(z^3)$ for some power series $g \in k[[z^3]]$. 
We may assume $u_i \equiv 1 \pmod{\idealm}$.
We eliminate $g$ by replacing $y$ with $y + z^3 g$, and then we transform $F$ as above.

Now consider $p = 2$.
We may assume $F_2$ is one of $xy + z^2$ (if irreducible), $xy$ (if reducible but not a square), $z^2$, or $x^2$ (square, of a linear factor containing $z$ or not).
If $F_2 = xy + z^2$ or $F_2 = xy$, then as above it is $A_{mp-1}$ and $F$ becomes $xy + z^{mp}$. 

Assume $p = 2$ and $F_2 = x^2$.
If $F_3$ has $yz^2$ then $F$ must have $xy^m$ and then it is $D_{2m+1}^0$.
We obtain $F = u_1 x^2 + u_2 y z^2 + u_3 x y^m + z^4 g(z) + f(y) + y^{2m} g'(y)$, 
where $f(y) = f_0(y)^2 + y f_1(y)^2$ is a polynomial of degree $< 2m$,
$g(z) \in k[[z]]$, and $g'(y) \in k[[y]]$.
We may assume $u_i \equiv 1 \pmod{\idealm}$.
We eliminate $f$ by replacing $x$ with $x + f_0(y)$ and $z$ with $z + f_1(y)$ and so on.
Then we eliminate $g$ and $g'$ by replacing $y$ and $x$ suitably, and take multiples by units as above.

If $F_3$ has no $yz^2$ then $F$ must have $y^3$ and $xz^2$ and then it is $E_6^0$.
We obtain $F = u_1 x^2 + u_2 y^3 + u_3 x z^2 + a y^2 z^2 + z^4 g$, $g = g_0(z^2) + y g_1(z^2) + y^2 g_2(z^2)$.
We may assume $u_i \equiv 1 \pmod{\idealm}$.
We eliminate $a$ and $g$ by replacing $y$ and $x$ suitably, and then we transform $F$ as above.

Assume $p = 2$ and $F_2 = z^2$.
Let $\overline{F_3} = (F_3 \bmod (z)) \in k[[x,y]]$.
If $\overline{F_3}$ has three distinct roots then we may assume $\overline{F_3} = x^3 + y^3$ and then it is $D_4^0$.
We can transform $F$ to $z^2 + x^3 + y^3$ as above, and then to $z^2 + x^2 y + x y^2$.
If $\overline{F_3}$ has two distinct roots then we may assume $\overline{F_3} = x^2y$ and $F$ must have $xy^m$ and then it is $D_{2m}^0$.
We obtain $F = u_1 z^2 + u_2 x^2 y + u_3 x y^m + g(x) + f(y) + y^{2m-1} g'(y)$,
where $f(y) = f_0(y)^2 + y f_1(y)^2$ is a polynomial of degree $< 2m-1$
and $g \in k[[x]]$ and $g' \in k[[y]]$.
We argue as in the case of $D_{2m+1}^0$.
If $\overline{F_3}$ has one (triple) root then we may assume $\overline{F_3} = x^3$ and $F$ must have $xy^3$ or $y^5$ and then it is $E_7^0$ or $E_8^0$.
We transform $F$ as above.
\end{proof}

Next we consider symplectic actions on fixed points.
\begin{lem} \label{lem:mu_p RDP}
Assume $w$ is fixed, and the action is symplectic at $w$.
\begin{enumerate}
\item \label{lem:mu_p RDP:fixed smooth symplectic}
Assume $w$ is a smooth point.
Then $w$ is an isolated fixed point and
$\pi(w)$ is an RDP of type $A_{p-1}$.
The eigenvalues of $D$ on the cotangent space $\idealm_w / \idealm_w^2$ are of the form $a,-a$ for some $a \in \bF_p^*$.
\item \label{lem:mu_p RDP:fixed RDP symplectic}
Assume $w$ is an RDP.
Let $\map{f}{X' = \Bl_w X}{X}$.
Then $X'$ is an RDP surface, 
$D$ uniquely extends to a derivation $D'$ on $X'$ which is symplectic at every fixed point above $w$,
and $g \colon Y' = (X')^{D'} \to Y$ is crepant.
\end{enumerate}
\end{lem}

\begin{proof}%[Proof of Lemma \ref{lem:mu_p RDP}]
(\ref{lem:mu_p RDP:fixed smooth symplectic})
By Lemma \ref{lem:mu_n RDP}(\ref{item:mu_n RDP:smooth}),
we have $D = a x \cdot (\partial / \partial x) - a y \cdot (\partial / \partial y)$ with $a \in \bF_p$ 
for some coordinates $x,y$,
and $a \neq 0$ since $D$ is nontrivial. 
Hence $w$ is an isolated fixed point of $D$.
We observe that $a,-a$ are the eigenvalues of the action on the cotangent space.
We have $\hat{\cO}_{X,w}^D = k[[x^p, xy, y^p]]$ and it is an RDP of type $A_{p-1}$.

(\ref{lem:mu_p RDP:fixed RDP symplectic})
By Remark \ref{rem:D=i} and assertion (\ref{lem:mu_p RDP:fixed smooth symplectic}), $w$ is an isolated fixed point.
By Proposition \ref{prop:fixedpoints}, $D$ uniquely extends to $D'$ on $X'$.
Let $\omega$ be a generator of $H^0(\Spec \cO_{X,w} \setminus \set{w}, \Omega^2)$ with $D(\omega) = 0$.
Since $w$ is an RDP, $X'$ is again an RDP surface,
and it follows from Proposition \ref{prop:2-forms of RDP} that $\omega$ extends to $\omega'$ on $(X')^{\sm}$
which generates $H^0(\Spec \cO_{X',w'} \setminus \set{w'}, \Omega^2)$ at any closed point $w' \in X'$ above $w$,
and that $D'(\omega') = 0$.
Hence $D'$ is symplectic at every fixed point above $w$.
Since as above such fixed points are isolated,
$Y'$ is smooth outside finitely many isolated points.
Applying Proposition \ref{prop:2-forms} to $\omega$ on $X \setminus \set{w}$ and $\omega'$ on $X' \setminus (\Sing(X') \cup \Fix(D'))$,
we obtain $2$-forms $\psi$ on $Y \setminus \set{\pi(w)}$ and $\psi'$ on $Y' \setminus \pi((\Sing(X') \cup \Fix(D')))$,
which are non-vanishing.
Comparing $\psi$ and $\psi'$ we observe that $g$ is crepant.
\end{proof}

\begin{proof}[Proof of Theorem \ref{thm:mu_p RDP}(\ref{thm:mu_p RDP:fixed RDP symplectic})]
By Remark \ref{rem:D=i} and Lemma \ref{lem:mu_p RDP}(\ref{lem:mu_p RDP:fixed smooth symplectic}),
$w$ is an isolated fixed point.
By shrinking $X$ we may assume that $D$ has no fixed point except $w$.

We construct a finite sequence $(X_j,D_j)_{0 \leq j \leq n}$ ($n \geq 0$)
of RDP surfaces $X_j$ and derivations $D_j$ on $X_j$ of multiplicative type
that is symplectic at each fixed point.
Let $(X_0,D_0) = (X,D)$.
If $X_j$ has no fixed RDP then we terminate the sequence at $n = j$.
If $X_j$ has at least one fixed RDP,
let $X_{j+1}$ be the blow-up of $X_j$ at the fixed RDPs
and $D_{j+1}$ the extension of $D_j$ to $X_{j+1}$.
Since any RDP becomes smooth after a finite number of blow-ups at RDPs,
this sequence terminates at some $n \geq 0$.
By Lemma \ref{lem:mu_p RDP}(\ref{lem:mu_p RDP:fixed RDP symplectic}),
$D_{j+1}$ on $X_{j+1}$ is symplectic at each fixed point,
and $(X_{j+1})^{D_{j+1}} \to (X_j)^{D_j}$ is crepant.
By Theorem \ref{thm:mu_p RDP}(\ref{thm:mu_p RDP:non-fixed RDP}) and Lemma \ref{lem:mu_p RDP}(\ref{lem:mu_p RDP:fixed smooth symplectic}),
$Y_n = (X_n)^{D_n}$ has canonical singularity (i.e.\ has no singularity other than RDPs),  
and since $Y_n \to Y = X^D$ is crepant,
also $Y$ has canonical singularity.
If $n > 0$ then $\pi(w)$ is not a smooth point since $Y_n \to Y$ is a crepant morphism non-isomorphic at that point,
and if $n = 0$ then $\pi(w)$ is not a smooth point by Lemma \ref{lem:mu_p RDP}(\ref{lem:mu_p RDP:fixed smooth symplectic}).
Hence in either case $\pi(w)$ is an RDP.
\end{proof}

Moreover, we can classify all possible symplectic $\mu_p$-actions on RDPs.
\begin{prop} \label{prop:list of fixed symplectic RDP}
Assume $w$ is a fixed RDP, and the action is symplectic at $w$.
Then there is a $\mu_p$-equivariant isomorphism $\hat{\cO}_{X,w} \cong k[[x,y,z]]/(F)$
with $F$ equal to one in Table \ref{table:fixed symplectic RDP}
and $\mu_p$ acts on $x,y,z$ by respective weights $a,-a,0$ for some $a \in \bF_p^*$.
The singularities of $X$, $X' = \Bl_w X$, $X/\mu_p$, and $X'/\mu_p$ are displayed in the table.
\end{prop}

\begin{table} 
\caption{Symplectic $\mu_p$-actions on RDPs} \label{table:fixed symplectic RDP} 
\begin{tabular}{lllll}
\toprule
$p$ & equation                     &              & $X$       & $X/\mu_p$ \\ 
\midrule
any & --- (smooth point)           &              & $A_0$     & $A_{p-1}$   \\ 
\midrule
any & $xy + z^m$                   & ($m \geq 2$) & $A_{m-1}$ & $A_{mp-1}$  \\ 
\midrule
$3$ & $z^2 + x^3 + y^3 + x^2y^2$   &              & $E_6^1$   & $E_6^1$     \\ 
$3$ & $z^2 + x^3 + y^3 + x^4y$     &              & $E_8^1$   & $E_6^0$     \\ 
\midrule
$2$ & $x^2 + y^2z + xyz^k + z^l $  & ($k \geq 1$, $l \geq 2$) & $D_{2k+l}^{\floor{l/2}}$ & $D_{2l}^{\positive{l-k}}$   \\
$2$ & $x^2 + z^2 + xyz + y^{2n-2}$ & ($n \geq 3$)             & $D_{2n}  ^{n-1}\another$ & $D_{n+2} ^{\floor{n/2}}   $ \\ 
$2$ & $x^2 + z^2 + z^3 + xy^3$     &                          & $E_7^2$                  & $D_5^0$ \\ 
$2$ & $x^2 + z^3 + y^4 + xyz$      &                          & $E_7^3$                  & $E_7^3$ \\ 
$2$ & $x^2 + z^3 + y^4 + xy^3$     &                          & $E_8^3$                  & $E_7^2$ \\ 
\bottomrule
\end{tabular}

\begin{tabular}{lllll}
\toprule
$p$ & $X$       & $X'$                    & $X/\mu_p$ & $X'/\mu_p$ \\ 
\midrule
any & $A_0$     & ---                     & $A_{p-1}$  & --- \\ 
\midrule
any & $A_{m-1}$ ($m \geq 3$) & $A_{m-3} + 2 A_0$       & $A_{mp-1}$ & $A_{(m-2)p-1} + 2 A_{p-1}$ \\ 
any & $A_{1}  $ & $          2 A_0$       & $A_{2p-1}$ & $               2 A_{p-1}$ \\ 
\midrule
$3$ & $E_6^1$   & $A_5\notfixed + 2A_0$   & $E_6^1$    & $A_1 + 2A_2$ \\ 
$3$ & $E_8^1$   & $E_7^{0}\notfixed + 2A_0$ & $E_6^0$    & $A_1 + 2A_2$ \\ 
\midrule
$2$ & $D_{2k+l}^{\floor{l/2}}$  ($l \geq 4$)          & $D_{2k+l-2}^{\floor{l/2}-1} + A_1$ & $D_{2l}^{\positive{l-k}}$ & $D_{2(l-2)}^{\positive{l-2-k}} + A_3$ \\
    &             \quad         ($l = 3$, $k \geq 2$) & $D_{2k+1}^0 \notfixed + A_0 + A_1$ & $D_{ 6}^{\positive{3-k}}$ & $A_1 + A_1 + A_3$ \\
    &             \quad         ($l = 3$, $k = 1$)    & $A_3        \notfixed + A_0 + A_1$ & $D_{ 6}^{           {2}}$ & $A_1 + A_1 + A_3$ \\
    &             \quad         ($l = 2$, $k \geq 2$) & $D_{2k}^0   \notfixed       + A_1$ & $D_{ 4}^{           {0}}$ & $            A_3$ \\
    &             \quad         ($l = 2$, $k = 1$)    & $2 A_1      \notfixed       + A_1$ & $D_{ 4}^{           {1}}$ & $            A_3$ \\
$2$ & $D_{2n}  ^{n-1}\another$  ($n \geq 3$)     & $D_{2n-2}^{n-2}\another + A_1\notfixed$ & $D_{n+2} ^{\floor{n/2}} $ & $D_{n+1}^{\floor{(n-1)/2}}$ \\ 
$2$ & $E_7^2$                   & $D_6^1              $                     & $D_5^0$ & $D_4^0$ \\ 
$2$ & $E_7^3$                   & $D_6^2\another + A_0$                     & $E_7^3$ & $D_5^1 + A_1$ \\ 
$2$ & $E_8^3$                   & $E_7^2 + A_0        $                     & $E_7^2$ & $D_5^0 + A_1$ \\ 
\bottomrule
\end{tabular}
\begin{itemize}
\item $A_0$ is a smooth point that is an isolated fixed point of $D$.
\item $\notfixed$ means that the RDP is not fixed by $D$.
\item $\floor{q} := \max\set{n \in \bZ \mid n \leq q}$ denotes the integer part of a real $q$.
\item $q^+ := \max\set{q,0}$ denotes the positive part of a real $q$.
\item $\another$: It follows from the classification that for each (formal) isomorphism class of RDP there exists only one fixed symplectic $\mu_p$-action up to isomorphism,
except for the case of $D_{2n}^{n-1}$ ($n \geq 3$) in $p = 2$,
in which case there are two and they are distinguished by the degree $2$ part $F_2$ being a square of a homogeneous element or not.
We distinguish them by notation $D_{2n}^{n-1}$ and $D_{2n}^{n-1}\another$.
We use the convention that $D_4^{1}\another = D_4^{1}$.
\end{itemize}
\end{table}

\begin{rem}
A polynomial $f \in k[x_1, \dots, x_m]$ is called \emph{quasi-homogeneous}
if, for some $a_1, \dots, a_m \in \bZ_{\geq 1}$, 
the monomials appearing in $f$ have the same degree with respect to $a$ (i.e.\ degree of the monomial $x_1^{i_1} \dots x_m^{i_m}$ is $i_1 a_1 + \dots + i_m a_m$).
RDPs whose completions are \emph{not} defined by quasi-homogeneous polynomials, 
which exist only if $p = 2,3,5$,
are precisely $D_n^r$ and $E_n^r$ with $r \neq 0$.
It follows from the classification given in Proposition \ref{prop:list of fixed symplectic RDP} 
(resp.\ given in Proposition \ref{prop:non-fixed mu_p RDP}, resp.\ which is omitted) that
if an RDP of type $D_n$ or $E_n$ admits a fixed symplectic (resp.\ non-fixed, resp.\ fixed non-symplectic) $\mu_p$-action
then the singularity is not defined (resp.\ is defined, resp.\ is defined) by a quasi-homogeneous polynomial.
We do not know any explanation of this phenomenon.
\end{rem}

\begin{proof}[Proof of Proposition \ref{prop:list of fixed symplectic RDP}]
We consider tuples $(a,b,c,F)$ with $a,b,c \in \bF_p$, not all $0$, 
and $F \in k[[x,y,z]]$ such that
$F = 0$ defines an RDP
and only monomials of weight $a+b+c$ ($\in \bF_p$) appear in $F$, where $x,y,z$ have respective weights $a,b,c$.
By Lemma \ref{lem:mu_n RDP}(\ref{item:mu_n RDP:RDP}), it suffices to consider $k[[x,y,z]]/(F)$ of this form.
We will show that there exist a $\mu_p$-equivariant isomorphism
$k[[x,y,z]]/(F) \cong k[[x',y',z']]/(F')$ 
with $F'$ in Table \ref{table:fixed symplectic RDP} and $\wt(x',y',z') = (1,-1,0)$ 
up to $\Aut(\mu_p) = \bF_p^*$ (which amounts to replacing $(a,b,c)$ with $(ta,tb,tc)$ for some $t \in \bF_p^*$).

We write $F = \sum_{h,i,j} a_{hij} x^h y^i z^j$,
and we say that a polynomial or a formal power series \emph{has} a monomial if its coefficient is nonzero.

First assume the degree $2$ part $F_2$ is a non-square.
Then we may assume that $F_2$ contains a non-square monomial, say $xy$
(indeed, if this is not the case then $p \neq 2$ and $F_2$ contains at least two square monomials, say $x^2$ and $y^2$,
then $x$ and $y$ has the same weight, and then after a linear coordinate change we may assume $F_2$ contains $xy$).
Then we have $c = 0$.
If $a + b \not\in \set{0,a,b}$ then $F \in (x,y)^2$, which implies that $F = 0$ is not an RDP. 
If $a + b = a \neq 0$ then $F \in (x)$, again not an RDP.
Same if $a + b = b \neq 0$.
So we have $a + b = 0$ and hence $F \in k[[x^p,xy,y^p,z]]$.
Since $F$ cannot belong to $(x,y)$, there exists an integer $m$ such that $F$ has the monomial $z^{m}$.
Let $m$ be the smallest such integer.
We have $F = u_1 z^m + u_2 xy + g_1(x^p) + g_2(y^p)$ for some units $u_1,u_2 \in k[[x^p, x y, y^p, z]]^*$ and power series $g_1, g_2$.
We may assume $u_1,u_2 \equiv 1 \pmod{\idealm}$.
We eliminate $g_1, g_2$ by replacing $x$ with $x + g_2/y$ and $y$ with $y + g_1/x$ (and repeating this), and we obtain $F = u_1 z^m + u_2 xy$.
By replacing $x,y,z,F$ with suitable multiples we obtain $F = z^m + xy$.

Next assume $p \geq 3$ and $F_2$ is square. We may assume $F_2 = z^2$.
We may assume $F_3 \not\equiv 0 \pmod{z}$.
If $F_3$ has $x^2y$
then by $2c = 2a + b = a + b + c$ we have $b = 0$ and $a = c$, hence $F \in (x,z)^2$, which is absurd.
Hence we may assume $F_3$ has $y^3$. 
By $2c = 3b = a+b+c$ we have $(a,b,c) = (a,2a,3a)$. 
If $F$ does not have $x^3$, then $F \in (z^2, x^3z, xyz, x^6, x^4y, x^2y^2, y^3)$,
and $F = 0$ cannot define an RDP.
Hence $F$ has $x^3$, hence $p = 3$ and then $F \in k[[x^3,xy,y^3,z]]$.
We may assume $F$ does not have $xyz$.
To define an RDP, $F$ must have one of $x^2y^2, x^4y, xy^4$.

If it has $x^2y^2$, then it is $E_6^1$.
We can eliminate $x^h y^i z$ and we have 
\begin{gather*} %notes p106
F = z^2 u_1 + x^3 + y^3 + x^2 y^2 u_2 
+ \sum_{(h,i,j) \in S_1} a_{hij} x^h y^i z^j 
+ \sum_{(h,i) \in S_2} b_{hi} x^h y^i + \sum_{(h,i) \in S_3} c_{hi} x^h y^i, \\
S_1 = \set{(4,1,0), (1,4,0)}, \qquad
S_2 = \set{(6,0), (7,1)}, \qquad
S_3 = \set{(0,6), (1,7)}, 
\end{gather*}
where 
$a_{hij} \in k$,
$b_{hi} \in k[[x^3]]$,
$c_{hi} \in k[[y^3]]$,
and $u_1, u_2 \in k[[x^3, x y, y^3, z]]^*$.
By replacing $x$ with $x + t y^2$ and $y$ with $y + t' x^2$, we eliminate $a_{410}$ and $a_{140}$.
Then by replacing $F$ with $(1 + x^3 b_{60} + x^4 y b_{71} + y^3 c_{06} + x y^4 c_{17}) F$ we eliminate all $b_{hi}$ and $c_{hi}$. Finally we replace $x,y,z,F$ with suitable multiples and achieve $u_1 = u_2 = 1$
(for example, we let $F = u_2^{-3} F'$, $x = u_2^{-1} x'$, $y = u_2^{-1} y'$, $z = (u_1 u_2^3)^{-1/2} z'$).

If it does not have $x^2 y^2$ but has $x^4 y$ or $x y^4$, then it is $E_8^1$.
By replacing $F$ with a unit multiple we may assume it has $x^4 y$ and does not have $x y^4$.
We can eliminate $x^h y^i z$ and we have 
\begin{gather*} %notes p106
F = z^2 u_1 + x^3 + y^3 + x^4 y u_2 
+ \sum_{(h,i,j) \in S_1} a_{hij} x^h y^i z^j 
+ \sum_{(h,i) \in S_2} b_{hi} x^h y^i + \sum_{(h,i) \in S_3} c_{hi} x^h y^i, \\
S_1 = \set{(6,0,0), (3,3,0), (0,6,0)}, \qquad
S_2 = \set{(9,0)}, \qquad
S_3 = \set{(0,9), (1,7), (2,5), (3,6)}, 
\end{gather*}
with $a_{hij}$, $b_{hi}$, $c_{hi}$, and $u_1,u_2$ as in the previous case.
By replacing $y$ with $y + t x^2$ we eliminate $a_{600}$.
By replacing $x$ with $x + t' y^2$ and $F$ with $(1 + t'' y^3) F$ we eliminate $a_{330}$ and $a_{060}$.
By replacing $F$ with $(1 + x^2 y^2 c_{25} + x y^4 c_{17} + y^6 c_{09}) F$, 
then with $(1 + x^6 b_{90} + x^3 y^3 c_{36}) F$,
we eliminate all $b_{hi}$ and $c_{hi}$.
We replace $x,y,z,F$ with suitable multiples and achieve $u_1 = u_2 = 1$.

Hereafter assume $p = 2$ and that $F_2$ is a square. 
If $\sqrt{F_2}$ is homogeneous,
then we may assume $F_2 = x^2$, 
we may assume $F_3$ has $y^2 z$ or $z^3$,
and then we have $(a,b,c) = (1,1,0)$, hence $F \in k[[x^2,xy,y^2,z]]$.
If $\sqrt{F_2}$ is not homogeneous,
then we may assume $F_2 = x^2 + z^2$ and $a = 1$ and $c = 0$, 
and then again we have $(a,b,c) = (1,1,0)$, hence $F \in k[[x^2,xy,y^2,z]]$,
and $F_3$ has $x y z$, $y^2 z$, $z^3$, or $x^2 z$.

Assume ($F_2$ is $x^2$ or $x^2 + z^2$ and) $F_3$ contains $y^2 z$.
Also, since $F \notin (x,y)^2$, we see that $F$ has $z^l$ ($l \geq 2$).
We have 
\[
F \equiv x^2 + y^2 z + z^l \pmod{(x^4, x^3 y, x^2 y^2, x y^3, y^4, x^2 z, x y z, y^2 z^2, z^{l+1})}.
\] 
Write $F = F_0(x^2, y^2, z) + xy F_1(x^2, y^2, z)$.
Then there exist unique $f, g \in k[[z]]$ such that $F_0 \in (x^2 - f(z)^2, y^2 - g(z)^2)$,
and they satisfy 
\[
l = \min\set{2 \ord_z(f), 2 \ord_z(g) + 1}.
\]
If $l$ is even then, by replacing $y$ with $y - x g/f$, we may assume $g = 0$.
If $l$ is odd then, by replacing $x$ with $x - y f/g$, we may assume $f = 0$.
We eliminate $a_{hij}$ with $h \geq 2$, $(h,i,j) \neq (2,0,0)$, by replacing $F$ with $(1 + a_{hij} x^{h-2} y^i z^j) F$,
and $a_{hij}$ with $i \geq 2$, $(h,i,j) \neq (0,2,0), (0,2,1)$, by replacing $z$ with $z + a_{hij} x^h y^{i-2} z^j$.
We obtain $F = x^2 + y^2 z + z^l u(z) + x y e(z)$, where $e(z) \in k[[z]]$ and $u(z) \in k[[z]]^*$.
We have $e(z) \neq 0$, since if $e(z) = 0$ then $F = F_0 \in ((x - f(z))^2, (y - g(z))^2)$, which is absurd.
Write $e(z) = z^k v(z)$, $k \geq 1$, $v(z) \in k[[z]]^*$.
It is $D_{2k+l}^{\floor{l/2}}$.
If $l$ is even then, since $F_0 \in (x^2 - f(z)^2, y^2 - g(z)^2)$ and $g = 0$,
we have $z^l u(z) = f(z)^2$ and hence $u(z)$ is a square, and then by replacing $x$ with $u(z)^{1/2} x$ and by replacing $F$ with a unit multiple we obtain 
$F = x^2 + y^2 z u'(z) + z^l + x y e(z)$ for some $u'(z) \in k[[z]]^*$.
Similarly, if $l$ is odd then, since $f = 0$,
 we have $z^l u(z) = z g(z)^2$ and hence $u(z)$ is a square, and then (by replacing $y$) we obtain
$F = x^2 u'(z) + y^2 z + z^l + x y e(z)$.
By replacing $x,y,z,F$ with unit multiples, we can achieve $u' = v = 1$.

Assume $F_2 = x^2$ and $F_3$ has $z^3$ but no $y^2z$.
To define an RDP $F$ must have $y^4$ and must have $xyz$ or $xy^3$.
If $F$ has $xyz$ then it is $E_7^3$.
We have 
\begin{gather*} % notes p65, p102, p104
F = x^2 + y^4 + z^3 u_1+ x y z u_2 + 
\sum_{(h,i,j) \in S_1} a_{hij} x^h y^i z^j \\
+ \sum_{(h,i) \in S_2} b_{hi} x^h y^i + \sum_{(h,i) \in S_3} c_{hi} x^h y^i + \sum_{(h,i) \in S_4} d_{hi} x^h y^i, \\
S_1 = \set{(3,1,0), (1,3,0), (1,5,0), (0,2,2), (2,0,1), (2,0,2), (0,4,1), (0,4,2)}, \\
S_2 = \set{(4,0), (5,1)}, \qquad
S_3 = \set{(0,6), (1,7)}, \qquad
S_4 = \set{(2,2)},
\end{gather*}
where $a_{hij} \in k$, 
$b_{hi} = \sum_{j = 0}^{2} b_{hij} z^j$ with
$b_{hij} \in k[[x^2]]$, 
$c_{hi} = \sum_{j = 0}^{2} c_{hij} z^j$ with 
$c_{hij} \in k[[y^2]]$, 
$d_{hi} \in k[[x^2, x y, y^2]]$,
and $u_1,u_2 \in k[[x^2, x y, y^2, z]]^*$.
We may assume $u_1, u_2 \equiv 1 \pmod{\idealm}$.
We replace $z$ with $z + a_{310} x^2 + a_{130} y^2 + a_{150} y^4$,
$x$ with $x + t y z$ (which eliminates $a_{022}$),
$y$ with $y + a_{201} x + a_{202} x z$, 
and
$x$ with $x + a_{041} y^3 + a_{042} y^3 z$,
thus eliminate all $a_{hij}$ ($(h,i,j) \in S_1$).
We replace 
$F$ with $(1 + x^2 b_{40} + y^2 c_{06}) F$,
$F$ with $(1 + x^3 y b_{51} + x y^3 c_{17}) F$,
and 
$z$ with $z + d_{22} x y$,
thus eliminate all $b_{hi}$, $c_{hi}$, and $d_{hi}$.
We replace $x,y,z,F$ with suitable multiples and achieve $u_1= u_2 = 1$.

Next, if $F$ does not have $xyz$ but has $x y^3$ then it is $E_8^3$.
We have 
\begin{gather*} % notes p66, p103, p105
F = x^2 + y^4 + z^3 u_1 + x y^3 u_2 + 
\sum_{(h,i,j) \in S_1} a_{hij} x^h y^i z^j 
+ \sum_{(h,i) \in S_2} b_{hi} x^h y^i + \sum_{(h,i) \in S_3} c_{hi} x^h y^i, \\
S_1 = \set{(2,0,1), (2,0,2), (0,4,1), (0,2,2), (1,1,2), (0,6,0), (0,4,2), (0,6,1), (0,6,2), \\ (2,2,0), (2,2,1), (2,2,2)}, \qquad
S_2 = \set{(4,0), (3,1), (4,2))}, \qquad
S_3 = \set{(0,8)}, 
\end{gather*}
with $a_{hij}$, $b_{hi}$, $c_{hi}$, and $u_1,u_2$ as in the previous case.
We may assume $u_1, u_2 \equiv 1 \pmod{\idealm}$.
We replace 
$F$ with $(1 + a_{201} z + a_{202} z^2) F$,
$x$ with $x + t y z$ and 
$z$ with $z + t' y^2$ (which eliminates $a_{041}$ and $a_{022}$),
$z$ with $z + a_{112} x y$, 
$x$ with $x + t'' y^3$ (which eliminates $a_{060}$),
$x$ with $x + a_{042} y z^2 + a_{061} y^3 z$,
$x$ with $x + a_{062} y^3 z^2$, 
and
$y$ with $y + (a_{220} + a_{221} z + a_{222} z^2) x$, 
thus eliminate all $a_{hij}$ ($(h,i,j) \in S_1$).
We replace 
$F$ with $(1 + x^2 b_{40} + x y b_{31} + x^2 y^2 b_{42} + y^4 c_{08}) F$,
thus eliminate all $b_{hi}$ and $c_{hi}$.
We replace $x,y,z,F$ with suitable multiples and achieve $u_1 = u_2 = 1$.

Assume $F_2 = x^2 + z^2$ and $F_3$ does not have $y^2 z$ and has $xyz$.
$F$ moreover needs $x y^i$, $y^i z$, or $y^l$.
Replacing $z$ with $z + xy^{i-1}$ (resp.\ $x$ with $x + y^{i-1}z$) we may assume there are no $x y^i$ (resp.\ $y^i z$) of low degree.
Thus we have 
\begin{gather*}
F \equiv x^2 + z^2 + xyz + y^{2n-2} \\
\pmod{(x^4, x^3 y, x^2 y^2, x y^{m} (m > n), y^{2n}, x^2 z, x y z^2, y^2 z^2, y^{m} z (m > n), z^3)}, 
\end{gather*}
 $n \geq 3$, and this is $D_{2n}^{n-1}\another$. 
We eliminate $x^h y^i$ ($h ,i \geq 1$, $(h, i) \neq (1, 1)$) by replacing $z$ with $z + a_{hi0} x^{h-1} y^{i-1}$, and $x^h z^j$ and $y^i z^j$ similarly.
We eliminate $x^h$, $y^i$, $z^j$ ($h \geq 3$, $i \geq 2n-1$, $j \geq 3$) by replacing $F$ with a unit multiple.
We obtain $F = x^2 + z^2 + xyz u + y^{2n-2}$ for some $u \in k[[x^2, x y, y^2, z]]^*$, and we can achieve $u = 1$.

Assume $F_2 = x^2 + z^2$ and $F_3$ does not have $y^2 z$ nor $xyz$.
By replacing $F$ with $(1 + a_{201} x^2 z)^{-1} F$ we may assume $F$ does not have $x^2 z$.
Then $F$ has $z^3$ and $F$ moreover needs $xy^3$, and then it is $E_7^2$.
We have 
\begin{gather*} %notes p84, p105, p112
F = x^2 + z^2 + z^3 u_1 + x y^3 u_2 + 
\sum_{(h,i,j) \in S_1} a_{hij} x^h y^i z^j 
+ \sum_{(h,i) \in S_2} b_{hi} x^h y^i + \sum_{(h,i) \in S_3} c_{hi} x^h y^i, \\
S_1 = \set{(0,4,0), (0,4,1), (0,6,0), (0,2,2), (3,1,0), (2,2,0), (1,1,2), (4,0,0), (2,0,2) }, \\
S_2 = \set{(4,0), (3,1), (2,2)}, \qquad
S_3 = \set{(0,4)},
\end{gather*}
with $a_{hij}$, $b_{hi}$, $c_{hi}$, and $u_1,u_2$ as in the case of $E_7^3$,
and moreover $b_{hi} \in (x^2, z)$ and $c_{hi} \in (y^4, y^2 z, z^2)$.
We may assume $u_1, u_2 \equiv 1 \pmod{\idealm}$.
We replace 
$z$ with $z + a_{040}^{1/2} y^2$, 
$x$ with $x + a_{041} y z$, 
$x$ with $x + t y^3$ (which eliminates $a_{060}$),
$F$ with $(1 + a_{022} y^2) F$, 
$y$ with $y + t' x$ and 
$z$ with $z + t'' x y$ and 
$F$ with $(1 + t''' x y) F$ (which eliminates $a_{310}$, $a_{220}$, and $a_{112}$),
$z$ with $z + t'''' x^2$ and 
$F$ with $(1 + t''''' x^2) F$ (which eliminates $a_{400}$ and $a_{202}$), 
thus eliminate all $a_{hij}$ ($(h,i,j) \in S_1$).
We replace $x$ with $x + y c_{04}$ 
and $F$ with $(1 + b_{40} x^2 + b_{31} x y + b_{22} y^2) F$ to eliminate all $b_{hi}$ and $c_{hi}$.
We replace $x,y,z,F$ with suitable multiples and achieve $u_1 = u_2 = 1$.
\end{proof}

\subsection{\texorpdfstring{$\mu_n$-actions}{mu\_n-actions} on RDPs}

In this section we classify $\mu_n$-actions on RDPs under each of the following assumptions.
\begin{itemize}
\item $w$ is not fixed by $\mu_n$ (Proposition \ref{prop:non-fixed mu_n RDP}).
\item $n = p^e$, and the subgroup scheme $\mu_p$ fixes $w$ and is symplectic (Proposition \ref{prop:mu_p^e RDP}).
\item $w$ is fixed by $\mu_n$ and the action is symplectic (Proposition \ref{prop:symplectic mu_n RDP}).
\end{itemize}

In Propositions \ref{prop:mu_p^e RDP} and \ref{prop:symplectic mu_n RDP}) we use the convention that a smooth point is of type $A_0$.

\begin{lem} \label{lem:image of non-fixed point}
Let $X$ be a $k$-scheme equipped with a $\mu_{p^2}$-action.
Let $\pi_1 \colon X \to X_1 = X / \mu_p$ 
be the quotient morphism by the action of the subgroup scheme $\mu_p \subset \mu_{p^2}$.
If $w \in X$ is non-fixed by the action of $\mu_p$
then $\pi_1(w) \in X_1$ is non-fixed by the action of $\mu_{p^2}/\mu_p$.
\end{lem}
\begin{proof}
Let $\cO_{X,w} = B = \bigoplus_{i \in \bZ/p^2\bZ} B_i$ be the corresponding decomposition.
Since $w$ is non-fixed by $\mu_p$ there exists $y \in \idealm_{w} \subset B$ with $1 + y \in \bigoplus_{i \equiv 1 \pmod{p}} B_i$ (Lemma \ref{lem:good coord}).
Then,
since $y^p \in \idealm_{\pi_1(w)} \subset \cO_{X_1, \pi_1(w)}$ satisfies $1 + y^p = (1 + y)^p \in B_p$,
we conclude by Proposition \ref{prop:fixedpoints} that $\pi_1(w)$ is non-fixed by $\mu_{p^2}/\mu_p$.
\end{proof}

Suppose $X$ is a scheme equipped with a $\mu_n$-action, $n = p^e r$ with $p \notdivides r$, and $w \in X$ is a closed point fixed by $\mu_r \subset \mu_n$.
Let $f$ be the maximal integer with $0 \leq f \leq e$ such that the subgroup scheme $\mu_{p^f} \subset \mu_{p^e} \subset \mu_n$ fixes $w$.
We say that $\mu_{p^f r}$ is the stabilizer of $w$ and denote it by $\Stab(w)$.

\begin{prop} \label{prop:non-fixed mu_n RDP}
Let $\cO_{X,w}$, together with a $\mu_n$-action, be as in the beginning of Section \ref{sec:mu_p RDP},
and assume $w$ is an RDP.
Write $n = p^e r$ with $p \notdivides r$,
and $\Stab(w) = \mu_{p^f r}$. 
Suppose $\Stab(w) \subsetneq \mu_n$ (hence $f < e$, in particular $e \geq 1$).
Then there exist $x,y,z \in \idealm$ generating $\idealm$,
with $x,y,1+z$ homogeneous
such that, up to replacing $r$ with a multiple
and up to $\Aut(\mu_n) = (\bZ/n\bZ)^*$,
the weights and the type of singularity are 
as in Table \ref{table:non-fixed mu_n RDP}.
\end{prop}

\begin{table} 
\caption{Non-fixed $\mu_n$-actions on RDPs (listed in order of appearance in the proof)} \label{table:non-fixed mu_n RDP} 
\begin{tabular}{lllllll}
\toprule
$p$      & $n$      & $p^e$       & $p^f$     & $\wt(x,y,1+z)$ & equation                           & RDP \\
   \midrule
any      & $p^e r$  & $p^e$       & $p^f$     & $1,-1,p^f r$   & $xy + z^{p^{e-f}m}$                & $A_{p^{e-f}m - 1}$ \\
$\neq 2$ & $2 p^e$  & $p^e$       & $1$       & $0,p^e,2$      & $x^2 + y^2 + z^{p^{e}m}$           & $A_{p^{e}m - 1}$ \\
   \midrule
$5$      & $30$     & $5$         & $1$       & $15,10,6$      & $x^2 + y^3 + z^5$                  & $E_8^0$      \\
   \midrule
$3$      & $12$     & $3$         & $1$       & $3,6,4$        & $x^2 + y^3 + yz^3$                 & $E_7^0$      \\
$3$      & $12$     & $3$         & $1$       & $6k,3,4$       & $x^2 + z^3 + y^4$                  & $E_6^0$      \\
$3$      & $ 6$     & $3$         & $1$       & $3,3k,2$      & $x^2 + z^3 + y^4$                  & $E_6^0$      \\
$3$      & $30$     & $3$         & $1$       & $15,6,10$      & $x^2 + z^3 + y^5$                  & $E_8^0$      \\
$3$      & $18$     & $9$         & $3$       & $3,2,6$        & $x^2 + y^3 + z^3 (1+z)$            & $E_6^0$      \\
   \midrule
$2$      & $4m-2$   & $2$         & $1$       & $2m,2,2m-1$    & $x^2 + y z^2 + x y^m$              & $D_{2m+1}^0$ \\
$2$      & $4(m-1)$ & (*)         & $2^{e-1}$ & $1,2,2(m-1)$   & $x^2 + y z^2 + y^m (1+z)$          & $D_{2m}^0$   \\
$2$      & $ 6$     & $2$         & $1$       & $0,2,3$        & $x^2 + x z^2 + y^3$                & $E_6^0$      \\
$2$      & $24$     & $8$         & $2$       & $3,2,6$        & $x^2 + y^3 + z^4 (1+z)$            & $E_8^0$      \\
 $2$      & $12$     & $4$         & $2$       & $3,2,6$        & $x^2 + y^3 + z^4 (1+z)$            & $E_8^0$      \\
  
$2$      & $4m-2$   & $2$         & $1$       & $2m-2,2,2m-1$  & $z^2 + x^2 y + x y^m$              & $D_{2m}^0$   \\
$2$      & $4m$     & (*)         & $2^{e-1}$ & $1,-2,2m$      & $z^2 + x^2 y + y^m (1+z)$          & $D_{2m+1}^0$ \\
$2$      & $ 6$     & $2$         & $1$       & $2,2k,3$       & $z^2 + x^3 + y^3$                  & $D_4^0$      \\
$2$      & $18$     & $2$         & $1$       & $6,4,9$        & $z^2 + x^3 + xy^3$                 & $E_7^0$      \\
$2$      & $30$     & $2$         & $1$       & $10,6,15$      & $z^2 + x^3 + y^5$                  & $E_8^0$      \\
      
$2$      & $12$     & $4$         & $2$       & $3,2,6$        & $x^2 + y^3 + z^2 (1+z)$            & $D_4^0$      \\
$2$      & $ 4$     & $4$         & $2$       & $1,0,2$        & {\scriptsize $x^2 (1+y) (1+z) + z^2 + y^{2m+1}$} & $D_{4m}^0$   \\
$2$      & $ 8$     & $8$         & $4$       & $2,1,4$        & $x^2 + z^2(1+z) + xy^2$            & $D_5^0$      \\
$2$      & $12$     & $4$         & $2$       & $3,1,6$        & $x^2 + z^2 (1+z) + xy^3$           & $E_7^2$      \\
$2$      & $20$     & $4$         & $2$       & $5,2,10$       & $x^2 + z^2 (1+z) + y^5$            & $E_8^0$      \\
  
$2$      & $2^e$ ($e \geq 2$)  & $2^e$       & $2$       & $0,-1,2$       & $x^2 + y^2 (1+z) + x z^{2^{e-1}m}$       & $D_{2^e m}^0$   \\
\bottomrule
\end{tabular}

(*): In the cases where $n = 4(m-1)$ or $n = 4m$, $p^e$ is the highest power of $p = 2$ dividing $n$.
\end{table}

In this case, we do not pursue the exact equation, 
and the equations in Table \ref{table:non-fixed mu_n RDP} are merely examples.

\begin{proof}
We first show that 
there exist elements $x,y,z \in \idealm$ generating $\idealm$ and satisfying $\wt(x,y,1+z) = (a,b,c)$ with $c = n/p^{e-f} = p^f r$.
Since the stabilizer of the $\mu_{n}$-action is $\mu_{p^f r} \subsetneq \mu_{n }$, 
there exist $i \in \bZ/n\bZ$ and a homogeneous element $t$ of weight $i$
such that $t \notin \idealm$
and $\gcd\set{i, n} = p^f r$.
We may assume $t \equiv 1 \pmod{\idealm}$, and then $z' := -1 + t \in \idealm$ satisfies $\wt(1 + z') = i$.
Take an integer $q$ such that $q i = p^f r$ (in $\bZ/n\bZ$), and then $1 + z := (1 + z')^q$ satisfies $z \in \idealm$ and $\wt(1 + z) = p^f r$.
Now take $x'^{(1)}, x'^{(2)}, x'^{(3)}$ generating $\idealm$.
We may assume each $x'^{(h)}$ is homogeneous with respect to $\mu_{p^f r}$,
that is, there exists $i^{(1)}, i^{(2)}, i^{(3)} \in \bZ/n\bZ$ such that 
$x'^{(h)} = \sum_{j = 0}^{p^{e-f}-1} x'^{(h,j)}$
with $\wt(x'^{(h,j)}) = i^{(h)} + j p^f r$.
Let $x^{(h)} := \sum_{j = 0}^{p^{e-f}-1} (1+z)^{-j} x'^{(h,j)}$.
Then $\wt(x^{(h)}) = i^{(h)}$ and, 
since $x^{(h)} \equiv x'^{(h)} \pmod{(z)}$,
the elements $x^{(1)}, x^{(2)}, x^{(3)}, z$ generate $\idealm$.
We can omit one of these $4$ elements, 
which cannot be $z$ since $\idealm$ is not generated by homogeneous elements.

In this proof, by a \emph{monomial} we mean a polynomial of the form $x^i y^j z^{p^{e-f} l} (1+z)^m$ with $0 \leq m < p^{e-f}$.
Any polynomial (resp.\ formal power series) is uniquely expressed as a finite (resp.\ possibly infinite) sum of monomials with $k$-coefficients,
and we say that a polynomial or a formal power series \emph{has} a monomial if its coefficient is nonzero.
Expressions such as $F = x^2 + y^3 (1+z) + \dots$ will indicate that $F$ has these monomials.
However, when we say of degree $m$ part $F_m$ of $F$, 
this is understood with respect to the usual monomials $x^i y^j z^l$.

Assume the degree $2$ part $F_2$ of $F$ is either 
irreducible or the product of two distinct homogeneous linear factors.
Then we may assume $F = xy (1+z)^i + \dots$, 
and $F$ must have $z^{m'} (1+z)^j$, 
we may assume $i = j = 0$ by replacing $y$ and $F$,
and it is $A_{m'-1}$ with $m' = p^{e-f} m$.

Assume $F_2$ is the product of two distinct non-homogeneous linear factors.
Then we have $p \neq 2$ and we may assume $F = x^2 (1+z)^i + y^2 (1+z)^j + z^{m'} (1+z)^k + \dots$ and 
$b \equiv a + r/2 \pmod{r}$.
We may assume $k = 0$ by replacing $F$ with $(1+z)^{-k} F$
and $i = j = 0$ by replacing $x$ and $y$ with $x (1+z)^{i(p^e+1)/2}$ and $y (1+z)^{j(p^e+1)/2}$.
We have $2a = 2b = 0$
and then we have $f = 0$ and $r = 2$ (otherwise $a,b,c$ cannot generate $\bZ/p^e r \bZ$).
It is $A_{m'-1}$ with $m' = p^e m$ and we may assume $(a,b) = (0,n/2)$.

Assume $p \geq 5$ and $F_2$ is a square.
We may assume $F_2 = x^2$, $F_3 = y^3$, $F = x^2 + y^3 + \dots$,4
and then $F$ must have $z^5 (1+z)^i$ and we have $p = 5$, it is $E_8^0$, 
and we may assume $i = 0$.
By $2a = 3b = 0$ we have $n \divides 30$ and we may assume $a = 15$, $b = 10$.

Assume $p = 3$ and $F_2$ is a square.
We may assume $F = x^2 + \dots$.
We may assume $F_3 \bmod (x)$ is either $y^3$, $z^3$, or $y^3 + z^3$.
If $F = x^2 + y^3 + \dots$ then $F$ must have $y z^3$ and it is $E_7^0$.
If $F = x^2 + z^3 + \dots$ then $F$ must have $y^4$ or $y^5$ and it is $E_6^0$ or $E_8^0$.
If $F = x^2 + y^3 + z^3 (1+z)^i + \dots$ then we may assume $i = 1$ 
(if $i = 0$ then by replacing $z$ with $y (1 + z)^l + z $ we reduce this case to the previous case)
and then it is $E_6^0$.

Assume $p = 2$ and $F_2$ is a square.
We may assume $F_2$ is $x^2$, $z^2$, $x^2 + z^2$, $x^2 + y^2$, or $x^2 + y^2 + z^2$.

Assume $F_2 = x^2$. We may assume $F = x^2 + \dots$.
Then $F$ must have $y z^2 (1+z)^i$ or $y^3 (1+z)^j$.
If $F$ has $y z^2 (1+z)^i$ then we may assume $i = 0$ and $F$ must have $x y^m (1+z)^i$ ($m \geq 2$) or $y^m (1+z)$ ($m \geq 3$).
In the former case we may assume $i = 0$ (by replacing $x$ with $x (1+z)^i$) and we have $D_{2m+1}^0$ with $(a,b,c) = (2m,2,2m-1)$, $n \divides 2(2m-1)$.
In the latter case we have $D_{2m}^0$ with $(a,b,c) = (1,2,2(m-1))$, $n \divides 4(m-1)$.
Now assume $F$ does not have $y z^2 (1+z)^i$ and has $y^3 (1+z)^j$. We may assume $j = 0$.
Then $F$ must have either $x z^2 (1+z)^i$ or $z^4 (1+z)^i$.
If $F$ has $x z^2 (1+z)^i$, then we may assume $i = 0$ and then
$F = x^2 + y^3 + x z^2 + \dots$ is $E_6^0$, and we have $e = 1$, $f = 0$, $a = 0$, $c = n/2$.
If $F$ does not have $x z^2 (1+z)^i$ and has $z^4 (1+z)^i$, then we may assume $i = 1$ and then
$F = x^2 + y^3 + z^4 (1+z) + \dots$ is $E_8^0$, and we have $e - f \leq 2$, $f = 1$, $(a,b,c) \equiv (1,2/3,2) \pmod{2^e}$.

Assume $F_2 = z^2$. 
We have $e - f = 1$. 
Suppose $F$ has $x^2 y (1 + z)^i$. We may assume $i = 0$.
We may assume $F$ does not have monomials $z^2 M$ ($M \neq 1$) nor $x^2 M$ ($M \neq 1, y$) of low degree
(by replacing $F$ or $y$ with $(1 + M)^{-1} F$ or $y + M$ respectively).
$F$ must have $x y^m$ or $y^m (1+z)$ and then it is $D_{2m}^0$ or $D_{2m+1}^0$. 
Now suppose $F$ does not have $x^2 y (1 + z)^i$ nor $x y^2 (1 + z)^i$ and has $x^3 (1 + z)^i$. We may assume $i = 0$.
Then $F$ must have $y^3 (1 + z)^j$ ($D_4^0$), $x y^3 (1 + z)^j$ ($E_7^0$), 
or $y^5 (1 + z)^j$ ($E_8^0$). 

Assume $F_2 = x^2 + z^2$.
We may assume $F = x^2 + z^2 (1+z)^i + \dots$.
If $i = 0$ then by replacing $z$ with $z + x$ or $z + x (1+z)$
we reduce this case to the previous case.
Assume $i = 1$.
$F$ cannot have $x^3 (1+z)^i$ nor $x z^2 (1+z)^i$.
(If $F$ has $x^3 (1+z)^i$ then we have $2a = c = 3a + ic$ and this implies $(2i+1)c = 0$, contradicting $c = n/p^{e-f}$. Other cases are similar.)
We may assume $F$ does not have $y z^2 (1+z)^i$.
If $F$ has $y^3 (1+z)^i$ then $F$ does not have $x y^2 (1+z)^j$ and it is $D_4^0$.
If $F$ does not have $y^3 (1+z)^i$ and has $x^2 y (1+z)^i$, then $F$ cannot have $x y^2 (1+z)^j$,
and $F$ must have $y^{2m+1} (1+z)^j$, and it is $D_{4m}^0$. 
If $F$ does not have $y^3 (1+z)^i$ nor $x^2 y (1+z)^i$ and has $xy^2 (1+z)^i$ then it is $D_5^0$.
If $F$ does not have $y^3 (1+z)^i$ nor $xy^2 (1+z)^i$ nor $x^2 y (1+z)^i$,
then $F$ must have $xy^3 (1+z)^i$ or $y^5 (1+z)^i$, and (we may assume $i = 0$ and) it is $E_7^2$ or $E_8^0$.
(This $E_7^2$ is the only example of $D_n^r$ or $E_n^r$ with $r > 0$ in this proposition.)

Assume $F_2 = x^2 + y^2$.
Write $F = x^2 + y^2 (1+z)^j + \dots$.
If $j = 0$ then by replacing $x$ with $x + y (1+z)^k$ we reduce to the $F_2 = x^2$ case.
Suppose $j = 1$. 
We may assume $F$ does not have monomials $x^2 M$ ($M \neq (1+z)^i$) nor $y^2 M$ ($M \neq (1+z)^i$) of low degree
(by replacing $F$ or $z$ with $(1 + M)^{-1} F$ or $z + M$ respectively).
$F$ must have $x z^{2m} (1+z)^k$ or $y z^{2m} (1+z)^k$,
by symmetry we may assume $F$ has $x z^{2m} (1+z)^k$, we may assume $k = 0$, 
and then it is $D_{4m}^0$.

Assume $F_2 = x^2 + y^2 + z^2$.
Write $F = x^2 (1+z)^i + y^2 (1+z)^j + z^2 (1+z)^k + \dots$, $i,j,k \in \set{0,1}$.
If $i = j$ then we reduce this case to $F = x^2 + z^2$ case
by replacing $x$ with $x + y (1+z)^l$.
If $i \neq j$ then either $i = k$ or $j = k$ and then we reduce this case to $F = x^2 + z^2$ case
by replacing $z$ with $z + x (1+z)^l$ or $z + y (1+z)^l$.
\end{proof}

\begin{prop} \label{prop:mu_p^e RDP}
Let $\cO_{X,w}$, together with a $\mu_n$-action, be as in the beginning of Section \ref{sec:mu_p RDP}.
Suppose $n = p^e$ with $p > 0$ and $e \geq 2$.
Let $\Stab(w) = \mu_{p^f}$ ($0 \leq f \leq e$).
Suppose $f > 0$ and that the subgroup scheme $\mu_p \subset \Stab(w)$ acts symplectically.
Then one of the following is true.
\begin{itemize}
	\item $w$ is $A_{p^{e-f} m-1}$ for some integer $m \geq 1$.
	\item $w$ is $E_7^2$ and $(p^f, p^e) = (2, 4)$.
	\item $w$ is $D_{2n+1}^{n-1}$ ($n \geq 2$) or $E_8^3$, and $(p^f, p^e) = (4, 4)$.
\end{itemize}
\end{prop}

\begin{proof}
Let $\cO_{x,w} = B = \bigoplus_{i \in \bZ/p^2\bZ} B_i$ be the corresponding decomposition.

Assume $w$ is a smooth point.
Since $\mu_p$ acts symplectically,
the maximal ideal $\idealm$ is generated by two elements 
$x \in \bigoplus_{i \equiv a \pmod{p}} B_i$ and 
$y \in \bigoplus_{i \equiv b \pmod{p}} B_i$ 
for some $a,b \in \bZ/p^e\bZ$ with $a,b \not\equiv 0$ and $a + b \equiv 0 \pmod{p}$.
Since $a,b \not\equiv 0 \pmod{p}$, we may assume moreover $x \in B_a$ and $y \in B_b$.
Then $w$ is fixed by the whole group scheme $\mu_{p^e}$ and hence $e = f$.
This case is done
(with $m = 1$: recall the convention that a smooth point is of type $A_0$).

Hereafter we assume $w$ is an RDP.
Let $\varepsilon = 0$ if $e = f$ and $\varepsilon = 1$ if $e > f$.
By arguing as in the beginning of the proof of Proposition \ref{prop:non-fixed mu_n RDP} 
and by using Proposition \ref{prop:list of fixed symplectic RDP},
$\idealm$ is generated by three elements $x,y,z$ with 
$x \in B_a$, $y \in B_b$, and $\varepsilon + z \in B_c$,
and we may assume $a \equiv -b \not\equiv 0 \pmod{p}$,
and if $e > f$ then we may moreover assume $c = p^f$.

If $e > f$ then, since $\mu_p$ acts symplectically, 
it follows (from the classification given in Proposition \ref{prop:non-fixed mu_n RDP})
that either the RDP is $A_{m'-1}$ and then
we may assume $F = xy + z^{m'} + \dots$
and hence $p^{e-f} \divides m'$,
or the RDP is $E_7^2$ and $(p^f, p^e) = (2, 4)$.

Hereafter assume $e = f$.
If the RDP is $A_{m-1}$ then there is nothing to prove.
The remaining possibilities are given in Table \ref{table:fixed symplectic RDP} (Proposition \ref{prop:list of fixed symplectic RDP})
and in particular we have $p \leq 3$.

Assume $p = 3$ and the RDP is $E_6^1$ or $E_8^1$ (as in Proposition \ref{prop:list of fixed symplectic RDP}).
We may assume that $F = z^2 + x^3 + y^3 + \dots$ with $\wt(x,y,z) \equiv (1,-1,0) \pmod{3}$.
Then $F$ cannot be homogeneous since $\wt(x^3) \not\equiv \wt(y^3) \pmod{3^2}$.

Assume $p = 2$ and the RDP is $D_n$ or $E_n$.
By the classification in Proposition \ref{prop:list of fixed symplectic RDP},
we have $(a,b,c) \equiv (1,1,0) \pmod{2}$, 
and $F_2 \not\in k z^2$.
If $(a,b,c) \equiv (1,\pm 1,0) \pmod{4}$ then $F \in (x,y)^2$ and $F$ cannot define an RDP. 
If $(a,b,c) \equiv (1,1,2) \pmod{4}$ (then we may assume $F_2 = x^2$),
or if $(a,b,c) \equiv (1,-1,2) \pmod{4}$ and $F_2 = x^2$,
then $F \in (x^2,xyz,xy^3,z^3,z^2y^2,zy^4,y^6)$ and hence $F$ cannot define an RDP. 
Hence we may assume $(a,b,c) \equiv (1,-1,2) \pmod{4}$ and $F_2 = x^2 + y^2$, hence $p^e = 4$.
If $F$ has $xyz$, then $F$ must have $z^{2n-1}$ for some $n \geq 2$,
and then it is $D_{2n+1}^{n-1}$.
If $F$ does not have $xyz$, then $F$ must have $z^3$ 
and $F$ must also have $x^3 y$ or $x y^3$,
and then it is $E_8^3$.
\end{proof}

\begin{prop} \label{prop:symplectic mu_n RDP}
Let $p \geq 0$.
Let $\cO_{X,w}$, together with a $\mu_n$-action ($n > 1$), be as in the beginning of Section \ref{sec:mu_p RDP}.
Suppose the action fixes $w$ and is symplectic.
Then $p$, $n$, the type of singularity at $w$, and the quotient singularity are as in Table \ref{table:symplectic mu_n RDP}.
\end{prop}

\begin{proof}
If $p \notdivides n$ then this is Remark \ref{rem:tame action} (Table \ref{table:tame RDP}).
If $n = p$ then this is Proposition \ref{prop:list of fixed symplectic RDP}  (Table \ref{table:fixed symplectic RDP}). 
If $n = p^e$ with $e \geq 2$, then by Proposition \ref{prop:mu_p^e RDP} the possibilities are $D_{2n+1}^{n-1}$, $E_8^3$, and $A_{m-1}$ (with quotient $A_{mn-1}$).
In the other cases we conclude by comparing the tables of the tame case and the $n = p$ case.
For example, $E_6^r$ with $(p,n) = (3,6)$ is impossible 
since the $\mu_2$-quotient $E_7^r$ of $E_6^r$ does not admit a symplectic $\mu_3$-action.
\end{proof}

To a point $w$ fixed by a $\mu_n$-action, 
we define its multiplicity $m(w)$ inductively by 
\begin{itemize}
	\item 
	if $w$ is a smooth point then $m(w) = 1$, and
	\item
	if $w$ is an RDP then $m(w) = \sum_{w' \in \Fix(\mu_n \actson \Bl_w X)} m(w')$.
\end{itemize}
The multiplicity for each case is displayed in Table \ref{table:symplectic mu_n RDP}.
We observe that 
$m(\pi_r(w)) = r m(w)$ for any divisor $r < n$ of $n$, 
where $\map{\pi_r}{X}{X/\mu_r}$ is the quotient map
and $m(\pi_r(w))$ is the multiplicity with respect to the $\mu_{n/r}$-action on $X/\mu_r$.

\begin{table} 
\caption{Symplectic $\mu_n$-actions on RDPs} \label{table:symplectic mu_n RDP}
\begin{tabular}{llllll}
\toprule
$p$        & $n$ & $X$                    &                       & $X/\mu_n$                & multiplicity \\
\midrule  
any        & any & $A_{m-1}$              & ($m \geq 1$)          & $A_{mn-1}$               & $m$ \\
$\neq 2$   & $2$ & $A_{m-1}$              & ($m \geq 4$ even)     & $D_{m/2+2}$              & $2$ \\
$\neq 2$   & $4$ & $A_{m-1}$              & ($m \geq 3$ odd)      & $D_{m  +2}$              & $1$ \\
$\neq 2$   & $2$ & $D_{m+2}$              &                       & $D_{2m+2}$               & $m+1$ \\
$\neq 2,3$ & $3$ & $D_{4}$                &                       & $E_6$                    & $2$ \\
$2$        & $3$ & $D_4^r$                & ($r = 0,1$)           & $E_6^r$                  & $2$ \\
$\neq 2,3$ & $2$ & $E_6  $                &                       & $E_7$                    & $3$ \\
$3$        & $2$ & $E_6^r$                & ($r = 0,1$)           & $E_7^r$                  & $3$ \\
\midrule 
$3$        & $3$ & $E_6^1$                &                       & $E_6^1$                  & $2$ \\
$3$        & $3$ & $E_8^1$                &                       & $E_6^0$                  & $2$ \\
\midrule
$2$        & $2$ & $D_{2m}^r$             & ($1 \leq r \leq m-1$) & $D_{4r}^{(3r-m)^+}$      & $2r$ \\
$2$        & $2$ & $D_{2m}^{m-1}\another$ &                       & $D_{m+2}^{\floor{m/2}}$  & $2$ \\
$2$        & $2$ & $D_{2m+1}^r$           & ($1 \leq r \leq m-1$) & $D_{4r+2}^{(3r-m+1)^+}$  & $2r+1$ \\
$2$        & $2$ & $E_7^2$                &                       & $D_5^0$                  & $2$ \\
$2$        & $2$ & $E_7^3$                &                       & $E_7^3$                  & $3$ \\
$2$        & $2$ & $E_8^3$                &                       & $E_7^2$                  & $3$ \\
$2$        & $4$ & $D_{2m+1}^{m-1}$       &                       & $D_{2m+1}^{m-1}$         & $1$ \\
$2$        & $4$ & $E_8^3$                &                       & $D_5^0$                  & $1$ \\
\bottomrule
\end{tabular}
\end{table}

\subsection{\texorpdfstring{$\mu_n$-actions}{mu\_n-actions} on smooth points}

Let $B = \cO_{X,w}$, together with a $\mu_n$-action, be as in the beginning of Section \ref{sec:mu_p RDP},
and assume it is smooth.
Assume the $\mu_n$-action fixes $w$ and is symplectic at $w$.
	As shown in Lemma \ref{lem:mu_p RDP}(\ref{lem:mu_p RDP:fixed smooth symplectic}),
	there exists $j \in (\bZ/n\bZ)^*$, unique up to sign,
	such that the maximal ideal of $B$ is generated by two homogeneous elements 
	of respective weight $j$ and $-j$.
	(We say that the weights of the $\mu_n$-action on 
	the tangent space are $j$ and $-j$.)

	Now let $\tilde{Y}$ be the minimal resolution of $Y = X / \mu_n$ at $\pi(w)$,
	and let $\pi' \colon X' = X \times_Y \tilde{Y} \to \tilde{Y}$.
	Let $e_{k}$ ($k = 1, \dots, n-1$) be the exceptional curves of $\tilde{Y}$,
	ordered in a way that $e_{k} \cap e_{k'} \neq \emptyset$ if and only if $\abs{k - k'} \leq 1$.
	The $\mu_n$-action induces a decomposition $\pi'_* \cO_{X'} = \bigoplus_{i \in \bZ/n\bZ} (\pi'_* \cO_{X'})_i$. 
	Let $I_i := \Image( ((\pi'_* \cO_{X'})_i)^{\otimes n} \to \cO_{\tilde{Y}})$ for $i = 0, \dots, n-1$.
	Then $I_i$ are described as follows.
	(Clearly $I_0 = \cO_{X'}$.)

\begin{lem} \label{lem:A}
	$(\pi'_* \cO_{X'})_i$ and $I_i$ are invertible sheaves.
	After possibly reversing the ordering of the exceptional curves, we have an equality
	\[
	I_i = \cO_{\tilde{Y}} \Bigl( -\sum_k f_n \bigl( (j^{-1} i \bmod n), k \bigr) e_k \Bigr).
	\]
	for each $i = 1, \dots, n-1$.
\end{lem}
	Here $j^{-1} i \bmod n$ denotes the unique integer $h \in \set{0, 1, \dots, n-1}$ 
	satisfying $h j \equiv i \pmod{n}$,
	and the function $\map{f_n}{\set{1, 2, \dots, n-1}^2}{\bZ}$
	is defined by 
	\[
	f_n(h, k) = \begin{cases}
	h k & (k \leq n - h), \\
	(n - h) (n - k) & (k \geq n - h) .
	\end{cases}
	\]
\begin{proof}
	Straightforward.
\end{proof}

\section{Tame quotients of K3 surfaces and abelian surfaces} \label{sec:tame quotient}

The following fact should be known to experts
(for example, if $X$ is a K3 surface in characteristic $0$ and $G$ is symplectic and commutative, then this is a result of Nikulin).
We give a proof since we could not find a complete reference (covering all characteristics).

\begin{thm} \label{thm:quotient:tame symplectic}
Let $X$ be either an abelian surface or an RDP K3 surface in characteristic $p \geq 0$,
and $G$ a finite group of order not divisible by $p$ acting on $X$. 

If $X$ is an RDP K3 surface and $G$ is symplectic, then the quotient $X/G$ is an RDP K3 surface.

If $X$ is an abelian surface and $G$ is symplectic, then $X/G$ is either an abelian surface or an RDP K3 surface.

If $X$ is an RDP K3 surface and $G$ is non-symplectic, then $X/G$ is either an RDP Enriques surface or a rational surface.

If $X$ is an abelian surface and $G$ is non-symplectic, then $X/G$ is either an RDP Enriques surface, a (quasi-)hyperelliptic surface, a surface birational to a ruled surface, or a rational surface.
\end{thm}

\begin{proof}
By Proposition \ref{prop:tame}, we may assume $X$ is smooth.
Let $\map{\pi}{X}{X/G = Y}$ be the quotient morphism
and $\tilde{Y} \to Y$ the minimal resolution.

We have $b_1(\tilde{Y}) = b_1(Y) \leq b_1(X)$, 
where $b_i(X) = \dim \Het^i(X, \bQ_l)$ is the $i$-th $l$-adic Betti number for an auxiliary prime $l \neq \charac k$.
Hence if $X$ is a K3 surface (hence $b_1(X) = 0$)
then $\tilde{Y}$ cannot be abelian, (quasi-)hyperelliptic, nor non-rational ruled.

\medskip

First suppose $G$ acts non-symplectically.
If some nontrivial $g \in G$ satisfy $\dim \Fix(g) = 1$,
then by the usual ramification formula
$K_Y$ has negative coefficients at the corresponding divisors of $Y$,
hence $\kappa(\tilde{Y}) = -\infty$.
If $Y$ has a non-RDP singularity, 
then $K_{\tilde{Y}}$ has negative coefficients at the corresponding exceptional curves,
hence $\kappa(\tilde{Y}) = -\infty$.
In either case $\tilde{Y}$ is either ruled or rational.
Suppose neither is the case.
Then $K_Y$ is an RDP surface with numerically trivial $K_Y$.
Since we have 
\[ 
H^0(Y^{\sm}, \Omega_Y^2) = H^0(\pi^{-1}(Y^{\sm}), \Omega_X^2)^G = H^0(X, \Omega_X^2)^G = 0,
\]
$K_Y$ is not trivial.
Then by the classification of such surfaces (see Table \ref{table:kappa=0 surfaces}), 
$\tilde{Y}$ is either an Enriques surface or a (quasi-)hyperelliptic surface.
This settles the non-symplectic case.

\medskip

Now suppose $G$ acts symplectically.
By Proposition \ref{prop:tame}, 
$Y$ is an RDP surface with $K_Y$ trivial.
By the classification of surfaces with trivial canonical divisor,
$Y$ is an RDP K3 surface, an abelian surface, a non-classical RDP Enriques surface ($p = 2$), or a (quasi-)hyperelliptic surface ($p = 2,3$).
(Note that abelian and (quasi-)hyperelliptic surface admit no smooth rational curves.)

Since $Y$ has only RDP singularities, we have 
$h^1(\tilde{Y}, \cO_{\tilde{Y}}) = h^1(Y, \cO_Y)$.
Since $p \notdivides \card{G}$, 
$(\cO_X)^G \subset \cO_X$ is a direct summand
and hence we have $h^1(Y, \cO_Y) = h^1(Y, (\pi_* \cO_X)^G) \leq h^1(Y, \pi_* \cO_X) = h^1(X, \cO_X)$.
Therefore, if $X$ is a K3 surface then $Y$ cannot be a non-classical RDP Enriques surface.

It remains to show that if $X$ is an abelian surface then $Y$ cannot be a (quasi-)hyperelliptic surface.
(If $p \neq 2,3$, then this is clear since (quasi-)hyperelliptic surfaces always have nontrivial canonical divisor.)
If $X$ is an abelian surface and $Y$ is a (quasi-)hyperelliptic surface,
then since a (quasi-)hyperelliptic surface admit no smooth rational curves,
no element of $G \setminus \set{1}$ have fixed points.
It suffices to show that any fixed-point-free symplectic automorphism $g$ of an abelian surface $X$
of finite order not divisible by $p$ is a translation,
since the quotient of an abelian variety by a finite group of translations is an abelian variety.

Suppose $g$ is such an automorphism.
Endow $X$ with a group variety structure (i.e.\ choose an origin)
and write $g(x) = h(x) + a$ 
with $h$ an automorphism of the group variety (i.e.\ $h$ fixes the origin)
and $a$ a point.
Then $h$ is symplectic (since $g$ and the translation by $a$ are symplectic)
and of finite order dividing $\ord(g)$,
since $x = g^{\ord(g)}(x) = h^{\ord(g)}(x) + (h^{\ord(g)-1}(a) + \dots + a)$.
If $h = \id$ then $g$ is a translation.
Suppose $h \neq \id$. 
Then $h$ acts on the tangent space of each fixed point
via $\SL_2(k)$ (since $h$ is symplectic and of finite order not divisible by $p$),
hence $\Fix(h)$ is isolated.
Hence $h - \id$ has finite kernel and hence is surjective.
Let $x$ be a point with $h(x) - x = - a$.
Then $g(x) = x$, contradiction.
\end{proof}

\section{\texorpdfstring{$\mu_p$-quotients}{mu\_p-quotients} of RDP K3 surfaces and abelian surfaces} \label{sec:quotient}

The following theorems are the $\mu_p$-analogue of Theorem \ref{thm:quotient:tame symplectic}.

\begin{thm} \label{thm:quotient:symplectic}
The quotient of an RDP K3 surface by a symplectic $\mu_{p}$-action
is again an RDP K3 surface.
\end{thm}

\begin{thm} \label{thm:quotient:non-symplectic}
	The quotient $X/\mu_p$ of an RDP K3 surface by a non-symplectic action of $\mu_p$
	is either a rational surface (possibly with non-RDP singularities) or an RDP Enriques surface.
	The quotient is an RDP Enriques surface if and only if the action is fixed-point-free,
	and this can happen only if $p = 2$. 
	
\end{thm}

\begin{thm} \label{thm:quotient:abelian}
A $\mu_{p^e}$-action on an abelian variety
is always symplectic, 
in the sense that the $1$-dimensional space of top differential forms is of weight $0$,
and the quotient is again an abelian surface.
\end{thm}

\begin{proof}[Proof of Theorem \ref{thm:quotient:abelian}]
It suffices to consider $\mu_p$-actions.
A $\mu_p$-action on $A$ corresponds (\cite{SGA3-1}*{Th\'eor\`eme VII.7.2(ii)}) to an element $v \in H^0(A, T) \cong T_0 A$ satisfying $v^p = v$, 
and hence the action can be identified with the translation action by a subgroup scheme of $A$.
Then the quotient is again an abelian variety,
and hence this action is symplectic 
by Proposition \ref{prop:2-forms}.
\end{proof}

\begin{proof}[Proof of Theorem \ref{thm:quotient:symplectic}]
By Theorem \ref{thm:mu_p RDP}(\ref{thm:mu_p RDP:non-fixed RDP},\ref{thm:mu_p RDP:fixed RDP symplectic}),
$Y = X/\mu_p$ is an RDP surface.
Then by Proposition \ref{prop:2-forms}, $K_{Y}$ is trivial, and by the classification of surfaces $Y$ is either
an RDP K3 surface, an abelian surface, a non-classical RDP Enriques surface, or a (quasi-)hyperelliptic surface.

Since $\map{\pi}{X}{Y}$ is purely inseparable we have $\dim \Het^i(X, \bQ_l) = \dim \Het^i(Y, \bQ_l)$,
in particular $b_1(Y) = b_1(X) = 0$ (where $b_i = \dim \Het^i$ is the $i$-th Betti number).
Since $Y$ is an RDP surface and 
since $\cO_Y = (\pi_* \cO_X)_0$ is a direct summand of $\pi_* \cO_X$,
we have $h^1(\tilde{Y}, \cO_{\tilde{Y}}) = h^1(Y, \cO_Y) \leq h^1(Y, \pi_* \cO_X) = h^1(X, \cO_X) = 0$.
Hence $Y$ is an RDP K3 surface.
\end{proof}

\begin{proof}[Proof of Theorem \ref{thm:quotient:non-symplectic}]
Let $X$ be an RDP K3 surface 
equipped with a nontrivial non-symplectic $\mu_p$-action.
Let $Y = X/\mu_p$.
We have $b_1(Y) = b_1(X) = 0$,
hence as in the tame case (Section \ref{sec:tame quotient}), 
the minimal resolution $\tilde{Y}$ of $Y$ cannot be abelian, (quasi-)hyperelliptic, nor non-rational ruled.

If $D$ has non-isolated fixed points then, 
by the Rudakov--Shafarevich formula $K_X \sim (p-1) \divisorialfix{D} + \pi^* K_Y$
\cite{Rudakov--Shafarevich:inseparable}*{Corollary 1 to Proposition 3}, or by Proposition \ref{prop:2-forms},
$\kappa(\tilde{Y}) = -\infty$.
If $D$ has an isolated fixed point $w \in X$, then 
by Theorem \ref{thm:mu_p RDP}(\ref{thm:mu_p RDP:fixed RDP non-symplectic}), $\pi(w) \in Y$ is a non-RDP singularity,
and then hence $\kappa(\tilde{Y}) = -\infty$.
In either case $Y$ is a rational surface.

Now assume $D$ is fixed-point-free.
Then by Theorem \ref{thm:mu_p RDP}(\ref{thm:mu_p RDP:non-fixed RDP}) 
$Y$ is an RDP surface and, 
by the Rudakov--Shafarevich formula, $K_{Y}$ is torsion.
Moreover it follows from Proposition \ref{prop:2-forms}
that the space $H^0(Y^{\sm}, (\Omega^2)^{\otimes n})$ is $0$ if $0 < n < p$ and
is generated by a non-vanishing multicanonical form if $n = p$.
Thus $K_{\tilde{Y}}$ is nonzero and $p$-torsion.
By the classification of surfaces
it follows that $Y$ is an RDP Enriques surface. 
Then since $2K_{Y} \sim 0$ and $H^0(Y^{\sm}, \cO(n K_{Y})) = 0$ for $0 < n < p$, we have $p = 2$.
\end{proof}

There is also the following relation with the height of K3 surfaces.
The height is an invariant of a K3 surface in positive characteristic 
which is either $\infty$ or an integer in $\set{1, \dots, 10}$.
	See Section \ref{sec:order} for more details.

\begin{cor}	\label{cor:quotient:height}
	Let $X$ be an RDP K3 surface in characteristic $p$ equipped with a nontrivial $\mu_p$-action.
	If $X$ is of finite height, then the action is symplectic and the quotient is an RDP K3 surface.
\end{cor}
\begin{proof}%[Proof of Corollary \ref{cor:quotient:height}]
We assume the action is non-symplectic and show that then $X$ is not of finite height.
By Theorem \ref{thm:quotient:non-symplectic}, the quotient  $Y = X / \mu_p$ is a rational surface or an RDP Enriques surface.
	Hence $X$ admits a purely inseparable finite morphism $\thpower{Y}{1/p} \to X$ from a rational surface or an RDP Enriques surface.
	Hence $\Het^2(X, \bQ_l)$, which is isomorphic to $H^2(\thpower{Y}{1/p}, \bQ_l)$, is generated by algebraic cycles,
	which is impossible if $X$ is of finite height by Lemma \ref{lem:H2 of K3}(\ref{lem:H2 of K3:bound of height}). 
\end{proof}
\begin{rem}	
	In a subsequent paper \cite{Matsumoto:k3rdpht}*{Theorem 1.3},
	we prove that also the converse holds,
	and we moreover determine the height in terms of the singularities of $X$ and $Y$.
\end{rem}

We call a proper birational morphism $X' \to X$ between RDP surfaces to be a \emph{partial resolution} if it is dominated by the minimal resolution $\tilde{X}$ of $X$.
The following proposition often enables us to reduce assertions on $\mu_p$-actions to simpler cases.

\begin{prop} \label{prop:to maximal}
Let $X$ be an RDP surface equipped with a $\mu_p$-action.
\begin{enumerate}
\item \label{prop:to maximal:existence} Among partial resolutions of $X$ to which the $\mu_p$-action extends,
there exists a unique maximal one, which we call the maximal partial resolution of $X$.
\item \label{prop:to maximal:characterization} A partial resolution $X' \to X$ is maximal if and only if it satisfies the property
$\Sing(X') \cap \pi^{-1}(\Sing(X'/\mu_p)) = \emptyset$.
\item \label{prop:to maximal:fixed-point-free} The action on $X$ is fixed-point-free if and only if the action on the maximal partial resolution is fixed-point-free.
\item \label{prop:to maximal:symplectic} Suppose $X$ is an RDP K3 surface.
The action on $X$ is symplectic if and only if the action on the maximal partial resolution is symplectic.
\end{enumerate}
\end{prop}
We say that an RDP surface $X$ equipped with a $\mu_p$-action is \emph{maximal} if it is the maximal partial resolution of itself.

\begin{proof}%[Proof of Proposition \ref{prop:to maximal}]
Note that isomorphism classes of partial resolutions are in one-to-one correspondence to subsets of the set of exceptional curves of $\tilde{X} \to X$.

(\ref{prop:to maximal:existence})
Let $X'$ be a partial resolution.
Let $\tilde{D}$ and $D'$ be the rational derivations induced by $D$ on $\tilde{X}$ and $X'$.
Clearly, $D'$ is regular and thus corresponds to a $\mu_p$-action 
if and only if the coefficient of $\divisorialfix{D'}$ for each exceptional curve of $X' \to X$ is non-negative.
Therefore, the contraction of all exceptional curves on $\tilde{X}$ with negative coefficients in $\divisorialfix{\tilde{D}}$ is the maximal partial resolution.

(\ref{prop:to maximal:characterization})
First suppose $X'$ does not satisfy the property.
There is an RDP $w \in X'$ such that $\pi(w)$ is not smooth.
If $w$ is a fixed RDP, then let $X'_1 = \Bl_w X'$.
If $w$ is a non-fixed RDP, then let $X'_1 = X' \times_{Y'} Y'_1$,
where $Y'$ is the $\mu_p$-quotient of $X'$ and $Y'_1 \to Y'$ is the minimal resolution at $\pi(w)$.
Then the $\mu_p$-action extends to $X'_1$ (by Proposition \ref{prop:fixedpoints} in the former case, clear in the latter case)
and $X'_1$ is a partial resolution of $X'$ (clear in the former case, by Theorem \ref{thm:mu_p RDP}(\ref{thm:mu_p RDP:non-fixed RDP}) in the latter case).
Thus $X'$ is not maximal.

Conversely, suppose $X'$ satisfies the property.
Let $w \in \Sing(X')$. 
To show that $X'$ is maximal, it suffices to show that 
every exceptional curve of $\tilde{X}$ above $w$ appears in $\divisorialfix{\tilde{D}}$ with negative coefficient, since then it must be contracted in the maximal partial resolution.
	This follows from the Rudakov--Shafarevich formula 
	$K_{\tilde{X}} \sim (p-1) \divisorialfix{\tilde{D}} + \pi^* K_{\tilde{Y}}$ \cite{Rudakov--Shafarevich:inseparable}*{Corollary 1 to Proposition 3},
	where $\tilde{Y} = \tilde{X}/\mu_p$:
	For each exceptional curve, 
	its coefficient in $K_{\tilde{X}}$ (resp.\ $\pi^* K_{\tilde{Y}}$) 
	is $0$ (resp.\ positive) since $\tilde{X}$ is the minimal resolution of the RDP $w$ (resp.\ since $\pi(w)$ is a smooth point).
Alternatively, one can use the explicit computation of $\divisorialfix{\tilde{D}}$ given in Lemma \ref{lem:above non-fixed RDP} below.
($w$ is non-fixed since $\pi(w)$ is a smooth point (Theorem \ref{thm:mu_p RDP}(\ref{thm:mu_p RDP:fixed RDP symplectic},\ref{thm:mu_p RDP:fixed RDP non-symplectic})).)

We also observe that the maximal partial resolution of $X$ can be constructed by 
repeatedly applying the procedure of constructing $X_1'$ from $X'$.
Indeed, the total index of RDPs strictly decreases through the procedure.

(\ref{prop:to maximal:fixed-point-free})
It suffices to show that the procedure above preserves the (non-)fixed-point-freeness.
If $w$ is a non-fixed RDP 
then there is no fixed point above $w$ by the functoriality of the fixed point scheme.
If $w$ is a fixed RDP then there is a fixed point above $w$ by Corollary \ref{cor:fixed point above}.

(\ref{prop:to maximal:symplectic})
The spaces of $2$-forms in question are naturally isomorphic.
\end{proof}

\begin{lem} \label{lem:above non-fixed RDP}
	Let $X$ be an RDP surface
	and $D$ a derivation of multiplicative type.
	Let $w \in X$ be a non-fixed RDP of index $n$ and suppose the image of $w$ is a smooth point.
	Let $\tilde X$ be the minimal resolution at $w$
	and $\tilde D$ be the induced rational derivation on $\tilde X$.
	Let $\divisorialfix{\tilde{D}}$ and $\isolatedfix{\tilde{D}}$ be the divisorial and isolated fixed locus of $\tilde D$ above $w$.
	Then every exceptional curve of $\tilde{X}$ above $w$ appear in $\divisorialfix{\tilde{D}}$ with negative coefficient,
	and we have $\deg \isolatedfix{\tilde{D}} = \frac{p-2}{p-1} n$
	and $\divisorialfix{\tilde{D}}^2 = - \frac{2}{p-1} n$.
\end{lem}
\begin{proof}
	We compute $\divisorialfix{\tilde{D}}$ and $\isolatedfix{\tilde{D}}$ by a straightforward calculation using the classification of
	non-fixed RDPs given in Proposition \ref{prop:non-fixed mu_p RDP}.
	See \cite{Matsumoto:k3alphap}*{Lemma 3.11} for the result.
	Then the assertions follow.
\end{proof}

\begin{prop} \label{prop:number of fixed points}
	Suppose each of $X$ and $Y$ is either an RDP K3 surface or an RDP Enriques surface.
	Let $\pi \colon X \to Y$ be a $\mu_p$-quotient morphism.
	Suppose $X$ is maximal with respect to the $\mu_p$-action.
	Then the total index $N_1$ of non-fixed RDPs 
	and the number $N_2$ of fixed points on $X$
	are as follows.
	\[
	(N_1, N_2) = \begin{cases}
	(24 \frac{p-1}{p+1}, 24 \frac{1}{p+1})                & \text{if $(\tilde X, \tilde Y)$ is $(\textup{K3},\textup{K3})$}, \\
	(12 \frac{2p-1}{p+1}, 12 \frac{p-2}{p^2-1}) = (12, 0) & \text{if $(\tilde X, \tilde Y)$ is $(\textup{K3},\textup{Enr})$}, \\
	(12 \frac{p-2}{p+1}, 12 \frac{2p-1}{p^2-1}) = (0, 12) & \text{if $(\tilde X, \tilde Y)$ is $(\textup{Enr},\textup{K3})$}, \\
	(12 \frac{p-1}{p+1}, 12 \frac{1}{p+1})                & \text{if $(\tilde X, \tilde Y)$ is $(\textup{Enr},\textup{Enr})$}. \\
	\end{cases}
	\]
	In the cases where $(\tilde{X}, \tilde{Y})$ is 
	$(\textup{K3}, \textup{Enr})$ or $(\textup{Enr}, \textup{K3})$,
	only $p = 2$ is possible.
\end{prop}

\begin{proof}
	Let $D$ be the corresponding derivation of multiplicative type.
	Since $\tilde{Y}$ is of Kodaira dimension $0$, 
	$\Fix(D)$ consists only of isolated points (possibly none).
	
	Define the rational derivation $\tilde D$ on $\tilde{X}$ as in Lemma \ref{lem:above non-fixed RDP}.
	Since the fixed locus of $\tilde D$ consists of those above non-fixed RDPs on $X$ and the $0$-cycle of fixed points on $X$,
	by Lemma \ref{lem:above non-fixed RDP} we have
	\[
	\divisorialfix{\tilde{D}}^2 = - \frac{2}{p-1} N_1, \quad
	\deg \isolatedfix{\tilde{D}} = N_2 + \frac{p-2}{p-1} N_1,
	\]
	and by Katsura--Takeda formula (Proposition \ref{prop:globalvectorfields}) we have
	\[
	\deg c_2(\tilde{X})
	= (N_2 + \frac{p-2}{p-1} N_1) + 0 + \frac{2}{p-1} N_1 
	= N_2 + \frac{p}{p-1} N_1.
	\]
	
	On the other hand,
	we have $\dim \Het^2(X, \bQ_l) = \dim \Het^2(\tilde{X}, \bQ_l) - N_1 = b_2(\tilde{X}) - N_1$.
	Since $Y$ is an RDP surface 
	whose RDPs are precisely the images (which are $A_{p-1}$) of the fixed smooth points of $X$,
	we have $\dim \Het^2(Y, \bQ_l) = \dim \Het^2(\tilde{Y}, \bQ_l) - (p-1)N_2 = b_2(\tilde{Y}) - (p-1)N_2$.
	Since $X \to Y$ is purely inseparable we have $\dim \Het^2(X, \bQ_l) = \dim \Het^2(Y, \bQ_l)$.
	Therefore we have $(p-1) N_2 - N_1 = b_2(\tilde{Y}) - b_2(\tilde{X})$.
	
	Combining the two equalities we determine $(N_1,N_2)$ in terms of $p$.
	In two cases only $p = 2$ is possible 
	since $12 \frac{p-2}{p^2-1}$ or $12 \frac{2p-1}{p^2-1}$ must be an integer.
\end{proof}

\section{Possible orders of symplectic \texorpdfstring{$\mu_n$-actions}{mu\_n-actions} on RDP K3 surfaces} \label{sec:symplectic order}

The following theorem is again parallel to the case of automorphisms of finite tame order,
but the proof (for $\mu_p$-actions) is quite different.

\begin{thm} \label{thm:symplectic mu_n}
Let $X$ be an RDP K3 surface in characteristic $p$ equipped with a symplectic $\mu_{n}$-action ($n > 1$, divisible by $p$ or not).
Then,
\begin{enumerate}
\item \label{thm:symplectic mu_n:bound}
$n \leq 8$.

\item \label{thm:symplectic mu_n:fixedpoints}
The number of fixed points, counted with the multiplicities defined after Proposition \ref{prop:symplectic mu_n RDP},
depends only on $n$ and is as in Table \ref{table:symplectic mu_n K3}.
%is equal to $(24/n) \prod_{l:\text{prime},l \divides n} (l/(l+1))$
%(i.e.\ $N = 8,6,4,4,2,3,2$ for $n = 2,3,4,5,6,7,8$ respectively).

\item \label{thm:symplectic mu_n:eigenvalues}
Assume all fixed points are smooth points.
Then the decomposition of $\bigoplus_{w \in \Fix(\mu_n)} T^*_w X$ with respect to the $\mu_n$-action
is concentrated on the subset $(\bZ/n\bZ)^* \subset \bZ/n\bZ$, 
and for every $i \in (\bZ/n\bZ)^*$ the $i$-th summand has dimension
as in Table \ref{table:symplectic mu_n K3}.
% equal to $2N / \phi(n) = (48/n^2) \prod_{l:\text{prime},l \divides n} (l^2/(l^2-1))$
%(i.e.\ $ = 16,6,4,2,2,1,1$ for $n = 2,3,4,5,6,7,8$ respectively).
\end{enumerate}
\end{thm}

\begin{table} \label{table:symplectic mu_n K3}
\caption{Symplectic $\mu_n$-actions on K3 surfaces}
\begin{tabular}{cccccccc}
\toprule
$n$ & $2$ & $3$ & $4$ & $5$ & $6$ & $7$ & $8$ \\
\midrule
number of fixed points & $8$ & $6$ & $4$ & $4$ & $2$ & $3$ & $2$ \\
dimension of each summand & $16$ & $6$ & $4$ & $2$ & $2$ & $1$ & $1$ \\
\bottomrule
\end{tabular}
\end{table}

If $p \notdivides n$ (resp.\ if $n = p$)
then assertion (\ref{thm:symplectic mu_n:eigenvalues}) means that 
every primitive $n$-th root of $1$ (resp.\ every element of $\bF_p^*$)
appears as an eigenvalue of a fixed generator of $\mu_n \cong \bZ/n\bZ$ (resp.\ of the corresponding derivation)
on the space $\bigoplus_{w \in \Fix(\mu_n)} T^*_w X$
with equal multiplicity.

For each $p$ and each $n \leq 8$ there indeed exists 
an RDP K3 surface equipped with a symplectic $\mu_{n}$-action in characteristic $p$. 
See Section \ref{subsec:examples:symplectic} for explicit examples.

\begin{proof}[Proof of Theorem \ref{thm:symplectic mu_n} for the case $p \notdivides n$]
We may assume $X$ is smooth.

Assertions (\ref{thm:symplectic mu_n:bound}) and (\ref{thm:symplectic mu_n:fixedpoints}) are proved by Nikulin \cite{Nikulin:auto}*{Section 5} ($p = 0$) and 
Dolgachev--Keum \cite{Dolgachev--Keum:auto}*{Theorem 3.3} ($p > 0$).
(Both proofs overlooked the case $n = 14$, 
but their arguments for the non-existence of the case $n = 15$ apply to case $n = 14$.)

(\ref{thm:symplectic mu_n:eigenvalues})
(If $p = 0$ then this follows from the argument in \cite{Mukai:automorphismsK3}*{Proposition 1.2}. We give another proof, applicable to all $p \geq 0$.)

Let $w \in X$ and let $\mu_r = \Stab(w) \subset \mu_n$ be its stabilizer group.
Suppose the decomposition of $T^*_w X$ ($ = \idealm_w / \idealm_w^2$) with respect to the $\mu_r$-action is concentrated on two (not necessarily distinct) weights $j_1, j_2 \in \bZ/r\bZ$.
Then since the action on $\Omega^2_{X,w} \cong \bigwedge^2 T^*_w X$ is trivial we have $j_1 + j_2 = 0$.
We have $j_1,j_2 \in (\bZ/r\bZ)^*$, since if $\gcd(j_1, r) = \gcd(j_2, r) = r' > 1$ then $\mu_{r'}$ acts trivially,
contradicting the assumption that the action is faithful.
This already proves the assertion if $n = 2,3,4,6$, since up to sign there is only one element in $(\bZ/n\bZ)^*$.

For each divisor $r \neq 1$ of $n$ 
and for each $0 \leq j \leq \floor{r/2}$, 
let $S_{r,j}$ be the set of the points $w \in X$
with $\Stab(w) = \mu_r$
and with $\mu_r$ acting on $T^*_w X$ by weights $j$ and $-j$.
Then $S_{r,j}$ is a finite set and it is empty if $j \not\in (\bZ/r\bZ)^*$.
Let $\tilde{S}_{r,j} = S_{r,j} / \mu_n$ be the set of $\mu_n$-orbits of points of $S_{r,j}$.
Let $N_{r,j} = \card{S_{r,j}}$ and $\tilde{N}_{r,j} = \card{\tilde{S}_{r,j}} = N_{r,j} / (n/r)$.

Let $\map{\rho}{\tilde{Y}}{Y = X / \mu_n}$ be the minimal resolution (then $\tilde{Y}$ is a smooth K3 surface),
and let $\pi' \colon X' = X \times_Y \tilde{Y} \to \tilde{Y}$.
Let $\pi'_* \cO_{X'} = \bigoplus_{i \in \bZ/n\bZ} (\pi'_* \cO_{X'})_i$ be the decomposition induced by the $\mu_n$-action.
Then by Lemma \ref{lem:A}, $(\pi'_* \cO_{X'})_i$ are invertible sheaves, and they are described as follows.
For each $i \in \bZ/n\bZ$, let $C_i$ be the corresponding class of Cartier divisors.
For each orbit $w \in \tilde{S}_{r,j}$, its image $\pi(w)$ is an RDP of type $A_{r-1}$,
and let $e_{w,k}$ ($k = 1, \dots, r-1$) be the exceptional curves in $\tilde{Y}$ above $\pi(w)$,
ordered in a way that $e_{w,k} \cap e_{w,k'} \neq \emptyset$ if and only if $\abs{k - k'} \leq 1$.
Then after possibly reversing the ordering we have a linear equivalence
\[ 
- r C_i \sim
\sum_{r \divides n, r \neq 1} \sum_{j = 0}^{\floor{r/2}} \sum_{w \in \tilde{S}_{r,j}} \sum_{k = 1}^{r-1} 
{f_r(({j^{-1}} i \bmod r), k)} e_{w,k}
\]
 (see Lemma \ref{lem:A} for the definition of the function $\map{f_r}{\set{1, \dots, r-1}^2}{\bZ}$) for each $i \neq 0$.

Let $m$ be any integer with $1 \leq m \leq r-1$.
Using the equality
\[
 \Biggl( \sum_{k=1}^{r-1} \frac{f_r(m,k)}{r} e_{w,k} \Biggr) \cdot e_{w,k'} 
 = \begin{cases}
 -1 & \text{if $k' = r - m$,} \\
  0 & \text{otherwise,}
 \end{cases}
\]
we compute that 
\[ \Biggl( \sum_{k=1}^{r-1} \frac{f_r(m,k)}{r} e_{w,k} \Biggr)^2 = -\frac{m(r-m)}{r}. \]
Hence we have 
\[ - C_i^2 = \sum_{r \divides n, r \neq 1} \sum_{j = 0}^{\floor{r/2}} 
\tilde{N}_{r,j} \cdot \frac{({j^{-1}}i \bmod r)(r - {j^{-1}}i \bmod r)}{r} \]
and this must belong to $2 \bZ$.

Assume $n = 5$.
Then we have $N_{5,j} = \tilde{N}_{5,j}$, $N_{5,1} + N_{5,2} = 4$, and
\[ -C_1^2 = \frac{4}{5} N_{5,1} + \frac{6}{5} N_{5,2} \in 2 \bZ. \]
Hence $(N_{5,1}, N_{5,2}) = (2, 2)$.

Assume $n = 7$. 
Then we have $N_{7,j} = \tilde{N}_{7,j}$, $N_{7,1} + N_{7,2} + N_{7,3} = 3$, and
\[ 
-C_1^2 = \frac{6}{7} N_{7,1} + \frac{12}{7} N_{7,2} + \frac{10}{7} N_{7,3} \in 2 \bZ.
\]
Hence $(N_{7,1}, N_{7,2}, N_{7,3}) = (1, 1, 1)$.

Assume $n = 8$.
By assertion (\ref{thm:symplectic mu_n:fixedpoints}) for the cases $n = 2,4,8$, 
we have $\tilde{N}_{2,1} = 1$, $\tilde{N}_{4,1} = 1$, $\tilde{N}_{8,1} + \tilde{N}_{8,3} = 2$,
and 
\[
-C_1^2 = \frac{1}{2} \tilde{N}_{2,1} + 
         \frac{3}{4} \tilde{N}_{4,1} + 
         \frac{7}{8} \tilde{N}_{8,1} + \frac{15}{8} \tilde{N}_{8,3} 
 \in 2 \bZ.
\]
Hence $({N}_{8,1}, {N}_{8,3}) = (\tilde{N}_{8,1}, \tilde{N}_{8,3}) = (1, 1)$.
\end{proof}

\begin{rem}
By above, we have $C_i^2 = -4$ for any $2 \leq n \leq 8$ and any $1 \leq i \leq n-1$. 
As we will see below, this holds also if $p$ divides $n$. 
This implies $\chi(\tilde{Y}, (\pi'_* \cO_{X'})_i) = 0$ for $i \neq 0$.
Then we obtain $\chi(Y, (\pi_* \cO_{X})_i) = 0$ for $i \neq 0$,
since $\rho_*((\pi'_* \cO_{X'})_i) = (\pi_* \cO_{X'})_i$
and $R^q \rho_*((\pi'_* \cO_{X'})_i) = 0$ for $q > 0$.
This is finer than the equality
\[
 \sum_{i \neq 0} \chi(Y, (\pi_* \cO_{X})_i)
 = \chi(X, \cO_X) - \chi(Y, \cO_Y) = 2 - 2 = 0.
\]
\end{rem}

\begin{proof}[Proof of Theorem \ref{thm:symplectic mu_n} for the case $n = p$]
We may assume that $X$ is maximal. 
(For assertion (\ref{thm:symplectic mu_n:fixedpoints}),
the multiplicity is by definition compatible with blow-ups at fixed points.) 
This means that all fixed points are smooth points,
and that the singularities of the quotient surface $Y$ are all $A_{p-1}$ and are precisely the images of the fixed points.
Let $D$ be the corresponding derivation.

As in the previous case,
let $\tilde{Y}$ be the minimal resolution of $Y$ (hence a smooth K3 surface)
and let $\pi' \colon X' = X \times_Y \tilde{Y} \to \tilde{Y}$.
The sheaf $\pi'_* \cO_{X'}$ admits a decomposition to invertible sheaves $(\pi'_* \cO_{X'})_i$.
For each $i \in \bZ/p\bZ$, let $C_i$ be the corresponding class of Cartier divisor.
As in the previous case we have 
\[ - C_i^2 = \sum_{j = 0}^{\floor{p/2}} 
N_{p,j} \cdot \frac{({j^{-1}}i \bmod p)(p - {j^{-1}}i \bmod p)}{p}. \]

By Proposition \ref{prop:number of fixed points} we have $N := \sum_{j} N_{p,j} = 24/(p+1)$.
Hence $p \in \set{2,3,5,7,11,23}$.
If $p = 23$ then $N = 1$ and 
the exceptional curves generate a negative definite sublattice of rank $p-1 = 22$ of the indefinite lattice $\Pic(\tilde{Y})$ of rank $\leq 22$,
contradiction.
If $p = 11$ then $N = 2$ and then $C_i^2$ (for any $i \in (\bZ/p\bZ)^*$) 
cannot be an integer since the sum of two nonzero squares in $\bF_{11}$ cannot be zero.
Hence we have $p \in \set{2,3,5,7}$,
and we can determine the multiplicities of the weights as in the $p \notdivides n$ case.
\end{proof}

\begin{cor} \label{cor:symplectic mu_n:fixed RDP}
Let $X$ be an RDP K3 surface equipped with a symplectic $\mu_{q}$-action with $q = 5,7$. Here, both $p = q$ and $p \neq q$ are allowed.
\begin{itemize}
\item If $q = 7$, then any fixed point is a smooth point.
\item If $q = 5$, then any fixed point is a smooth point or an RDP of type $A_1$.
\end{itemize}
\end{cor}
\begin{proof}
Let $w \in X$ be a fixed point. By Proposition \ref{prop:symplectic mu_n RDP}, $w$ is of type $A_{m-1}$ for some $m \geq 1$. 
Let $\pm i \in (\bZ/q\bZ)^*$ be the nonzero weights of $\idealm_w/\idealm_w^2$ 
with respect to the $\mu_q$-action.
Let $\tilde{X}$ be the minimal resolution of $X$ at $w$ (to which the $\mu_q$-action extends).
One can calculate the local equation to show that all (smooth) fixed point of $\tilde{X}$ above $w$ has weights $\pm i$.
Since there are $m$ such points,
it follows from Theorem \ref{thm:symplectic mu_n}(\ref{thm:symplectic mu_n:eigenvalues}) that $m \leq 1$ if $q = 7$ and $m \leq 2$ if $q = 5$.
\end{proof}

\begin{proof}[Proof of Theorem \ref{thm:symplectic mu_n} for the case $n = p^e$ ($e \geq 2$)]

For each $0 \leq j \leq e$,
let $\pi_j \colon X \to X_j = X/\mu_{p^j}$ be the quotient morphism by the subgroup scheme $\mu_{p^j} \subset \mu_{p^e}$,
and for each $0 \leq j \leq e-1$, let $D_j$ be the derivation on $X_j$ corresponding to the action of $\mu_{p^{j+1}}/\mu_{p^j}$.

Let $w \in X$ be a $\mu_p$-fixed point.
Let $\mu_{p^f} = \Stab(w)$ (then $1 \leq f \leq e$).
Then by Remark \ref{rem:check locally mu_p}, $\mu_{p^f}$ acts symplectically at $w$.
By Proposition \ref{prop:mu_p^e RDP}, 
either $w$ is of type $A_{m-1}$ for some $m \geq 1$ with $p^{e-f} \divides m$, or $w$ is $D_{2m+1}^{m-1}$ or $E_7^2$ or $E_8^3$ and $p^e = 4$.
(Again we use the convention that a smooth point is of type $A_0$.)
Then since each $D_j$ ($j < f$) is symplectic at $\pi_j(w)$, we observe that $\pi_f(w) \in X_f$ is of type $A_{p^e m - 1}$ or $D_5^0$ or $D_{2m+1}^{m-1}$.
(Since $X$ has a $\mu_p$-fixed point and since $X_f$ is an RDP K3 surface, this already implies $p^e - 1 < 22$.)

By Lemma \ref{lem:image of non-fixed point}, any preimage of any fixed point of $D_{j}$ is again fixed.
In other words, the fixed points of $D_j$ on $X_j$ are precisely the images of the $\mu_{p^{j+1}}$-fixed points on $X$.

For each $1 \leq f \leq e$,
let $\tilde{S}_f \subset X$ be the points with stabilizer equal to $\mu_{p^f}$.
For each $w \in \tilde{S}_f$, 
let $m(w)$ be its multiplicity of $w$ defined after Proposition \ref{prop:symplectic mu_n RDP}
with respect to the $\mu_{p^f}$-action.
Then, by Propositions \ref{prop:mu_p^e RDP} and \ref{prop:symplectic mu_n RDP},
we have $p^{e-f} \divides m(w)$.
Let $M_f = \sum_{w \in \tilde{S}_f} m(w)$ for each $1 \leq f \leq e$, then by above $p^{e-f} \divides M_f$.
Using the equality mentioned after Proposition \ref{prop:symplectic mu_n RDP}
and assertion (\ref{thm:symplectic mu_n:fixedpoints}) 
for $D_{e-1}$ and $D_{e-2}$ on $X_{e-1}$ and $X_{e-2}$,
we obtain $p^{e-1} M_e = p^{e-2} (M_e + M_{e-1}) = 24/(p+1)$,
hence $M_e = 24 / (p^{e-1}(p+1))$ and 
$M_{e-1} / p = 24 (p-1) / (p^{e}(p+1))$.
Since $M_{e-1} / p$ is an integer, $p^e$ divides $24$. 
Therefore $p^e = 2^2, 2^3$.
Moreover we obtain $M_f = p^{e-f} \cdot 24 (p-1) / (p^{e}(p+1))$ ($1 \leq f \leq e-1$)
by applying assertion (\ref{thm:symplectic mu_n:fixedpoints}) to $D_{j}$ ($0 \leq j \leq e-1$).

Assertion (\ref{thm:symplectic mu_n:eigenvalues}) is trivial if $n = 4$. 
Suppose $n = 8$.
For each $1 \leq f \leq 3$ and $0 \leq j \leq 2^f/2$, 
let $\tilde{S}_{2^f,j}$ be the set of the points with stabilizer $\mu_{2^f}$ and with primitive weights $\pm j \in (\bZ/2^f\bZ)^*$,
and let $\tilde{N}_{2^f,j} = (2^{e-f})^{-1} \sum_{w \in \tilde{S}_{2^f,j}} m(w)$.
We have $\sum_{j} \tilde{N}_{2^f,j} = (2^{e-f})^{-1} M_f$ for each $1 \leq f \leq e$.
Then we again have $\tilde{N}_{2,1} = 1$, $\tilde{N}_{4,1} = 1$, $\tilde{N}_{8,1} + \tilde{N}_{8,3} = 2$,
and 
\[
-C_1^2 = \frac{1}{2} \tilde{N}_{2,1} + 
        \frac{3}{4} \tilde{N}_{4,1} + 
        \frac{7}{8} \tilde{N}_{8,1} + \frac{15}{8} \tilde{N}_{8,3} 
 \in 2 \bZ.
\]
Hence $(\tilde{N}_{8,1}, \tilde{N}_{8,3}) = (1,1)$.
\end{proof}

\begin{proof}[Proof of Theorem \ref{thm:symplectic mu_n} for the remaining cases]
First we show that if $n = pq$ where $q$ is a prime $\neq p$, then $n = 6$.
We have $\mu_n = \mu_p \times \mu_q \cong \mu_p \times \bZ/q\bZ$.
We may assume that $X$ is maximal with respect to the $\mu_p$-action.
Let $\pi_q \colon X \to X_q = X / \mu_q$ and $\pi_p \colon X \to X_p = X / \mu_p$.
Note that $w \in X$ is fixed by the $\mu_p$-action if and only if $\pi_q(w) \in X_q$ is fixed by the $\mu_p$-action.
Let $a_1$ and $a_q$ be the number of $\mu_q$-orbits of length $1$ and $q$ of $\mu_p$-fixed points of $X$ (which are all smooth by assumption).
Then the $\mu_p$-fixed points of $X_q$ consists of $a_1$ points of type $A_{q-1}$ and $a_q$ smooth points. 
Applying assertion (\ref{thm:symplectic mu_n:fixedpoints}) to the $\mu_p$-actions on $X$ and $X_q$ we have $a_1 + q a_q = q a_1 + a_q = 24/(p+1)$.
Therefore $a_1 = a_q = 24/(p+1)(q+1)$
and hence $(a_1, \set{p, q }) = (2, \set{2, 3 }), (1, \set{2, 7 } ), (1, \set{3, 5 })$.

The cases $(a_1, \set{p, q }) = (1, \set{2, 7 } ), (1, \set{3, 5 })$ are impossible since,
letting $w \in X$ be the unique $\mu_{pq}$-fixed point (which is a smooth point),
if $pq = 14$ 
then $\pi_2(w) \in X_2$ is a $\mu_7$-fixed RDP of type $A_{1}$,
and if $pq = 15$ 
then $\pi_3(w) \in X_3$ is a $\mu_5$-fixed RDP of type $A_{2}$,
both contradicting Corollary \ref{cor:symplectic mu_n:fixed RDP}.

Now we consider general $n$. 
It remains to show that the cases $(p,n) = (2,12), (3,12)$ are impossible.

Assume $(p,n) = (3,12)$.
As above we may assume $X$ is maximal with respect to the $\mu_3$-action.
There are exactly six $\mu_3$-fixed points, all smooth.
By the above argument for $(p,n) = (3,6)$,
exactly two of them are $\mu_2$-fixed, and among the images of these two points in $X/\mu_2$ exactly one is $(\mu_4/\mu_2)$-fixed. 
This is impossible since non-$(\mu_4/\mu_2)$-fixed points in $X / \mu_2$ come by pairs.

Now assume $(p,n) = (2,12)$.
As in the proof of the $n = p^e$ case (applied to the $\mu_4$-action),
let $\tilde{S}_1$ be the set of $\mu_2$-fixed non-$\mu_4$-fixed points,
and then we have $M_1 = \sum_{w \in \tilde{S}_1} m(w) = 4$ and $2 \divides m(w)$.
Hence $\card{\tilde{S}_1}$ is $1$ or $2$.
Since the $\mu_3$-action on $X$ preserves this $1$- or $2$-point set $\tilde{S}_1$, 
it acts on $\tilde{S}_1$ trivially, hence fixes at least $4$ $\mu_2$-fixed points (counted with multiplicity $m(w)$), 
contradicting the observation $a_1 = 2$ for $\mu_6$-actions.

Assertion (\ref{thm:symplectic mu_n:eigenvalues}) for $n = 6$ is trivial.
\end{proof}

\section{Possible orders of \texorpdfstring{$\mu_n$-actions}{mu\_n-actions} on RDP K3 surfaces} \label{sec:order}

Let $\Scyc(p)$ (resp.\ $\Smu(p)$) be the set of positive integers $n$
for which there exists an RDP K3 surface equipped with an automorphism of order $n$ (resp.\ a $\mu_n$-action) in characteristic $p$.
We clearly have $\Scyc(0) = \Smu(0)$ and $\Scyc(p)^{p'} = \Smu(p)^{p'}$,
where $(-)^{p'}$ denotes the subset of prime-to-$p$ elements.

\begin{rem} \label{rem:Scyc}
Keum \cite{Keum:orders}*{Main Theorem} proved the following results on $\Scyc(p)$
(this set is denoted $\mathrm{Ord}_p$ in his paper).
The sets $\Scyc(p)$ for $p \neq 2,3$ are given by 
\begin{gather*}
\Scyc(0) = \set{n : \phi(n) \leq 20} \\
= \set{1, \dots, 22, 24, 25, 26, 27, 28, 30, 32, 33, 34, 36, 38, 40, 42, 44, 48, 50, 54, 60, 66},
\end{gather*}
and
\[
\Scyc(p) = \begin{cases}
\Scyc(0)  & \text{if $p = 7$ or $p \geq 23$,} \\
\Scyc(0) \setminus \set{p,2p} & \text{if $p = 13,17,19$,} \\
\Scyc(0) \setminus \set{44} & \text{if $p = 11$,} \\
\Scyc(0) \setminus \set{25,50,60} & \text{if $p = 5$.} \\
\end{cases}
\]
Moreover,
$\Scyc(p)^{p'} = \Scyc(0)^{p'}$ for all $p \geq 2$.
(The sets $\Scyc(2)$ and $\Scyc(3)$ are not determined.)
\end{rem}

In this section we determine the set $\Smu(p)$ for all $p$.
\begin{thm} \label{thm:mu_n}
We have 
\[
\Smu(p) = \begin{cases}
\Scyc(0) & \text{if $p \neq 2,3,5,11$,} \\
\Scyc(0) \setminus \set{33,66} & \text{if $p = 11$,} \\
\Scyc(0) \setminus \set{25,40,50} & \text{if $p = 5$,} \\
\Scyc(0) \setminus \set{27,33,48,54,66} & \text{if $p = 3$,} \\
\Scyc(0) \setminus \set{34,40,44,48,50,54,66} & \text{if $p = 2$.} 
\end{cases}
\]
In particular, there exists an RDP K3 surface equipped with a nontrivial $\mu_p$-action in characteristic $p$
if and only if $p \leq 19$.
\end{thm}

We need some preparations.	
The \emph{height} $h$ of a K3 surface $X$ in characteristic $p > 0$, whose definition we do not recall here, is either $\infty$ or an integer in $\set{1, \dots, 10}$,
and $X$ is called \emph{supersingular} or \emph{of finite height} respectively.
If $h < \infty$ then the inequality $\rho \leq 22 - 2h$ holds
(Lemma \ref{lem:H2 of K3}(\ref{lem:H2 of K3:bound of height})), 
where $\rho = \rank \Pic(X)$ is the Picard number.
This implies that if $\rho \geq 21$ then $X$ is supersingular.

In fact the Tate conjecture for K3 surfaces, now a theorem, states conversely that if $X$ is supersingular then $\rho = 22$: see Lemma \ref{lem:H2 of K3}(\ref{lem:H2 of K3:Tate conjecture}) below for references.
In this case, the $\bZ_p$-lattice $\Hcrys^2(X/W(k))^{F = p}$ is isomorphic to $\Pic(X) \otimes \bZ_p$ 
\cite{Ogus:K3crystals}*{Corollary 1.6}, 
and the discriminant group of $\Hcrys^2(X/W(k))^{F = p}$ 
(isomorphic to the discriminant group of $\Pic(X)$) is of the form
$(\bZ/p\bZ)^{2\sigma_0}$ for an integer $\sigma_0 \in \set{1, \dots, 10}$.
This $\sigma_0$ is called the \emph{Artin invariant} of $X$.
Here the \emph{discriminant group} of a non-degenerate lattice (resp.\ non-degenerate $\bZ_p$-lattice) $L$ is defined to be the finite group $L^*/L$,
where $L^* = \Hom_{\bZ}(L, \bZ)$ (resp.\ $L^* = \Hom_{\bZ_p}(L, \bZ_p)$) is the dual of $L$.

We define the \emph{crystalline transcendental lattice} $T(X) = \Tcrys(X) \subset \Hcrys^2(X/W(k))$ to be the orthogonal complement of the image of $\Pic(X) \otimes W(k)$,
where $W(k)$ is the ring of Witt vectors over $k$.
We have $\rho + \rank T(X) = 22$.

We collect some facts:

\begin{lem} \label{lem:H2 of K3}
Let $X$ be a K3 surface in characteristic $p > 0$.
\begin{enumerate}
\item \label{lem:H2 of K3:Q}
 $\Aut(X)$ acts on $\Hcrys^2(X/W(k))$ and $\Het^2(X, \bZ_l)$ (for any prime $l \neq p$) faithfully,
 and the characteristic polynomial of any element is independent of the cohomology (and $l$)
 and has coefficients in $\bQ$.
\item \label{lem:H2 of K3:bound of height}
 If $X$ is of finite height $h$, then $\rho \leq 22 - 2 h$.
%\item \label{lem:H2 of K3:cycle map}
% If $X$ is of finite height, then the Chern class morphism $\Pic(X) \otimes W(k) \to \Hcrys^2(X/W(k))$ 
% is a primitive embedding (i.e.\ it is injective and its cokernel is torsion-free).
\item \label{lem:H2 of K3:transcendental}
 Let $g \in \Aut(X)$ and suppose it acts on $H^0(X, \Omega^2)$ by a primitive $N$-th root of $1$.
 If $X$ is of finite height and $p \geq 3$,
 then the characteristic polynomial of $g^*$ on $\Tcrys(X)$
 is the product of cyclotomic polynomials $\Phi_{N p^{e_i}}$ with non-negative integers $e_i$.
 In particular $\phi(N) \divides \rank \Tcrys(X)$, 
 in particular $\rank \Tcrys(X) \geq \phi(N)$.
\item \label{lem:H2 of K3:Tate conjecture}
 If $X$ is supersingular, then $\rho = 22$.
\item \label{lem:H2 of K3:Artin invariant}
 Let $g \in \Aut(X)$ and define $N$ as in (\ref{lem:H2 of K3:transcendental}).
 If $X$ is supersingular of Artin invariant $\sigma_0$,
 then $N \divides (p^{\sigma_0} + 1)$.
\end{enumerate}
\end{lem}

An immediate consequence of (\ref{lem:H2 of K3:Artin invariant}) is that, 
letting $g \in \Aut(X)$ and $N$ be as in (\ref{lem:H2 of K3:transcendental}), 
if there exists no integer $\sigma_0$ with $N \divides (p^{\sigma_0} + 1)$,
then $X$ is not supersingular.
This applies to, e.g., $(p,N) = (3,8), (5,4)$.

\begin{proof}
(\ref{lem:H2 of K3:Q}) \cite{Keum:orders}*{Theorem 1.4}. 

(\ref{lem:H2 of K3:bound of height})
\cite{Illusie:deRham--Witt}*{Proposition II.5.12}.

%(\ref{lem:H2 of K3:cycle map})
%Since the crystalline Chern classes and the de Rham Chern classes are compatible \cite{Berthelot--Ogus:dRcrys}*{Proposition 3.4},
%it suffices to observe that if $X$ is of finite height then 
%$\Pic(X) \otimes_{\bZ} k \to \HdR^2(X)$ is injective, and this is proved by 
%van der Geer--Katsura \cite{vanderGeer--Katsura:stratification}*{Proposition 10.3}.

(\ref{lem:H2 of K3:transcendental})
See \cite{Matsumoto:non-symplectic}*{Lemma 2.4(3)},
which deduces the assertion from \cite{Jang:lifting}*{Theorem 3.2}.

(\ref{lem:H2 of K3:Tate conjecture})
This assertion, the Tate conjecture for supersingular K3 surfaces,
is proved by 
Madapusi Pera \cite{MadapusiPera:TateK3}*{Theorem 1} for characteristic $\geq 3$
and by Kim--Madapusi Pera \cite{Kim--MadapusiPera:2-adic}*{Theorem A.1} for characteristic $2$.

The assertion under the assumption that $X$ admits an elliptic fibration,
which is true for example if $\rho \geq 5$ (which is always the case when we use this assertion in this paper),
was proved much earlier by Artin \cite{Artin:supersingularK3}*{Theorem 1.7}.

(\ref{lem:H2 of K3:Artin invariant})
This is proved by Nygaard \cite{Nygaard:higherdeRham-Witt}*{Theorem 2.1} under the assumption $p \neq 2$.
The argument is in fact valid for $p = 2$, see \cite{Matsumoto:non-symplectic}*{Remark 2.2}.
\end{proof}

\begin{proof}[Proof of Theorem \ref{thm:mu_n}]
Let $\Smu'(p)$ be the set on the right hand side of the statement.
If $n$ is a positive integer not divisible by $p$
then $\mu_n$-action is equivalent to the action a cyclic group of order $n$ and, as noted in Remark \ref{rem:Scyc}, 
Keum \cite{Keum:orders}*{Main Theorem} proved that $n$ is the order of some automorphism of a K3 surface in characteristic $p$
if and only if $n \in \Smu'(p)$ (equivalently $n \in \Smu(0)$),

If $n \in \Smu'(p)$ and $p \divides n$,
then the examples given in Example \ref{ex:mu_n} show that $n \in \Smu(p)$.

Now take $n \in \Smu(p)$ with $p \divides n$, and we will show $n \in \Smu'(p)$.
Write $n = p^e r$ with $p \notdivides r$.
Since a smooth K3 surface never admits a $\mu_p$-action,
an example $X$ must have an RDP $w$.
Since $\mu_{p^e}$-fixed RDPs can be blown up, we may assume $w$ is not $\mu_{p^e}$-fixed.
Note that such RDPs are classified in Proposition \ref{prop:non-fixed mu_n RDP}.
We show in each case that $n$ belongs to $\Smu'(p)$.

Suppose $w$ is $D_m$ or $E_m$.
Let $\mu_{p^f s} = \Stab(w) \subset \mu_{p^e r}$ (with $p \notdivides s$).
Then the pair $(w, p^e s)$ appears in Table \ref{table:non-fixed mu_n RDP} 
and we have $(r/s)m < 22$.
Then we observe that $n \in \Smu'(p)$ except in the following cases:
$(p,n,s,(r/s)w) = 
(2,54,9,3E_7^0), 
(2,40,5,D_{21}^0), 
(2,34,17,D_{18}^0),
(2,34,17,D_{19}^0)$.
These exceptional cases do not occur,
since it follows that $\mu_s$ acts trivially on the classes of every ($21, 21, 18, 19$) exceptional curves,
which gives a too large invariant subspace of $\Het^2(\tilde{X}, \bQ_l)$ for an order $s$ automorphism 
(which must act on $\Het^2$ faithfully with a characteristic polynomial with coefficients in $\bQ$ by Lemma \ref{lem:H2 of K3}(\ref{lem:H2 of K3:Q})).

Suppose $w$ is $A_{m'-1}$.
Let $\mu_{p^f s} = \Stab(w) \subset \mu_{p^e r}$ (with $p \notdivides s$).
We have $0 \leq f < e$.
It follows from Proposition \ref{prop:non-fixed mu_n RDP} that $p^{e-f} \divides m'$, so write $m' = p^{e-f} m$ with $m \geq 1$, and that
$\mu_{p^f t}$ acts on $w$ symplectically
where either $s = t$ or $(s,t) = (2,1)$.
Then $X/\mu_{p^f t}$ has $(r/s) A_{p^f t m' - 1}$.
We have $22 > (r/s) (p^f t m' - 1) = (r/s)(p^e t m - 1) \geq r (t/s) m (p^e - 1)$.
If $s = t$ then this implies $(p^e - 1)r \leq (p^e - 1)rm < 22$,
and if $(s,t) = (2,1)$ then this implies ($p \neq 2$ and) $2 \divides r$ and $(p^e - 1)r \leq (p^e - 1)rm < 44$.
We observe that this condition implies either $n \in \Smu'(p)$ or $(p,n) = (5,40), (3,48), (2,34)$.
It remains to show that each of the latter $3$ cases is impossible.

\medskip

If $(p,n) = (3,48)$, then $(m,s,t) = (1,2,1)$. 
Since $\mu_s$ does not act symplectically,
the generator of $\mu_{16}$ acts on $H^0(\tilde{X}, \Omega^2)$ by a primitive $16$-th root of unity.
If $\tilde{X}$ is of finite height, 
then by Lemma \ref{lem:H2 of K3}(\ref{lem:H2 of K3:transcendental}) we have $\rank T(\tilde{X}) \geq \phi(16) = 8$ and
we have $\rho(\tilde{X}) > 8 \cdot 2 = 16$ (from $8A_2$),
contradicting $\rho + \rank T(\tilde{X}) = 22$.
By the remark after Lemma \ref{lem:H2 of K3}, $\tilde{X}$ cannot be supersingular.

\medskip

If $(p,n) = (2,34)$, then $(m,s,t) = (1,1,1)$.
Let $e_i$ ($i \in \bZ/17\bZ$) be the exceptional curves above the $\mu_{17}$-orbit of $w$,
numbered in a way that a generator $g \in \mu_{17}$ acts by $g(e_i) = e_{i+1}$.
Let $L \subset \Pic(\tilde{X})$ be the sublattice generated by $e_i$'s, 
and $L' = \Pic(\tilde{X}) \cap \bQ L$ its primitive closure.

First suppose $\tilde{X}$ is of finite height.
Then by Corollary \ref{cor:quotient:height} the $\mu_2$-action is symplectic.
We may assume that $X$ is maximal, in which case the number of RDPs on $X$ (which are all of type $A_1$) is $8$ by Proposition \ref{prop:number of fixed points}, 
which is not compatible with the $\mu_{17}$-action.

Next suppose $\tilde{X}$ is supersingular.
By Lemma \ref{lem:H2 of K3}(\ref{lem:H2 of K3:Artin invariant}), the only possible Artin invariant is $\sigma_0 = 4$.
By Lemma \ref{lem:rank of lattice} below (applied to $L_1 = L'$, $L_2 = L'^{\perp}$, $M = \Pic(\tilde{X})$, $\bar{M} = \Pic(\tilde{X})^*$, $\bar{L}_i = L_i^*$),
we have 
\begin{align*}
\rank (\disc(L')) 
&\leq \dim_{\bF_2}(\disc(\Pic(\tilde{X}))) + \rank(\disc(L'^{\perp})) \\
&\leq \dim_{\bF_2}(\disc(\Pic(\tilde{X}))) + \rank(L'^{\perp}) \\
&= 2 \sigma_0 + (22 - \rank L') = 13.
\end{align*}
Since $\disc(L) \cong (\bZ/2\bZ)^{\oplus 17}$ has rank $17$, we obtain $L \subsetneq L'$.
Let $V \subset 2^{\bZ/17\bZ}$ be the set of subsets $S \subset \bZ/17\bZ$ such that $(1/2) \sum_{i \in S} e_i \in L'$.
Then $V$ is naturally a $g$-stable $\bF_2$-vector space, and is nonzero, 
and we can identify it with a nonzero $\bF_2[x]$-submodule $V$ of $\bF_2[x] / (x^{17} - 1)$.
Clearly $V = Q(x) \cdot \bF_2[x] / (x^{17} - 1)$ for some $Q(x) \in \bF_2[x]$ dividing $x^{17} - 1$.
Using the factorization 
$
 x^{17} - 1 = (x - 1) F_{17,1}(x) F_{17,2}(x)
$
in $\bF_2[x]$, where 
\begin{align*}
F_{17,1}(x) &= x^8 + x^7 + x^6 + x^4 + x^2 + x + 1 \quad \text{and} \\
F_{17,2}(x) &= x^8 + x^5 + x^4 + x^3 + 1
\end{align*}
are irreducible, 
it follows that $V$ contains at least one of
\begin{align*}
 (x^{17} - 1) / (x - 1)     &= x^{16} + \dots + 1, \\
 (x^{17} - 1) / F_{17,1}(x) &= x^9 + x^8 + x^6 + x^3 + x   + 1, \quad \text{or} \\
 (x^{17} - 1) / F_{17,2}(x) &= x^9 + x^6 + x^5 + x^4 + x^3 + 1.
\end{align*}
Hence there exists a set $S \in V$ with $\# S = 17$ or $\# S = 6$.
But then $((1/2) \sum_{i \in S} e_i)^2 = (1/2)^2 \cdot \# S \cdot (-2) \not\in 2\bZ$, contradiction.

\medskip

If $(p,n) = (5,40)$,
then $(m,s,t) = (1,2,1)$
and there are $4$ non-$\mu_5$-fixed $A_4$
on which $g$ acts transitively,
where $g$ is a fixed generator of $\mu_8 \subset \mu_{40}$,
and $g^4$ is non-symplectic. 
Moreover $\Fix(g^4)$ is $1$-dimensional, passing through the $4$ points of type $A_4$.

By the remark after Lemma \ref{lem:H2 of K3}, $\tilde{X}$ cannot be supersingular.
This implies that the $\mu_5$-action is symplectic (Corollary \ref{cor:quotient:height})
and hence the quotient $Y = X/\mu_5$ is an RDP K3 surface.

We will show in a subsequent paper \cite{Matsumoto:k3alphap}*{Proposition 2.15(4)} that the $\mu_p$-action induces 
a canonical nonzero element $v \in H^0(Y^{\sm}, \Omega^2) \otimes \Der(Y)$,
which we can write $v = \omega_Y \otimes D_Y$ (uniquely up to $k^*$), such that
$D_Y$ is $p$-closed and satisfies $Y^{D_Y} = \pthpower{X}$.
It is characterized by $D(f)^p  D_Y(h) \omega_Y = d(f^p) \wedge dh$ for local sections $f$ of $\cO_X$ and $h$ of $\cO_Y$.
Since the $\mu_5$-action in our case is $g$-invariant, it follows that $v = \omega_Y \otimes D_Y$ is $g$-invariant.

Fix a decomposition $v = \omega_Y \otimes D_Y$.
We have $D_Y^p = \phi D_Y$ for some meromorphic function $\phi$.
Since both $Y$ and the quotient $Y^{D_Y} \cong \pthpower{X}$ are RDP K3 surfaces, 
it follows from the Rudakov--Shafarevich formula \cite{Rudakov--Shafarevich:inseparable}*{Corollary 1 to Proposition 3} 
that $D_Y$ has only isolated fixed points,
and this implies that $\phi$ is holomorphic, hence constant.
We have $\phi \neq 0$, since if $\phi = 0$ then,
as we will prove in a subsequent paper \cite{Matsumoto:k3alphap}*{Lemma 3.6 or Theorem 4.6},
the $\alpha_p$-action corresponding to $D_Y$ must have quotient singularities different from $A_{p-1}$,
but $X$ has only RDPs of type $A_{p-1}$.
Since $\Fix(D_Y)$ is isolated and $Y^{D_Y} = Y^{g^*(D_Y)}$, 
we have $g^*(D_Y) = \lambda D_Y$ with $\lambda \in H^0(Y, \cO)^* = k^*$,
and since $D_Y^p = \phi D_Y$ with $\phi \in k^*$ we have $\lambda^{p-1} = 1$,
hence $(g^4)^{*}(D_Y) = D_Y$.

On the other hand, 
since $\Fix(g^4 \actson Y)$ is homeomorphic to $\Fix(g^4 \actson X)$ and hence is $1$-dimensional,
we have $(g^4)^{*}(\omega_Y) = - \omega_Y$.
This contradicts the $g$-invariance of $v = \omega_Y \otimes D_Y$.
\end{proof}

\begin{lem} \label{lem:rank of lattice}
Let $L_1 \subset \bar{L}_1$, $L_2 \subset \bar{L}_2$, and 
$L_1 \oplus L_2 \subset M \subset \bar{M} \subset \bar{L}_1 \oplus \bar{L}_2$ be sequences of abelian groups, where the bars have no specific meaning.
Assume that the projection $\bar{M} \to \bar{L}_1$ are surjective and that $M \cap (\bar{L}_1 \oplus 0) = L_1 \oplus 0$.
Then we have $\rank(\bar{L}_1 / L_1) \leq \rank(\bar{M}/M) + \rank(\bar{L}_2 / L_2)$,
where the rank of an abelian group is the minimum number of generators.
\end{lem}
\begin{proof}
We may assume $L_i = 0$.
The assumption then implies that $M \to \bar{L}_2$ is injective.
Since the rank behaves subadditively with respect to subgroups, quotients, and extensions, we obtain
$\rank(\bar{L}_1) \leq \rank(\bar{M}) \leq \rank(M) + \rank(\bar{M}/M) \leq \rank(\bar{L}_2) + \rank(\bar{M}/M)$.
\end{proof}

\section{Examples} \label{sec:examples}

For a projective variety with projective coordinates $(x_i)$, 
we use the notation $\wt(x_i) = (n_i)$ to mean that $\wt(x_j/x_i) = (n_j - n_i)$ on the affine piece $(x_i \neq 0)$ for each $i$.
Note that $\wt(x_i) = (n_i)$ is equivalent to $\wt(x_i) = (a + n_i)$.
We use a similar notation for subvarieties of $\bP(3,1,1,1)$.

\subsection{Symplectic actions} \label{subsec:examples:symplectic}

\begin{example}[Symplectic $\mu_4 \times \mu_4$-action] \label{ex:symplectic mu_4 mu_4} 
The quartic surface $X = (w^4 + x^4 + y^4 + z^4 + wxyz = 0)$ in characteristic $p = 2$ is an RDP K3 surface.
It has $6$ RDPs, all of type $A_3$,
at the points where two of $w,x,y,z$ are $0$ and the others are $1$.
This surface admits a symplectic action of $G = H_1 \times H_2$,
where $H_1 = \mu_4$ and $H_2 = \mu_4$ act by 
$\wt(w,x,y,z) = (0,0,1,-1)$
and $\wt(w,x,y,z) = (0,1,0,-1)$ respectively.

With respect to the action of the subgroup scheme $\mu_2 \subset H_1$, 
the $2$ RDPs at $(0,0,1,1)$ and $(1,1,0,0)$ are fixed
and the other $4$ RDPs are non-fixed and non-maximal.
The quotient surface by this $\mu_2$-action is $(W^2 + X^2 + Y^2 + Z^2 + AB = WX - A^2 = YZ - B^2 = 0)$ in $\bP^5$,
where $W = w^2, \dots, Z = z^2$ and $A = wx, B = yz$,
with 
$2$ RDPs of type $A_7$ at $(Y = Z = B = W+X = W+A = 0),(W = X = A = Y+Z = Y+B = 0)$
and $4$ of type $A_1$ at $(WX = A = YZ = B = W + X + Y + Z = 0)$.

The quotient morphism by the subgroup scheme $\mu_2 \times \mu_2$ (resp.\ the full group $G$) 
is the relative Frobenius morphism $X \to X^{(2)}$ (resp.\ $X \to X^{(4)}$).
\end{example}

\begin{example}[Symplectic $\mu_3 \times \mu_3$-action] \label{ex:symplectic mu_3 mu_3} 
The surface $X = (v^3 + w^3 + x^3 + y^3 + z^3 + vwx = v^2 - yz = 0) \subset \bP^4$
in characteristic $p = 3$
is an RDP K3 surface, and
has $2A_5$ at $(1,0,0,1,1),(0,1,-1,0,0)$ 
and $4A_2$ at $(0,1,0,-1,0),(0,1,0,0,-1),(0,0,1,-1,0),(0,0,1,0,-1)$.
This surface admits a symplectic action of $G = H_1 \times H_2$,
where $H_1 = \mu_3$ and $H_2 = \mu_3$ act by 
$\wt(v,w,x,y,z) = (0,1,-1,0,0)$ and 
$\wt(v,w,x,y,z) = (0,0,0,1,-1)$ respectively.
Let $D_1,D_2$ be the corresponding derivations.
The fixed points of $D_1$ (resp.\ $D_2$) is the first (resp. second) $A_5$ point.
The fixed points of $D_1 + D_2$ (resp.\ $D_1 - D_2$), which corresponds to the diagonal (resp.\ anti-diagonal) subgroup of $G$,
are the first and the fourth (resp.\ the second and the third) $A_2$ points.
The quotient morphism by $G$ is the relative Frobenius morphism $X \to X^{(3)}$.
\end{example}

\begin{example}[Symplectic $\mu_n$-action ($n = 5,6,7,8$)] 
For each $n = 5,6,7,8$,
let $F$ be a linear combination of the monomials listed in Table \ref{table:symplectic mu_n K3 equation}, in characteristic $p$,
and then $X = (F = 0) \subset \bP^3$ admits a $\mu_n$-action with the indicated weights.
If $F$ is a generic such polynomial, then $X$ is an RDP K3 surface and the $\mu_n$-action is symplectic.
The fixed locus is $X \cap \set{(1,0,0,0), \allowbreak (0,1,0,0), \allowbreak (0,0,1,0), \allowbreak (0,0,0,1)}$.

\begin{table} 
\caption{Examples of symplectic $\mu_n$-actions on RDP K3 surfaces} \label{table:symplectic mu_n K3 equation}
\begin{tabular}{llll}
\toprule
$n$ & $p$ & monomials & $\wt(w,x,y,z)$ \\
\midrule
$5$ & $5$   & $w^3 x, x^3 z, z^3 y, y^3 w, w^2 z^2, w x y z, x^2 y^2$ & $1,2,3,4$ \\
$6$ & $2,3$ & $w^4, w y^3, w x y z, x^3 z, z^4, w^2 z^2, x^2 y^2$     & $0,1,2,3$ \\
$7$ & $7$   & $w^4, x^3 z, z^3 y, y^3 x, w x y z$ & $0,1,2,4$ \\
$8$ & $2$   & $w^4, x^4, y^3 z, y z^3, w x y z$   & $0,2,1,5$ \\
\bottomrule
\end{tabular}
\end{table}

For example, for $n = 5,6,7,8$ respectively, the polynomials with coefficients $(1,1,1,1,0,0,0)$, $(1,1,1,1,1,0,0)$, $(1,1,1,1,1)$, $(1,1,1,1,1)$ 
satisfy the condition. 
\end{example}

\subsection{Non-symplectic actions}

\begin{example}[Non-symplectic $\mu_2$-action with Enriques quotient in characteristic $2$] \label{ex:2:enr:multi}
Following \cite{Bombieri--Mumford:III}*{Section 3}, let 
$L_1, L_2, L_3$ be three linear polynomials in $12$ variables
and let 
$X \subset \bP^5$ be the intersection of three quadrics $F_1, F_2, F_3$
defined by $F_h = L_h(x_k^2, x_i x_j, y_k^2, y_i x_j + x_i y_j + y_i y_j)_{1 \leq k \leq 3, 1 \leq i < j \leq 3} \in k[x_1,x_2,x_3,y_1,y_2,y_3]$.
Then for generic $L_h$, $X$ is an RDP K3 surface (with $12$ RDPs of type $A_1$), 
$\mu_2$ acts on ($\bP^5$ and) $X$ by 
$\wt(x_i, y_i + x_i) = (0,1)$ without any fixed point on $X$,
and the quotient $X/\mu_2$ is an Enriques surface.

\end{example}

\begin{example}[Non-symplectic $\mu_2$-action with rational quotient in characteristic $2$]
The quartic surface $w^2 (xy + z^2) + x^4 + y^4 + z^4 + yz(y^2 + z^2) = 0$
is an RDP K3 surface,
and the $\mu_2$-action with $\wt(w,x,y,z) = (0,1,1,1)$ is non-symplectic.
The fixed locus consists of the curve $(w = 0)$ and the RDP $(x = y = z = 0)$ of type $A_1$.
The image of this RDP in the quotient surface is a non-RDP singularity.
\end{example}

In the following example,
for two polynomials $A(t),B(t)$ with $\deg A \leq 8$ and $\deg B \leq 12$,
``the elliptic (or quasi-elliptic) surface defined by the equation $y^2 = x^3 + A(t) x + B(t)$''
is an abbreviation for the projective surface that is the union of four affine surfaces 
\begin{align*}
&\Spec k[x,y,t]        / (- y^2 + x^3 + A(t) x + B(t)), \\
&\Spec k[x',y',t^{-1}] / (- y'^2 + x'^3 + t^{-8} A(t) x' + t^{-12} B(t)), \\
&\Spec k[z,w,t]        / (- z + w^3 + A(t) wz^2 + B(t) z^3), \\
&\Spec k[z',w',t^{-1}] / (- z' + w'^3 + t^{-8} A(t) w'z'^2 + t^{-12} B(t) z'^3),
\end{align*}
glued by the relations 
$x' = t^{-4} x$, $y' = t^{-6} y$, 
$z = y^{-1}$, $w = x y^{-1}$, 
$z' = y'^{-1} = t^6 y^{-1}$, $w' = x' y'^{-1} = t^2 x y^{-1}$.
For generic $A$ and $B$ this is an RDP K3 surface.

\begin{example}[Non-symplectic $\mu_{n}$-actions] \label{ex:mu_n}

Table \ref{table:mu_n K3} proves the existence part of Theorem \ref{thm:mu_n} for $n$ divisible by $p$.
The first group consists of elliptic (or quasi-elliptic) RDP K3 surfaces,
the second of double sextics, and the third of quartics.
Only the non-$\mu_n$-fixed RDPs are listed,
except in the example for $(p,n) = (2,32)$ the $D_{20}^0$ point is fixed
and after blowing-up this point we find a non-fixed $D_{18}^0$ point.

The examples are characteristic $p$ reductions of 
the examples (of an automorphism of order $n$) in characteristic $0$ obtained 
respectively by
Brandhorst \cite{Brandhorst:non-symplectic}*{Theorem 5.9},
Keum \cite{Keum:orders}*{Example 3.2},
Kondo \cite{Kondo:trivially}*{Sections 3 and 7}, and
Oguiso \cite{Oguiso:globalindices}*{Proposition 2},
except that for the ones marked ``?'' we could not find a reference.
An asterisk means that we made a coordinate change $t \mapsto t^{-1}$.

\begin{table} 
\caption{Examples of non-symplectic $\mu_n$-actions on RDP K3 surfaces} \label{table:mu_n K3}
\begin{tabular}{llllrl}
\toprule
$p$ & $n$ & equation & $\wt(x,y,t)$ & RDPs & references \\
\midrule
$19$ & $38$ & $y^{2} = x^{3} + t^{7} x + t$           & ${2},{3},{6}$   & $  A_{18}$   & \cite{Kondo:trivially} \\
\midrule
$17$ & $34$ & $y^{2} = x^{3} + t^{7} x + t^{2}$       & ${4},{23},{6}$  & $  A_{16}$   & \cite{Kondo:trivially} \\
\midrule
$13$ & $26$ & $y^{2} = x^{3} + t^{5} x + t$           & ${2},{3},{6}$   & $  A_{12}$   & \cite{Kondo:trivially} \\
\midrule
$11$ & $44$ & $y^{2} = x^{3} + x + t^{11}$            & ${22},{11},{2}$ & $ 2A_{10}$   & \cite{Kondo:trivially} \\
\midrule
$7$  & $42$ & $y^{2} = x^{3} + t^{7} x + 1$           & ${14},{21},{4}$ & $ 3A_{6}$    & ? \\ 
$7$  & $28$ & $y^{2} = x^{3} + x + t^{7}$             & ${14},{7},{2}$  & $ 2A_{6}$    & \cite{Kondo:trivially} \\
\midrule
$5$  & $60$ & $y^{2} = x^{3} + t (t^{10} - 1)$        & ${2},{3},{6}$   & $ 2E_{8}^0$  & \cite{Keum:orders} \\ 
\midrule
$3$  & $60$ & $y^{2} = x^{3} + t (t^{10} - 1)$        & ${2},{3},{6}$   & $10A_{2}$    & \cite{Keum:orders} \\ 
$3$  & $42$ & $y^{2} = x^{3} + t (t^{7} - 1)$         & ${2},{3},{6}$   & $ 7A_{2}$    & \cite{Brandhorst:non-symplectic}{*} \\
$3$  & $36$ & $y^{2} = x^{3} - t (t^{6} - 1)$         & ${2},{3},{6}$   & $ 2E_{6}^0$  & \cite{Kondo:trivially}{*} \\
$3$  & $24$ & $y^{2} = x^{3} + t^{2} (t^{8} - 1)$     & ${2},{3},{3}$   & $ 8A_{2}$    & ? \\
\midrule
$2$  & $60$ & $y^{2} = x^{3} + t (t^{10} - 1)$        & ${2},{3},{6}$   & $ 5D_{4}^0$  & \cite{Keum:orders} \\ 
$2$  & $38$ & $y^{2} = x^{3} + t^{7} x + t$           & ${2},{3},{6}$   & $19A_{1}$    & \cite{Kondo:trivially} \\
$2$  & $36$ & $y^{2} = x^{3} - t (t^{6} - 1)$         & ${2},{3},{6}$   & $ 3D_{4}^0$  & \cite{Kondo:trivially}{*} \\
$2$  & $32$ & $y^{2} = x^{3} + t^{2} x + t^{11}$      & ${18},{11},{2}$ & $  D_{20}^0$ & \cite{Oguiso:globalindices} \\
$2$  & $26$ & $y^{2} = x^{3} + t^{5} x + t$           & ${2},{3},{6}$   & $13A_{1}$    & \cite{Kondo:trivially} \\
$2$  & $24$ & $y^{2} = x^{3} + t^{5} (t^{4} + 1)$     & ${2},{3},{6}$   & $  E_{8}^0$  & \cite{Brandhorst:non-symplectic} \\
\midrule
$p$ & $n$ & equation & $\wt(w,x,y,z)$ & RDPs & references \\
\midrule
$7$ & $42$ & $w^2 = x^5 y + y^5 z + z^5 x$     & $-1,0,-2,8$  & $3A_6$ & ? \\ 
\midrule
$3$ & $42$ & $w^2 = x^5 y + y^5 z + z^5 x$     & $-1,0,-2,8$  & $7A_2$ & ? \\ 
\midrule
$2$ & $42$ & $w^2 = x^5 y + y^5 z + z^5 x$     & $-1,0,-2,8$  & $21A_1$ & ? \\ 
$2$ & $22$ & $w^2 = x^5 y + y^5 z + x y^2 z^3$ & $-1,0,-2,8$  & $11A_1$ & ? \\ 
\midrule
$p$ & $n$ & equation & $\wt(w,x,y,z)$ & RDPs & references \\
\midrule
$2$ & $28$ & $w^4 + x^3 y + y^3 z + z^3 x = 0$ & $0,1,-3,9$   & $7A_3$ & ? \\ 
\bottomrule
\end{tabular}
\end{table}
\end{example}

\subsection{\texorpdfstring{$\mu_p$-actions}{mu\_p-actions} on abelian surfaces}

As shown in Theorem \ref{thm:quotient:abelian},
the nontrivial $\mu_p$-actions of abelian surfaces $A$, up to automorphisms of $\mu_p$, 
are precisely the translations by subgroup schemes of $A[p]$ isomorphic to $\mu_p$.

\begin{rem}
In the case of finite order automorphisms on abelian surfaces
there are examples with non-abelian quotients.
Kummer surfaces in characteristic $\neq 2$ are the minimal resolution of 
the RDP K3 quotient (with $16A_1$) by the symplectic involution $x \mapsto -x$ on abelian surfaces
(for characteristic $2$ see Remark \ref{rem:wild}). 
Also certain non-symplectic (or sometimes symplectic) actions give (quasi-)hyperelliptic quotients.
It seems that there are no $\mu_p$-analogue of these actions.
\end{rem}

\begin{rem}
If we consider \emph{rational} vector fields (i.e.\ possibly with poles) of multiplicative type
there are other kinds of examples.
See \cite{Katsura--Takeda:quotients}*{Example 6.2}
for a rational vector field of multiplicative type on an abelian surface (in characteristic $2$)
with a general type quotient.
\end{rem}

\subsection*{Acknowledgments}
I thank Hiroyuki Ito, Kazuhiro Ito, Shigeyuki Kondo, Christian Liedtke, Gebhard Martin, Hisanori Ohashi, Shun Ohkubo, and Tomoaki Shirato 
for helpful comments and discussions.
I also thank the anonymous referee for a number of corrections and suggestions.

\end{document}